\newif\ifsubmit
\submitfalse

\ifsubmit
    \RequirePackage{fix-cm}
    \documentclass[smallextended]{svjour3}       
    \smartqed  
\else
    \documentclass[12pt, hidelinks]{article}
\fi

\usepackage{hyperref}
\usepackage{setspace}
\usepackage{amsmath}
\usepackage{amssymb}
\usepackage{mathtools}
\usepackage[dvipsnames]{xcolor}
\usepackage{adjustbox}
\ifsubmit\else
    \usepackage{amsthm}
    \usepackage[utf8]{inputenc}
    \usepackage{fullpage}
\fi

\usepackage{caption}
\usepackage{subcaption}
\usepackage{graphicx}
\graphicspath{{figures/}}

\ifsubmit
\else
    \usepackage[style=alphabetic,backend=bibtex,maxbibnames=99]{biblatex}
    \bibliography{refs}
\fi

\usepackage[frozencache,cachedir=minted-cache]{minted}
\definecolor{LightGray}{gray}{0.25}
\setminted{bgcolor=LightGray, bgcolorpadding=0.5em, framerule=0.5pt, frame=single}

\usepackage{booktabs}
\newcommand{\ra}[1]{\renewcommand{\arraystretch}{#1}}

\title{GeNIOS: an (almost) second-order operator-splitting solver for large-scale convex optimization}

\ifsubmit
\author{
    Theo Diamandis \and 
    Zachary Frangella \and 
    Shipu Zhao \and
    Bartolomeo Stellato \and 
    Madeleine Udell
    }
\else
\author{
    Theo Diamandis  \\
    \texttt{tdiamand@mit.edu}
    \and Zachary Frangella  \\
    \texttt{zfran@stanford.edu} 
    \and Shipu Zhao \\
    \texttt{sz533@cornell.edu}
    \and Bartolomeo Stellato\\
    \texttt{bstellato@princeton.edu}
    \and Madeleine Udell\\
    \texttt{udell@stanford.edu}
    }
\fi

\ifsubmit
    \titlerunning{GeNIOS: a solver for large-scale convex optimization}
\fi

\ifsubmit
    \institute{T. Diamandis \at
                  Massachusetts Institute of Technology\\
                  Cambridge, MA, USA \\
                  \email{tdiamand@mit.edu}           
    }
    \date{Received: date / Accepted: date}
\else
    \date{October 2023}
\fi

\newcommand{\method}{\texttt{GeNIOS}}
\newcommand{\mlsolver}{\texttt{MlSolver}}
\newcommand{\qpsolver}{\texttt{QPSolver}}
\newcommand{\genericsolver}{\texttt{GenericSolver}}
\newcommand{\osqp}{\texttt{OSQP}}
\newcommand{\cosmo}{\texttt{COSMO}}

\newcommand{\cf}{{\it cf. }}
\newcommand{\eg}{{\it e.g., }}
\newcommand{\ie}{{\it i.e., }}

\newcommand{\ones}{\mathbf 1}
\newcommand{\reals}{\mathbf{R}}
\newcommand{\symm}{\mathbf{S}}
\newcommand{\eps}{\varepsilon}
\renewcommand{\epsilon}{\varepsilon}
\newcommand{\abs}[1]{\lvert{#1}\rvert}

\newcommand{\diag}[1]{\mathop{{\bf diag}\left({#1}\right)}}

\newcommand{\bmat}[1]{\begin{bmatrix}#1\end{bmatrix}}
\newcommand{\argmin}{\mathop{\rm argmin}}

\newcommand{\prox}{\mathbf{prox}}

\ifsubmit\else
    \newtheorem{theorem}{Theorem}
    \newtheorem{lemma}[theorem]{Lemma}
    
    \newtheorem{proposition}[theorem]{Proposition}
\fi

\newif\iftodos

\begin{document}

\maketitle

\begin{abstract}
    We introduce the GEneralized Newton Inexact Operator Splitting solver (\method{}) for large-scale convex optimization.
    \method{} speeds up ADMM by approximately solving approximate subproblems: 
    it uses a second-order approximation to the most challenging ADMM subproblem and solves it inexactly with a fast randomized solver.
    Despite these approximations, \method{} retains the convergence rate of classic ADMM and can detect primal and dual infeasibility from the algorithm iterates.
    At each iteration, the algorithm solves a positive-definite linear system that arises from a second-order approximation of the first subproblem and computes an approximate proximal operator.
    \method{} solves the linear system using an indirect solver with a randomized preconditioner, 
    making it particularly useful for large-scale problems with dense data.
    Our high-performance open-source implementation in Julia allows users to specify convex optimization problems directly 
    (with or without conic reformulation) and allows extensive customization.
    We illustrate \method{}'s performance on a variety of problem types.
    Notably, \method{} is up to ten times faster than existing solvers on large-scale, dense problems.
    \ifsubmit
        \keywords{Convex optimization \and Operator splitting \and Inexact ADMM \and Large-scale optimization \and Julia}
        \subclass{90C25 \and 90C06 \and 90C90 \and 65K05 \and 65K10}
    \fi
\end{abstract}

\section{Introduction}\label{sec:intro}
Data sets in modern optimization problems are large, motivating the search for optimization algorithms that scale well with the problem size.
The alternating direction method of multipliers (ADMM) is a particularly powerful algorithm for tackling these large, data-driven optimization problems.
Compared to interior point methods,
ADMM has a modest per-iteration cost and is easy to parallelize.
While the method is slow to produce a high-accuracy solution,
ADMM often finds a low-accuracy solution quickly,
which usually suffices for problems with real-world---and often noisy---data.
However, as problem sizes increase, even ADMM iterations can become unacceptably slow.
In particular, similar to~\cite{schubiger2020gpu}, we observe slowdowns for problems with 
data matrices with tens of millions of non-zeros or more---a scale easily exceeded by problems with dense datasets.
For example, a dense LASSO problem on, \eg a gene expression dataset
with $n=1,000$ samples and $p=10,000$ variables is already this large.

In this paper, we introduce the GEneralized Newton Inexact Operator Splitting solver (\method{}, pronounced ``genie-\={o}s''), a new inexact ADMM solver for convex optimization problems.
\method{} is designed to solve any convex optimization problem that can be represented as the sum of a smooth and non-smooth term, where the non-smooth term admits a tractable proximal operators.
This problem class includes standard-form LP, QP, SOCP, and SDP, and is particularly well-suited for regularized statistical learning problems such as the Lasso.
\method{} speeds up ADMM by approximating the smooth ADMM subproblem at each iteration, and then solving these approximate subproblems inexactly. 
Each iteration has two steps: 
1) solve a linear system that results from a second-order approximation to the smooth ADMM subproblem inexactly with a fast randomized solver, preconditioned conjugate gradient with the Nystr\"{o}m preconditioner~\cite{frangella2023randomized}, which offers a particular advantage for large-scale dense linear systems;
2) compute an approximate proximal operator.
We observe speedups of up to 50$\times$ compared to classic ADMM,
with the largest speedups on large-scale dense optimization problems.

\method{} maintains classic ADMM convergence guarantees by controlling subproblem errors, as outlined in recent theoretical work~\cite{frangella2023linear}.
Existing ADMM algorithms such as imsPADMM~\cite{chen2017efficient} have theory that supports approximate subproblems with inexact solves, but the available implementations do not exploit this level of generality. 
See the discussion in section~\ref{subsec:related_work} for further details.
Thus, to the best of our knowledge, \method{} is the first general-purpose ADMM solver that exploits both forms of inexactness: subproblem approximation and inexact solves.

\subsection{The optimization problem}
Consider the optimization problem
\begin{equation}\label{eq:problem-formulation}
    \begin{aligned}
        &\text{minimize}     && f(x) + g(z) \\
        &\text{subject to}   && Mx - z = c,
    \end{aligned}
\end{equation}
where $x \in \reals^n$ and $z \in \reals^m$ are decision variables, and $f:\reals^n \to \reals$, $g:\reals^m \to \reals \cup \{+\infty\}$, $M \in \reals^{m \times n}$, and $c \in \reals^m$ are the problem data.
Assume that the function $f$ is smooth and convex, and that the function $g$ is convex, proper, and lower-semicontinuous.
Thus, problem~\eqref{eq:problem-formulation} is a convex optimization problem, which we will refer to as a \emph{convex program}.
The flexibility of this formulation can provide an important speedup for statistical learning problems relative to conic reformulation.
The formulation can also be specialized to recover many special cases, including quadratic programs and conic programs.

\paragraph{Quadratic programs.}
A \emph{quadratic program} (QP) has the form
\[
    \begin{aligned}
        &\text{minimize}     && (1/2)x^TPx + q^Tx \\
        &\text{subject to}   && l \le Mx \le u,
    \end{aligned}
\]
where $x \in \reals^n$ is the decision variable and $P \in \symm_+^n$, $q \in \reals^n$, $M \in \reals^{m \times n}$, $l \in \reals^m$, and $u \in \reals^m$ are the problem data.
Linear equality constraints are encoded by setting $l_i = u_i$ for some $i \in \{1, \dots, m\}$.
A QP is a special case of~\eqref{eq:problem-formulation}
with 
\[
    f(x) = (1/2)x^TPx + q^Tx \qquad \text{and} \qquad  g(z) = I_{[l, u]}(z) =
    \begin{cases}
        0 & l \le z \le u \\
        +\infty & \text{otherwise,}
    \end{cases}
\]
and $c = 0$.
We call $I_S(z)$ the \emph{indicator function} of the set $S$.
Linear programs can be written in this form by setting $P = 0$.
QPs are ubiquitous in practice and encompass a wide variety of problems in disparate fields:
portfolio optimization in finance~\cite{markowitz1952portfolio,boyd2017multiperiodtradingportfolio}; model predictive control~\cite{garcia1989mpc,rawlings2000tutorialmpc}; denoising in signal processing~\cite{palomar2010convex}; model fitting in machine learning~\cite{sra2012optimization}; and the transport problem in operations research~\cite{kantorovich1948problem,dantzig1951application}, among others.

\paragraph{Conic programs.}
Let the function $g$ instead be the indicator function of a convex cone to recover the \emph{conic program}:
\begin{equation}
\label{eq:conic-program-formulation}
    \begin{aligned}
        &\text{minimize}     && (1/2)x^TPx + q^Tx \\
        &\text{subject to}   && Mx - z = c \\
                            &&& z \in \cal K,
    \end{aligned}
\end{equation}
where $\mathcal{K}$ is a non-empty, closed, convex cone.
Again, we recognize~\eqref{eq:conic-program-formulation} as a special case of~\eqref{eq:problem-formulation} with
\[
    f(x) = (1/2)x^TPx + q^Tx \qquad \text{and} \qquad  g(z) = I_{\mathcal{K}}(z).
\]
Any convex optimization problem can be written in this form~\cite{nesterov1992conic}, including linear programs, quadratic programs, second-order cone programs (SOCPs), and semidefinite programs (SDPs),
and many modeling languages can translate convex optimization problems to conic form~\cite{grant2014cvx,diamond2016cvxpy,agrawal2018cvxpy,convexjl}. 
As a result, this formulation is used by many of the popular convex optimization solvers (\eg \texttt{SCS}~\cite{scs1,scs2}, \cosmo{}~\cite{cosmo}, \texttt{Hypatia}~\cite{hypatia}, and \texttt{Mosek}~\cite{mosek}).
SOCPs appear in robust optimization~\cite{ben1998robust,ben1999robust,ben2009robust}, model predictive control~\cite{blackmore2010rocket}, and many engineering design problems~\cite{lobo1998applications}.
Applications of SDPs include convex relaxations of binary optimization problems~\cite{lovasz1991cones}, experiment design~\cite[\S7.5]{cvxbook}, circuit design~\cite{gamal1997circuit,vandenberghe1998circuit}, and sum-of-squares programs~\cite{blekherman2012semidefinite,parrilo2003semidefinite,laurent2009sums}.

\paragraph{Machine learning problems.}
Consider machine learning problems of the form
\[
\begin{aligned}
&\text{minimize} && \sum_{i=1}^N \ell(a_i^Tx - b_i) + (1/2)\lambda_2 \|x\|_2^2 + \lambda_1 \|x\|_1,
\end{aligned}
\]
where $\ell:\reals \to \reals_+$ is a convex per-sample loss function, $\{(a_i, b_i)\}_{i=1}^N \subseteq \reals^n \times \reals$ are problem data, and $\lambda_1$, $\lambda_2 \ge 0$ are regularization parameters.
In the framework of~\eqref{eq:problem-formulation}, take
\[
    f(x) = \sum_{i=1}^N \ell(a_i^Tx - b_i) + (1/2)\lambda_2 \|x\|_2^2 
    \qquad \text{and} \qquad  
    g(z) = \lambda_1 \|z\|_1,
\]
and $M = I$, $c = 0$.
The conic reformulation of this problem typically has many additional variables.
As an example, consider the $\ell_1$-regularized logistic regression problem:
\[
\begin{aligned}
& \text{minimize} && \sum_{i=1}^N \log\left(1 + \exp(b_ia_i^Tx)\right) + \lambda_1 \|z\|_1 \\
& \text{subject to} && x - z = 0.
\end{aligned}
\]
The reformulation of logistic regression as a conic program (specifically, an exponential cone program) has at least $2n + 3N$ variables (see appendix~\ref{app:logistic-conic}).
In general, this reformulation slows solve time since the per-iteration time scales superlinearly in the problem size. 
For example, if the dominant operation is an $O(n^2)$ matrix-vector product, increasing the number of variables by a factor of $5$ can increase solve time by a factor of $25$.
Handling machine learning problems directly via \method{} provides a significant performance boost over general-purpose conic solvers.

\subsection{Solution methods}
\label{subsec:related_work}
A variety of algorithms can be used to solve~\eqref{eq:problem-formulation} and its specific subclasses.
We briefly review some of the most popular and relevant methods.

\paragraph{Interior point methods and solvers.}
Primal-dual interior point methods (IPMs), with roots largely in work by Karmarkar, Nesterov, Nemirovisky, and Mehrotra~\cite{karmarkar1984new,nesterov1994interior,mehrotra1992implementation}, have historically been the method of choice for convex optimization.
Commercial convex optimization solvers, including \texttt{Gurobi}~\cite{gurobi}, \texttt{Mosek}~\cite{mosek}, and \texttt{CVXGEN}~\cite{mattingley2012cvxgen} use IPMs.
Several open-source IPM solvers exist as well, including \texttt{Clarabel}~\cite{chen2023efficient}, \texttt{ECOS}~\cite{domahidi2013ecos}, \texttt{CVXOPT}~\cite{andersen2013cvxopt}, and \texttt{Hypatia}~\cite{hypatia}.
Most similar to the spirit of our work, \texttt{Hypatia} goes beyond the set of \emph{standard cones} (usually the positive orthant, second order cone, semidefinite cone, exponential cone, and power cone) used by other solvers and implements many \emph{exotic cones}, which allow the user to more naturally specify the problem and avoid problem size bloat under reformulation. 
Still, this solver uses a conic formulation~\eqref{eq:conic-program-formulation}, just with an expanded set of cones $\mathcal{K}$.
IPMs quickly produce very accurate solutions for small and medium-sized problems.
However, they cannot be easily warm-started, and they do not scale well for very-large problems due to the formation and factorization of a large matrix at each iteration.
A natural alternative to address the scaling issue is to consider indirect methods for solving the linear system, as they only require matrix-vector products.
Unfortunately, the condition number of the associated linear system becomes large as the solver nears a solution, rendering indirect methods for solving the system ineffective~\cite{gondzio2012interior}.

\paragraph{First-order methods.}
The poor scaling of interior point methods and rise of large-scale, data-driven optimization prompted the resurgence of first-order methods.
These methods can be implemented so that the iteration time is dominated by matrix-vector product computations, which can be accelerated on modern computing architectures.
Although first-order methods are slow to converge to high accuracy solutions, they usually produce moderately accurate solutions quickly.
Two of the most common methods in solvers today are ADMM~\cite{glowinskiadmm1,gabayadmm2,douglas1956numerical,lions1979splitting,gabay1983applications} (see~\cite{boyd2011distributed} for a modern survey) and primal-dual hybrid gradient (PDHG), also called Chambolle-Pock~\cite{chambolle2011first}.
Both of these methods (and many others) are special cases of Rockafellar's proximal point algorithm~\cite{rockafellar1976monotone,eckstein1992douglas} (see~\cite[\S 3]{ryu2022large} and~\cite{zhao2021automatic} for related derivations).

\paragraph{First-order solvers.}
In the last decade, many open-source solvers have implemented variants of ADMM or PDHG.
The conic solver \texttt{SCS}~\cite{scs1,scs2} applies ADMM directly to the homogenous self-dual embedding~\cite{ye1994nl,xu1996simplified} of a conic program and was the first ADMM-based solver that could handle infeasible or unbounded problems.
Later, the QP solver \osqp{}~\cite{osqp} and conic solver \cosmo{}~\cite{cosmo} used the results of~\cite{banjac2019infeasibility} to apply ADMM more directly to~\eqref{eq:conic-program-formulation} while still being able to detect infeasibility and unboundedness by looking at differences of the iterates.
PDHG has been applied to solve large-scale LPs in \texttt{PDLP}~\cite{pdlp1,pdlp2} and SDPs in \texttt{ProxSDP}~\cite{proxsdp2022}.
Notably, PDHG does not require a linear system solve at each iteration for conic programs of the form~\eqref{eq:conic-program-formulation}.
Both \texttt{PDLP} and \texttt{ProxSDP} use the method of~\cite{applegate2021infeasibility} to detect infeasibility, but to the best of our knowledge, it's unclear to what extent these theoretical results extend to the case of conic programs with inexact subproblem solves.
The ADMM-based solvers require the solution of a linear system at each iteration, and
although this system can be solved by indirect methods, ill-conditioning of the problem data---a common phenomenon in real-world data matrices~\cite{udell2019big}---can slow convergence of these methods.
In addition, all of these solvers require problems to be passed in a conic form resembling~\eqref{eq:conic-program-formulation}, which typically leads to a substantial increase in problem size, especially in machine learning problems.

\paragraph{Beyond conic programs.}
All of the solvers discussed solve conic problems of a form similar to~\eqref{eq:conic-program-formulation}, or solve a specialized version of this form more tailored to a specific problem class (\eg LPs or QPs).
While this form allows solvers to handle essentially all convex optimization problems provided by a user (usually with the help of a modeling framework such as \texttt{JuMP}~\cite{jump} or \texttt{Convex.jl}~\cite{convexjl}), transforming these problems to conic form can make the problem more difficult to solve by increasing the size or obscuring the structure of the constraint matrix $M$ or objective matrix $P$.
Another line of work avoids transforming problems into conic form by building solvers directly on a library of proximal operators, including
\texttt{Epsilon}~\cite{wytock2015prox}, which is unmaintained, \texttt{POGS}~\cite{fougner2018parameter}, which requires a separable objective, and \texttt{ProxImaL}~\cite{heide2016proximal}, which focuses on image optimization problems.
In general, defining $M$ directly as a linear operator (rather than as a concrete matrix) speeds up computation but can be difficult to achieve for solvers that require a conic form.
\texttt{SCS} does support matrix-free linear operators but still requires a conic form problem~\cite{diamond2016matrix}.

\paragraph{Inexact ADMM.}
It is well known that ADMM applied to~\eqref{eq:problem-formulation} converges to an optimum at an $O(1/k)$ rate if both subproblems are solved exactly~\cite{he20121,monteiro2013iteration}.
Classic work~\cite{eckstein1992douglas} established that ADMM also converges when subproblems are solved inexactly, provided the errors of subproblems are summable.
A more contemporary line of work~\cite{ouyang2015accelerated,deng2016global,chen2017efficient} considers replacing the $x$- and $z$-subproblems of ADMM with an approximate subproblem that is easier to solve.
The $x$-subproblem is typically approximated as a quadratic optimization, which can be solved with any linear system solver.
These papers establish convergence of the resulting methods provided the approximate subproblems are solved exactly. 
In recent work, the authors show that algorithms that use inexact solves of approximate ADMM subproblems preserve ADMM's $O(1/k)$ convergence rate \cite{frangella2023linear},
and that accelerating ADMM with randomized Nystr{\"o}m preconditioning and approximate $x$-subproblem solves yields significant speed-ups over standard solvers for a variety of machine learning problems \cite{pmlr-v162-zhao22a},
providing a strong motivation for the more general solver presented here.
Second-order subproblem approximations are also important in solvers based on the augmented Lagrangian method, 
such as the \texttt{ALADIN} solver~\cite{houska2016augmented}.

The imsPADMM algorithm from \cite{chen2017efficient} shares many similarities with the algorithmic framework in this paper:
both use function linearization, inexact subproblem solving, and non-isotropic quadratic penalty terms. 
However, they differ significantly in their algorithmic goals, problem classes, and implementation details. 
GeNIOS is a general convex first-order solver for problems that fit in memory but are expensive to solve exactly, while imsPADMM focuses on high-dimensional linearly constrained convex composite quadratic conic programs. 
GeNIOS uses a variable metric based on the Hessian, updates all coordinates simultaneously, and employs randomized Nyström preconditioning for solving the $x$-subproblem. 
In contrast, imsPADMM uses a fixed metric, Gauss-Seidel updates, and uses PCG with a preconditioner constructed via the Lanczos algorithm to accelerate solution of the subproblems. 
For large-scale problems that fit in memory, 
GeNIOS is likely to be more efficient as it can update all coordinates at once, better leveraging the massive parallelism of modern computing hardware.

While the code for imsPADMM is not available, the algorithms QSDPNAL~\cite{li2018qsdpnal} and QPPAL~\cite{liang2022qppal} build on~\cite{chen2017efficient} and use many similar ideas. 
However, the associated codes target only a specific problem class (QP or SDP, respectively) and are not intended to be used outside the research setting.
In contrast, \method{} offers a variety of easy-to-use interfaces, including a QP interface, an interface that accepts problems specified via the JuMP modeling language \cite{jump}, and a general interface that accepts a gradient oracle and proximal operator and can handle problems like logistic regression without conic reformulation.

\begin{table}[H]
    \centering
    \ra{1.3}
    \small
\ifsubmit\begin{adjustbox}{max width=\textwidth}\fi
\begin{tabular}{@{}lrrrrr@{}}
\toprule
&\method{} & \cosmo{} & \osqp{} & \texttt{SCS} & \texttt{ProxSDP} \\
\midrule
method                          & ADMM  & ADMM & ADMM & ADMM & PDHG \\
interface                       & any convex  & conic & QP & conic & SOCP, SDP \\
approximate subproblems         & yes   & no    & no    & no    & no\\
linear system solver            & indirect  & (in)direct   & (in)direct   & (in)direct & n/a \\
inexact solves     & yes & yes & GPU only & yes & n/a \\
preconditioning & Nystr\"{o}m & diagonal & diagonal & diagonal & n/a \\
inexact projection              & yes  & no &  n/a  & no & yes \\
\bottomrule
\end{tabular}
\ifsubmit\end{adjustbox}\fi
\caption{
    \method{} offers a more flexible interface and exploits approximations and inexactness more than similar existing solvers. 
    \texttt{ProxSDP} uses an inexact projection onto the positive semidefinite cone. 
    A conic interface indicates that the solver solves~\eqref{eq:conic-program-formulation} for the full set of standard cones: the positive orthant, second-order cone, positive semidefinite cone, exponential cone, and power cone.
}
    \label{tab:solvers-compare}
\end{table}

\subsection{Contributions}
This paper showcases, through the \method{} solver, how several recent theoretical ideas can be combined to speed up ADMM while preserving convergence guarantees.
Our contributions can be summarized as follows:
\begin{itemize}
    \item We develop the \texttt{GeNIOS.jl} solver: 
    the first open-source solver
    that
    \begin{itemize}
        \item accelerates subproblem solves by inexactly solving approximate ADMM subproblems using Nystr\"{o}m preconditioning; and
        \item allows direct specification of ADMM problems via function, gradient and Hessian oracles for~$f$ and function and proximal operator oracles for~$g$.
    \end{itemize}
    \item We show that \method{} can dectect infeasibility in conic programs despite inexact solves using~\cite{rontsis2022efficient}, and we show that \method{} converges at the standard $O(1/k)$ rate (and faster when the problem is strongly convex) based on results from the authors' prior work~\cite{frangella2023linear}. 
    
    \item We show how \method{} allows the user to exploit problem structure by specifying ADMM problems directly (not in conic form) and by using the Julia programming language's multiple dispatch to implement efficient linear operators, including ones that can run on the GPU.
    \item We showcase \method{}'s up to 50$\times$ speedup over classic ADMM on a variety of optimization problems with real-world and simulated data.
\end{itemize}
The code is available online at 
\begin{center}
    \texttt{https://github.com/tjdiamandis/GeNIOS.jl}
\end{center}
with documentation that includes several examples.
Table~\ref{tab:solvers-compare} compares \method{} to the most similar existing solvers.

\paragraph{Roadmap.}
We overview the \method{} algorithm in~\S\ref{sec:method}, including the overall method, convergence guarantees, infeasibility detection, and randomized preconditioning for the linear system solve at each iteration.
In~\S\ref{sec:apps}, we discuss the solver interface for general convex optimization problems and the specialized interfaces for quadratic programming and machine learning problems.
In~\S\ref{sec:experiments}, we numerically demonstrate the performance improvements gained by leveraging inexactness, randomized preconditioning, and the more natural problem formulations allowed by \method{}.
Finally, we point to some directions of future work in~\S\ref{sec:conclusion}.


\section{Method}\label{sec:method}
Our method \method{} uses inexact ADMM and techniques from randomized numerical linear algebra to speed up solve times on large-scale optimization problems.
Recall that \method{} solves convex problems in the form~\eqref{eq:problem-formulation},
\begin{equation}\tag{\ref{eq:problem-formulation}}
\begin{aligned}
    &\text{minimize}     && f(x) + g(z) \\
    &\text{subject to}   && Mx - z = c,
\end{aligned}
\end{equation}
where $x \in \reals^n$ and $z \in \reals^m$ are decision variables, and $f: \reals^n \to \reals$, $g: \reals^m \to \reals \cup \{+\infty\}$, $M \in \reals^{m \times n}$ and $c \in \reals^m$ are the problem data.
The function $f$ is smooth and convex, and the function $g$ is convex, proper, and lower-semicontinuous.
\method{} requires the ability to evaluate $f$ and $g$, the gradient $\nabla f$, a Hessian-vector product (HVP) $v \mapsto \nabla^2 f(x) v$, and the proximal operator of $g$,
\[
\prox_{g/\rho}(v) = \argmin_{\tilde z} g(\tilde z) + (\rho/2)\|\tilde z - v\|_2^2.
\]
The HVP and proximal operator may be approximate, subject to a condition on the incurred errors.
This flexibility allows \method{} to easily model a variety of problems of interest, as gradients and HVPs can be easily specified---including with automatic differentiation---and proximal operators can be efficiently computed for many functions (see, \eg\cite{combettes2011proximal,parikh2014proximal,wytock2015prox} and references therein).
As a result, \method{} not only handles conic programs but also specializes to problems with bespoke objective functions,
such as robust regression problems in machine learning.

\method{} replaces the exact subproblem solutions of classic ADMM with inexact ones.
The standard (scaled) ADMM algorithm applied to~\eqref{eq:problem-formulation} consists of the iterations
\begin{align}
    x^{k+1} &= \argmin_{x} \left(f(x) + (\rho / 2)\|Mx - z^k - c + u^k\|_2^2 \right) \label{eq:x-update-admm}\\
    z^{k+1} &= \argmin_{z} \left(g(z) + (\rho / 2)\|Mx^{k+1} - z - c + u^k\|_2^2 \right) \label{eq:z-update-admm} \\
    u^{k+1} &= u^k + Mx^{k+1} - z^{k+1} - c. \label{eq:u-update-admm}
\end{align}
\method{} replaces the function $f$ in the $x$-subproblem with a second-order approximation,
\[
f(x) \approx f(x^k) + \nabla f(x^k)^T (x - x^k) + (1 / 2)\|x - x^k\|^2_{\nabla^2 f(x^k) + \sigma I}.
\]
\method{} does not require elementwise access to the matrix $\nabla^2 f(x^k) + \sigma I$, but only to matrix-vector products.
When these are expensive to compute (\eg for very large-scale problems) \method{} may approximate the regularized Hessian,
for example, by updating a Hessian estimate every few iterations.
Any value $\sigma >0$ guarantees convergence \cite[Theorem 1]{frangella2023linear}.
 
After approximating $f$, the $x$-subproblem becomes
\begin{equation}
\ifsubmit
    \begin{aligned}
    x^{k+1} = \argmin_{x} \Big(&f(x^k) + \nabla f(x^k)^T (x - x^k) + (1/ 2)\|x - x^k\|^2_{\nabla^2 f(x^k) + \sigma I}\\
    &+ (\rho / 2)\|Mx - z^k - c + u^k\|_2^2 \Big).    
    \end{aligned}
\else
    x^{k+1} = \argmin_{x} \Big(f(x^k) + \nabla f(x)^T (x - x^k) + (1/ 2)\|x - x^k\|^2_{\nabla^2 f(x^k) + \sigma I} + (\rho / 2)\|Mx - z^k - c + u^k\|_2^2 \Big).
\fi
\label{eq:x-update-nysadmm}
\end{equation}
Minimizing this unconstrained convex quadratic is equivalent to solving a linear system.
GeNIOS requires only an inexact solution $\tilde x^{k+1}$ to \eqref{eq:x-update-nysadmm}  that satisfies $\|\tilde x^{k+1} - x^{k+1}\| \leq \epsilon_x^k$.

The $z$-subproblem~\eqref{eq:z-update-admm} is unchanged,
but the subroblem solver may return any $\epsilon_z^k$-suboptimal solution $z^{k+1}$: 
denoting the true solution as $z^{k+1, \star} = \mathbf{prox}_{g/\rho}(Mx^{k+1} - c + u^k)$, the approximate solution $z^{k+1}$ must satisfy
\[
\begin{aligned}
    &g(z^{k+1}) + (\rho / 2)\|Mx^{k+1} - z^{k+1} - c + u^k\|_2^2 \\
  - &g(z^{k+1, \star}) + (\rho / 2)\|Mx^{k+1} - z^{k+1, \star} - c + u^k\|_2^2 < \eps_z^k.
\end{aligned}
\]
Convergence is guaranteed, provided that the subproblem errors $\epsilon_x^k$ and $\sqrt{\epsilon^k_z}$ are summable \cite{frangella2023linear}.

\subsection{Solving the linear system}
\label{sec:linsys}
The $x$-subproblem update after approximation~\eqref{eq:x-update-nysadmm} is an unconstrained convex QP.
Its solution solves the linear system
\begin{equation}\label{eq:x-update-lin-sys}
\ifsubmit
\begin{aligned}
    \Big(\nabla^2 f(x^k) &+ \rho M^TM + \sigma I\Big) x =\\ 
    &(\nabla^2 f(x^k) + \sigma I) x^k - \nabla f(x^k) + \rho M^T (z^k + c - u^k).
\end{aligned}
\else
    \left(\nabla^2 f(x^k) + \rho M^TM + \sigma I\right) x = (\nabla^2 f(x^k) + \sigma I) x^k - \nabla f(x^k) + \rho M^T (z^k + c - u^k).
\fi
\end{equation}
The term $\sigma I$ ensures that this system is positive definite even when $\nabla^2 f(x^k) + \rho M^TM$ is rank deficient.
\method{} targets large-scale problems by using a preconditioned conjugate gradient method (CG)~\cite{hestenes1952methods} to solve \eqref{eq:x-update-lin-sys}.
\method{} uses the CG implementation in \texttt{Krylov.jl}~\cite{montoison-orban-krylov-2020}.
This linear system is often ill-conditioned, as real-world data is generally approximately low-rank~\cite{udell2019big}.
The Nystr\"{o}m preconditioner of \cite{frangella2023randomized} improves the convergence rate in this setting.
The resulting algorithm only requires matrix-vector products with the problem data. Hence \method{} enjoys the best of both worlds: it reduces the number of outer ADMM iterations through a very accurate subproblem approximation (using the problem Hessian) but allows for fast iterations as it accesses the Hessian only through Hessian-vector products.

\paragraph{Rank deficiency.}
The left-hand-side matrix of~\eqref{eq:x-update-lin-sys} is usually full rank.
For example, $\nabla^2 f(x)$ may be the a sum of a positive semidefinite matrix and a positive diagonal matrix, 
or $M$ may include an identity block. 
In this case, the user may set $\sigma = 0$.
If the rank is not obvious, the user can estimate the minimum eigenvalue of $M^TM$ via power-iteration or the randomized Lanczos method~\cite[Alg. 5]{martinsson2020randomized}.
For a conic program (or QP), the user can also estimate the minimum eigenvalue of $\nabla^2f(x) = P$, which is constant across iterations.
These algorithms are relatively cheap to compute and can be done as part of the problem setup.
In general, however, \method{} cannot assume that the left-hand-side matrix is full-rank and must use $\sigma > 0$.

\paragraph{Inexact solves.}
The the $x$-subproblem errors $\epsilon_x^k$ must be summable for \method{} to converge.
\method{} ensures this condition by using the following relative tolerance for \eqref{eq:x-update-lin-sys}:
\[
\frac{\min\left(\sqrt{\|r_p^k\|\|r_d^k\|},\; 1.0 \right)}{k^\gamma}.
\]
The above stopping tolerance combines strategies of existing conic solvers: \texttt{SCS}~\cite{scs1} and \cosmo{}~\cite{cosmo}, which use $1/k^{\gamma}$ with $\gamma = 1.5$ by default; and \osqp{}'s GPU implementation~\cite{schubiger2020gpu}, which uses an error proportional to the geometric mean of the residuals $\sqrt{\|r_p^k\|\|r_d^k\|}$.
Intuitively, \method{} uses a looser tolerance when the residuals are large but then tightens this tolerance as the residuals decrease to hasten convergence.

\subsection{Randomized preconditioning}
\label{sec:preconditioner}
\method{} uses techniques from randomized numerical linear algebra to build a Nystr\"{o}m preconditioner~\cite{frangella2023randomized} for the $x$-subproblem linear system~\eqref{eq:x-update-lin-sys}.
The left-hand-side matrix is often ill-conditioned, resulting in slow convergence of CG.
To build a preconditioner, we want to quickly find an approximate inverse of the dominant eigenspace of this matrix, which we do by creating a random sketch that is easy to invert.
When the matrix is not-too-sparse and low rank, which is often true for real-world data (see~\cite{udell2019big} and references therein), this preconditioner provides a significant speedup over standard CG.

\paragraph{Preliminaries.}
Let $H$ denote the matrix in the $x$-subproblem linear system solve~\eqref{eq:x-update-lin-sys}:
\[
    H = \nabla^2 f(x^k)+\rho M^TM+\sigma I.
\]
In many practical settings, the optimization problem~\eqref{eq:problem-formulation} can be written such that either the Hessian or $M^TM$ is a diagonal matrix.
\method{} uses a preconditioner for matrices of the form
\begin{equation}
\label{eq:lrdiag}
A+D \qquad \text{where} \qquad A = \nabla^2f(x) \quad \text{or} \quad A = \rho M^T M
\end{equation}
and $D$ is diagonal.
Without loss of generality, we can assume that $D = \nu I$, for some $\nu > 0$ because the linear system
\[
(A+D)w = b
\]
is equivalent to the linear system
\[
\left(D^{-1/2}AD^{-1/2}+I\right)\tilde w = D^{-1/2}b,
\]
where $w = D^{-1/2}\tilde w$.
For the remainder of this section, let $D = \nu I$ and assume the linear system in~\eqref{eq:x-update-lin-sys} has the form
\begin{equation}
\label{eq:reduc-lin-sys}
(A+\nu I) x = b.
\end{equation}

\paragraph{The preconditioner.}
To precondition the linear system with left-hand-side matrix $A + \nu I$, \method{} first constructs a randomized Nystr{\"o}m approximation~\cite[Alg. 16]{martinsson2020randomized} to $A$ using test matrix $\Omega \in \reals^{n\times r}$:
\[
\hat A = (A\Omega)(\Omega^TA\Omega)^{\dagger}(A\Omega)^T = U\hat \Lambda U^{T},
\]
where $U \in \reals^{n\times r}$ has orthonormal columns, and $\hat \Lambda \in \reals^{r\times r}$ is diagonal.
\method{} uses a standard normal\footnote{
    This choice does not preserve sparsity, and other options like subsampled trigonometric transforms, may provide better performance for sparse matrices. See~\cite[\S9]{martinsson2020randomized} for discussion and references.
} test matrix $\Omega$,
as extensive theoretical and numerical work has found that this choice yields an excellent low-rank approximation~\cite{alaoui2015fast,musco2017recursive,tropp2017fixed,tropp2019streaming}.
Given the low-rank assumption holds, \method{} can take $r \ll n$ without losing much accuracy in the approximation to $A$.
Computing the sketch then approximately scales as the cost of a matrix-vector-product with $A$.
The randomized Nystr{\"o}m preconditioner is
\begin{equation}\label{eq:preconditioner}
    L^{-1} = (\hat \Lambda_{r,r}+\nu I)U\left(\hat \Lambda +\nu I\right)^{-1}U^{T}+I-UU^{T}.
\end{equation}
This preconditioner (approximately) inverts the dominant eigenspace of $A$, while leaving the orthogonal complement unaffected.
Additionally, $L^{-1}$ never needs to be formed explicitly; \method{} stores it in a factored form with a light storage footprint.
In this form, \method{} cheaply applies the preconditioner in $O(nr)$ time and updates the parameter $\nu$ without recomputing the Nystr{\"o}m approximation.

Low-rank preconditioners have been employed in other ADMM solvers, such as imsPADMM \cite{chen2017efficient}. 
However, these preconditioners are typically constructed using traditional numerical linear algebra algorithms like the Lanczos algorithm. 
Modern techniques based on randomized linear algebra \cite{martinsson2020randomized},
such as those used by \method{}, can often construct low-rank approximations much faster than the Lanczos algorithm,
although they require about the same number of floating point operations \cite{halko2011finding}.
Parallel computing drives this difference: while the matrix-vector products computed by the Lanczos algorithm must be executed sequentially,
the matrix-vector products that form the sketch $A\Omega$ can be computed in parallel \cite{frangellarla2025}.

\paragraph{Adaptive sketch size selection.}
A good preconditioner decreases the (preconditioned) condition number of the linear system~\eqref{eq:reduc-lin-sys} to a small constant.
In this case, the standard analysis of CG~\cite[\S38]{trefethen1997numerical} guarantees convergence of the PCG iterates to an $\eps$-ball of the solution
within $O\left(\log(1/\eps)\right)$ iterations.

By default \method{} uses a constant sketch size, which works well across diverse experiments.
It also implements an adaptive technique to update the sketch size
that uses the following bound on the condition number $\kappa$ of the left-hand-side matrix~\cite[Prop. 5.3]{frangella2023randomized}:
\[
    \kappa(L^{-1/2}(A+\nu I)L^{-1/2})\leq 1+\frac{\hat \Lambda_{rr} + \|E\|}{\nu} \qquad \text{where } E = A - \hat A.
\]
Starting from some small initial sketch size, \method{} can increase the size of the sketch until $\|E\|$ and $\hat \Lambda_{rr}$ are suitably small, or until the sketch size is unacceptably large.
\method{} computes $\hat \Lambda_{rr}$ as part of the Nystr\"{o}m sketch, and it can reliably estimate the error $\|E\|$ using a few iterations of the randomized power method~\cite[Alg. 4]{martinsson2020randomized}.
See~\cite[\S 5.4]{frangella2023randomized} for additional discussion.

\subsection{Convergence}
\label{sec:convergence}
\method{} is a special case of the \texttt{GeNI-ADMM} framework in ~\cite{frangella2023linear}, and so its convergence can be derived as a special case of the theory developed there.
Provided the subproblem errors $\{\epsilon_x^k\}$ and $\{\sqrt{\varepsilon_z^k}\}$ are summable, 
Theorem 1 in~\cite{frangella2023linear} guarantees the averaged iterates
\[
    \bar x^{k+1} = \frac{1}{k}\sum^{k+1}_{t=2} x^t \qquad \text{ and } \qquad \bar z^{k+1} = \frac{1}{k}\sum^{k+1}_{t=2} z^t
\]
produce objective value error and primal residual error that converge at rate $O(1/k)$, \ie
\[
f(\bar x^{k+1})+g(\bar z^{k+1})-p^{\star} = O\left(1/k\right),
\]
and
\[
\|M\bar x^{k+1}+\bar z^{k+1}-c\| = O\left(1/k\right).
\]
As standard ADMM converges at an $O(1/k)$ rate \cite{he20121,beck2017first}, the preceding theory suggests that \method{} will require approximately the same number of iterations.
However, each iteration of \method{} will be faster due to the approximation of the $x$-subproblem and inexact subproblem solves.
Thus, overall we expect \method{} to converge faster than standard ADMM. 
Our numerical experiments corroborate these expectations.

GeNIOS can also be shown to converge at a linear rate under strong convexity, similar to other ADMM variants \cite{deng2016global,tang2024self}.
When $f$ is strongly convex, \cite[Theorem 2]{frangella2023linear} guarantees linear convergence, provided that the subproblem errors decay \emph{geometrically}.
However, the geometric decay requirement appears to be an artifact of the analysis: 
empirically \cite{frangella2023linear} observes linear convergence even when the subproblem error sequences are only summable.

GeNIOS inherits these convergence guarantees:
it is guaranteed to converge linearly when $f$ is strongly convex and subproblem errors decay geometrically, 
and in practice still converges linearly even with less exact subproblem solves.
Figure \ref{fig:high-prec-solves} shows that \method{} converges linearly on the strongly convex elastic-net and logistic regression problems.

\paragraph{Optimality conditions.}
Optimality conditions for the problem~\eqref{eq:problem-formulation} are primal feasibility,
\[
Mx^\star - z^\star = c,
\]
and that the Lagrangian $\mathcal{L}(x, z, y) = f(x) + g(z) + \rho u^T(Mx - z - c)$
has a vanishing gradient when evaluated at the optimal primal and dual variables:
\[
\begin{aligned}
    0 &= \nabla f(x^\star) + \rho M^T u^\star \\
    0 &\in \partial g(z^\star) - \rho u^\star.
\end{aligned}
\]
In classic ADMM, $z^{k+1}$ and $u^{k+1}$ always satisfy the second condition above exactly~\cite[\S 3.3]{boyd2011distributed}.
\method{} only requires a routine that solves the $z$-subproblem inexactly,
so instead it finds $u^{k+1}$ so that $\rho u^{k+1}$ is almost a subgradient of $g$ at $z^{k+1}$:
\begin{lemma}[Approximate optimality condition~{\cite[Lemma 4]{frangella2023linear}}]
	At each iteration $k$, there exists $s^k$ with 
 \[
 \|s^k\|\leq \sqrt{\frac{2\eps^k_z}{\rho}},
 \]
 such that the approximate $z$-subproblem solution $z^{k+1}$ satisfies
    \[
    \rho u^{k+1}+s^k\in \partial_{\eps_z^k}g(z^{k+1}).
    \]
    In words, $\rho u^{k+1}+s^k$ belongs to the \emph{$\eps^k_z$-subdifferential} at $z^{k+1}$,
    which means
    \[
    g(z)-g(z^{k+1})\geq \langle \rho u^{k+1}+s^k,z-z^{k+1}\rangle-\eps_z^k, \quad \forall z\in \reals^m.
    \]
\end{lemma}
The lemma shows that at each iteration, $\rho u^{k+1}$ is nearly a subgradient of $g$ at $z^{k+1}$.
The user-provided routine to solve the $z$ subproblem must guarantee the sequence of errors $\{\sqrt{\eps^k_z}\}$ is summable.
Consequently, $\eps^k_z$ decays quickly to zero, and the error in satisfying the second optimality condition becomes negligible.
Hence \method{} only monitors the error in the first condition, $\nabla f(x^{k}) + \rho M^{T}u^{k}$, to determine when it should terminate.

\paragraph{Termination criteria.}
Based on the above discussion, we define the primal and dual residuals of problem~\eqref{eq:problem-formulation} as
\begin{align}
    \label{eq:residual-primal}
    r_\mathrm{p}^k &= Mx^k - z^k - c, \\
    \label{eq:residual-dual}
    r_\mathrm{d}^{k} &= \nabla f(x^k) + \rho M^T u^k.
\end{align}
Under the inexactness assumptions above, these residuals converge to $0$ as $k \to \infty$.
\method{} terminates when an absolute and relative criterion based on these residuals are satisfied:
\[
\begin{aligned}
    \|r_\mathrm{p}^k\|_2 &\le \sqrt{m} \eps_\mathrm{abs} + \eps_\mathrm{rel} \max\{\|Mx^k\|_2, \|z^k\|_2, \|c\|_2 \}, \\
    \|r_\mathrm{d}^k\|_2 &\le \sqrt{n} \eps_\mathrm{abs} + \eps_\mathrm{rel} \|\rho M^T u^k\|_2,
\end{aligned}
\]
In some cases, the user may wish to use another convergence criterion, which \method{} supports.
For example, in machine learning problems, \method{} supports a duality gap criterion
(see \S \ref{sec:app-ml} for details).

\subsection{Infeasibility detection}
\label{sec:infeasibility}
\method{}'s iterates will diverge if a solution does not exist to~\eqref{eq:problem-formulation}.
For conic programs~\eqref{eq:conic-program-formulation}, \method{} can detect and certify infeasiblity despite inexact solves.
Our result follows~\cite{banjac2019infeasibility} and is a direct consequence of~\cite[Thm. 2.1]{rontsis2022efficient}.
For conic programs~\eqref{eq:conic-program-formulation}, the $x$-subproblem finds an $\epsilon^k_x$-approximate solution to the linear system
\[
(P + \sigma I + \rho M^TM)x = \sigma x^k + \rho M^T(z^k-u^k+c) - q.
\]
The $z$-subproblem produces an approximate projection $z^{k+1}$ of $Mx^{k+1} + u^k - c$ onto the cone $\mathcal{K}$ satisfying
\[
\|\tilde z^{k+1}-(Mx^{k+1} + u^k - c)\|^2 - \|\Pi_{\mathcal K}(Mx^{k+1} + u^k - c)-(Mx^{k+1} + u^k - c)\|^2\leq \varepsilon_z^k,
\]
absorbing the constant $\rho/2$ into the error $\eps_z^k$.
\method{} detects infeasibility by monitoring the sequences of differences $\delta x^k = x^{k+1}-x^k$ and $\delta u^k = u^{k+1}-u^k$.

\begin{proposition}[Infeasibility certificate]
\label{thm:infeasibility}
    If the error sequences $\{\epsilon^k_x\}_k$ and $\{\sqrt{\epsilon^k_z}\}_k$ are summable, then as $k \to \infty$, the differences $\delta x^k \to \delta x$ and $\delta u^k \to \delta u$ converge.
    Further,
    \begin{enumerate}
        \item If $\delta u \neq 0$, then \eqref{eq:conic-program-formulation} is primal infeasible.
        The difference $\delta u$ provides a certificate of primal infeasibility that satisfies
        \begin{equation}
        \label{eq:primal_cert}
           M^{T}\delta u = 0 \quad \text{and} \quad S_{\mathcal K}(\delta u)<0,
        \end{equation}
        where $S_{\mathcal K}$ is the support function of $\mathcal K$.\footnote{
        The support function of $\mathcal{K}$ is defined as $S_{\mathcal K}(\delta u) = \sup_{y \in \mathcal K} y^T\delta u$.
        }
        \item If $\delta x \neq 0$, then \eqref{eq:conic-program-formulation} is dual infeasible.
        The difference $\delta x$ provides a certificate of dual infeasibility that satisfies
        \begin{equation}
        \label{eq:dual_cert}
           P\delta x = 0, \quad M\delta x \in \mathcal K^{\infty}, \quad \text{and} \quad q^{T}\delta x<0,
        \end{equation}
	where $\mathcal K^{\infty}$ is the recession cone of $\mathcal K$.\footnote{
        The recession cone of $\mathcal{K}$ is defined as $\mathcal K^{\infty} = \{y \in \reals^n \mid x + \tau y \in \mathcal{K} ~ \text{for all} ~ x\in\mathcal{K},\, \tau \ge 0\}$.
    }

        \item If $\delta x \neq 0$ and $\delta u \neq 0$, then \eqref{eq:conic-program-formulation} is both primal and dual infeasible, and the differences $\delta x$ and $\delta u$ provide certificates of infeasibility as above.
    \end{enumerate}
\end{proposition}
\begin{proof}
    $\rho$-Strong-convexity of the $z$-subproblem, along with the $z$-subproblem inexactness condition implies that
    \[
    \|z^{k+1}-\Pi_{\mathcal K}(Mx^{k+1} + u^k - c)\|\leq \sqrt{\frac{2}{\rho}\epsilon_z^k}.
    \]
    As $\sum_k \sqrt{\eps_z^k}<\infty$ by the construction of the \method{} algorithm, it follows that the errors in the solution to the $z$-subproblem are summable.
    The same holds for the $x$-subproblem errors, again by the construction of the \method{} algorithm.
    The desired result then follows immediately from Theorem 2.1 in \cite{rontsis2022efficient}, which proves the result in the case where the $x$ and $z$-subproblem solution errors are summable.
\end{proof}

\paragraph{Algorithmic certificates.}
Following~\cite{banjac2019infeasibility,rontsis2022efficient,osqp}, we translate Proposition~\ref{thm:infeasibility} into simple algorithmic certificates of infeasibility.
\method{} declares a problem to be primal infeasible if the primal infeasibility certificate~\eqref{eq:primal_cert} holds approximately:
\[
\|M^{T}\delta u^k\| < \epsilon_{\textrm{inf}} \|\delta u^k\|, \quad S_{\mathcal K}(\delta u^k)<\epsilon_{\textrm{inf}}\|\delta u^k\|,
\]
where $\epsilon_{\textrm{inf}}$ is a positive tolerance.
Similarly, \method{} declares a problem to be dual infeasible if the dual infeasibility certificate~\eqref{eq:dual_cert} holds approximately:
\[
\|P\delta x^k\|<\epsilon_{\textrm{inf}} \|\delta x^k\|, \quad \textrm{dist}_{\mathcal K^{\infty}}(M\delta x^k)<\epsilon_{\textrm{inf}}\|\delta x^k\|, \quad q^{T}\delta x^k<\epsilon_{\textrm{inf}} \|\delta x^k\|.
\]
For the case of QPs, the support functions and recession cone for the hyperrectangle $[l, u] \subseteq \reals^n$ are well-defined, even though this set is not necessarily a cone.

\subsection{Performance improvements} \label{sec:performance-imp}
\method{} includes performance improvements which are known to speed up convergence in practice and are implemented in many ADMM solvers.

\paragraph{Over-relaxation.}
In the $z$- and $u$-updates, \method{} replaces the quantity $Mx^{k+1}$ with
\[
\alpha Mx^{k+1} + (1 - \alpha)(z^k + c),
\]
where $\alpha \in (0, 2)$ is a relaxation parameter.
Experiments in the literature~\cite{eckstein1994parallel,eckstein1998operator} show empirically that $\alpha$ in the range  $[1.5, 1.8]$ can improve convergence.

\paragraph{Adjusting the penalty parameter.}
First-order algorithms, like ADMM, are sensitive to scaling of the problem data and to the penalty parameter $\rho$.
\method{} uses preconditioning to moderate the impact of the problem data scaling and selects $\rho$ using the simple rule~\cite{he2000alternating,wang2001decomposition,boyd2011distributed}
\[
\rho^{k+1} = \begin{cases}
    \tau \rho^k & \|r^k_\mathrm{p}\| > \mu \|r^k_\mathrm{d}\| \\
    \rho^k/\tau & \|r^k_\mathrm{d}\| > \mu \|r^k_\mathrm{p}\| \\
    \rho^k & \mathrm{otherwise}.
\end{cases}
\]
Since \method{} uses an indirect method, these updates can be applied cheaply; \method{} does not need to refactor a matrix.
\method{} can also update the preconditioner by simply changing a scalar parameter (see~\S\ref{sec:preconditioner}).
In practice, we find that the update only makes sense to apply every 25 iterations or so (this parameter is adjustable by the user).
In the scaled version of ADMM, the (scaled) dual variable $u^k$ must also be updated when $\rho$ changes.
These penalty parameter updates must stop after some finite number of iterations for convergence guarantees to hold.


\section{Applications}
\label{sec:apps}
In this section, we detail the three problem classes \method{} handles and provide an example of the interface for each class: 
general convex programs (as in~\eqref{eq:problem-formulation}) with the \genericsolver{}, quadratic programs with the \qpsolver{}, and regularized machine learning problems with the \mlsolver{}.

\subsection{General convex programs}
Recall that \method{} solves convex optimization problems of the form
\[
    \begin{aligned}
        &\text{minimize}     && f(x) + g(z) \\
        &\text{subject to}   && Mx - z = c,
    \end{aligned}
\]
where the variables are $x \in \reals^n$ and $z \in \reals^m$, and the problem data are the functions $f$ and $g$, the linear operator $M$, and the vector $c$.
\method{} uses multiple dispatch in the Julia programming language~\cite{bezanson2017julia} to implement fast primitives for inexact ADMM.
The base \method{} implementation uses a fully generic interface, accessible through \method{}'s \genericsolver{}, which we describe in this section.
In the subsequent sections, we will detail how this interface is specialized for quadratic programs and machine learning problems, each with its own interface.
Multiple dispatch allows \method{} to optimize performance for these problem subclasses while using the infrastructure of the \genericsolver{}.

\method{}'s \genericsolver{} uses these ingredients to define an optimization problem:
\begin{itemize}
    \item The function $f: \reals^n \to \reals$, its gradient $\nabla f: \reals^n \to \reals^n$, and a Hessian-vector product (HVP) oracle $\mathcal O_H:\mathbb \reals^{n}\times \reals^{n} \rightarrow \reals^n$, such that $\mathcal O_H(v;x) = \nabla^2 f(x)v$.
    \item The function $g: \reals^m \to \reals \cup \{+\infty\}$ and its (approximate) proximal operator, $\prox_{g/\rho}: \reals^m \to \reals^m$.
    \item The linear operator $M: \reals^n \to \reals^m$ and the vector $c \in \reals^m$.
\end{itemize}
Efficient implementations of the HVP for the $x$-subproblem 
and of the proximal operator for the $z$-subproblem
improve the performance of \method{} substantially
over a basic implementation. 
We provide several examples of these more efficient implementations here, 
and defer more (including a GPU interface\footnote{\url{https://tjdiamandis.github.io/GeNIOS.jl/dev/gpu/\#GPU-Support}}) to the package documentation.

\paragraph{Example.}
The Lasso regression problem is
\begin{equation}\label{eq:ex-lasso}
\begin{aligned}
        &\text{minimize} && (1/2)\|Ax - b\|^2_2 + \lambda \|x\|_1,
    \end{aligned}
\end{equation}
with variable $x \in \reals^n$, and problem data $A \in \reals^{N \times n}$, $b \in \reals^N$, and $\lambda \in \reals_+$.
In the form of~\eqref{eq:problem-formulation}, this problem becomes
\[
    \begin{aligned}
        &\text{minimize} && (1/2)\|Ax - b\|^2_2 + \lambda \|z\|_1 \\
        &\text{subject to} && x - z = 0,
    \end{aligned}
\]
where we have introduced a new variable $z \in \reals^n$.
To put this problem into the general form, take
\[
f(x) = (1/2)\|Ax - b\|^2_2 \qquad \text{and} \qquad g(z) = \lambda \|z\|_1
\]
with $M = I$ and $c = 0$.
The gradient of $f$ is $\nabla f(x) = A^T(Ax - b)$,
and its HVP is $\nabla^2 f(x): z \mapsto A^TA z$.
(The Hessian is constant, so the HVP is independent of the current iterate $x$.)
The Lasso is typically used when $A \in \reals^{N \times n}$ with $N \ll n$,
so the Hessian is more efficiently computed with two matrix-vector products as $(A^T (Az))$ instead of forming the $n \times n$ matrix $A^TA$.
The proximal operator of $g$ is the soft-thresholding operator,
\[
\prox_{g/\rho}(v)_i = \left(\argmin_{\tilde z} \lambda \|x\|_1 + (\rho/2)\|\tilde z - v\|_2^2\right)_i = \begin{cases}
v_i - \lambda / \rho & v_i > \lambda / \rho \\
0 & \abs{v_i} \le \lambda / \rho \\
v_i + \lambda / \rho & v_i < - \lambda / \rho.
\end{cases}
\]

\paragraph{Code example.}
To use \method{}'s \genericsolver{}, the user first defines the function $f$, its gradient, the function $g$, and its proximal operator directly in the Julia language.\footnote{
    This is not optimized code!
    Performance in the examples appearing in~\S\ref{sec:apps} has been sacrificed for the sake of clarity.
    For performant examples, please see the \method{} documentation and the code for the experiments in~\S\ref{sec:experiments}.
}

\begin{minted}[bgcolor=LightGray]{julia}
# Assume A, b, N, n, lambda have been defined
p = (; A=A, b=b, lambda=lambda)

f(x, p) = 0.5 * norm(p.A*x .- p.b)^2
function grad_f!(g, x, p)
    g .= p.A'*(p.A*x .- p.b)
end

g(z, p) = p.lambda*sum(abs, z)
soft_thresh(zi, kappa) = sign(zi) * max(0.0, abs(zi) - kappa)
function prox_g!(v, z, rho, p)
    v .= soft_thresh.(z, p.lambda/rho)
end
\end{minted}
Julia language style guidelines use an exclamation point \texttt{!} at the end of function names that modify their arguments;
here, the gradient and proximal functions both modify their first argument to avoid allocating memory.
The dot \texttt{.} next to a scalar function \emph{broadcasts} this function to act elementwise over vector arguments.
To access problem data inside the functions, we define the named tuple \texttt{p}, which is passed to each function as the last argument and will be later used to construct the solver.

To define the HVP, the user implements \method{}'s \texttt{HessianOperator} type:
\begin{minted}[bgcolor=LightGray]{julia}
struct HessianLasso{T, S<:AbstractMatrix{T}} <: HessianOperator
    A::S
    vN::Vector{T}
end
function LinearAlgebra.mul!(y, H::HessianLasso, x)
    mul!(H.vN, H.A, x)
    mul!(y, H.A', H.vN)
end
update!(::HessianLasso, ::Solver) = nothing

Hf = HessianLasso(A, zeros(N))
\end{minted}
The \texttt{update!} function updates the internal data of the \texttt{HessianOperator} object based on the current iterate.
Here, the Hessian is independent of the current iterate $x^k$, so the \texttt{update!} function does \texttt{nothing}.
The HVP is implemented with two multiplications: one by $A$, and one by $A^T$,
using a cached vector \texttt{vN} to avoid additional allocation of memory.
(Other functions do the same in our code, but [for brevity] not in this paper.)
Finally, the user defines a \genericsolver{} by combining these ingredients and \texttt{solve!}s the problem.
\begin{minted}[bgcolor=LightGray]{julia}
solver = GeNIOS.GenericSolver(
    f, grad_f!, Hf,         # f(x)
    g, prox_g!,             # g(z)
    I, zeros(n);            # M, c: Mx - z = c
    params=p
)
res = solve!(solver)
\end{minted}
Here, $M$ need not be a concrete matrix but can be any linear operator (with a \texttt{mul!} method) on a vector.
For example, Julia represents the identity matrix $I$ as a special \texttt{UniformScaling} type to efficiently dispatch linear algebra subroutines at compile time.
These routines recognize that $I$ is the identity operator and can be computed in $O(1)$ time.

In the rest of this section, we show how to specialize this solver for certain problem subclasses,
deferring problem-specific performance improvements to~\S\ref{sec:experiments}.

\subsection{Quadratic Programs}
\method{}'s \qpsolver{} solves constrained QPs of the form
\begin{equation}\label{eq:problem-formulation-qp}
    \begin{aligned}
        &\text{minimize}     && (1/2)x^TPx + q^Tx \\
        &\text{subject to}   && Mx = z \\
                            &&& l \leq z \leq u,
    \end{aligned}
\end{equation}
where $x \in \reals^n$ and $z \in \reals^m$ are variables,
and $P \in \symm_+^n$, $q \in \reals^n$, $M \in \reals^{m \times n}$, $l \in \reals^m$, and $u \in \reals^m$ are the problem data.
This problem can be cast in the form~\eqref{eq:problem-formulation} by taking $f(x) = (1/2)x^TPx + q^Tx$, $g(z) = I_{[l, u]}(z)$, and $c = 0$:
\[
    \begin{aligned}
        &\text{minimize}     && (1/2)x^TPx + q^Tx + I_{[l, u]}(z)\\
        &\text{subject to}   && Mx - z = 0.
    \end{aligned}
\]

\paragraph{ADMM subproblems.}
In the quadratic program case, the $x$-subproblem requires the (approximate) solution to the linear system
\[
(P + \rho M^T M + \sigma I)x = \sigma x^k - q + \rho M^T(z^k + c - u^k).
\]
Because $f$ is quadratic, the second-order approximation is exact when $\sigma = 0$.
The $z$-subproblem is simply a projection onto a hyperrectangle defined by $l$ and $u$, which can be computed exactly in $O(m)$ time.
Thus, the main computational bottleneck at each iteration is the solution to the linear system defining the $x$-subproblem.

\paragraph{Solving the linear system.}
By default, \method{} sketches $P$ and sets the regularization parameter in the preconditioner to be $\rho$. This approach works well when $P$ contains the problem data.
In these formulations, $M$ typically includes an identity matrix block, and therefore $M^TM \succeq I$.
If the problem data instead appear in the constraints, \method{} could  divide the system by $\rho$ and then sketch $M^TM$, but this feature is not implemented as of this writing.

\paragraph{Performance improvements.}
The unique features of \method{} and the Julia programming language allow for several performance optimizations.
Often, the matrix $P$ has a more efficient form for both storage and computation.
For example, $P$ may be the sum of a diagonal and low rank component, \ie $P = D + FF^T$,
where $D$ is diagonal and $F \in \reals^{n \times k}$ with $k \ll n$.
In this case, the user can create a custom object which stores $P$ using only $O(nk)$ numbers and computes a matrix-vector product in $O(nk)$ time.

\paragraph{Interface.}
\method{} includes a simple interface \qpsolver{} to define QPs more directly than through the \genericsolver{} interface: the user must specify $P$, $q$, $M$, $l$, and $u$ in~\eqref{eq:problem-formulation-qp}.
Alternatively, the user can access this interface by forming a QP in the JuMP modeling language~\cite{jump}.
The linear operators $P$ and $M$ can be any linear operator (with a \texttt{mul!} method) on a vector and can exploit problem structure; see in~\S\ref{sec:ex-portfolio-opt} for an example.

\paragraph{Code example.}
Reformulate the lasso problem~\eqref{eq:ex-lasso} as a QP (\cf\eqref{eq:problem-formulation-qp}):
\[
\begin{aligned}
&\text{minimize}     && (1/2)\bmat{x\\t}^T \bmat{A^TA & \\ & 0}\bmat{x\\t} + \bmat{-A^Tb \\ \lambda \ones}^T\bmat{x\\t} \\
&\text{subject to}   &&
\begin{bmatrix}0 \\ -\infty\end{bmatrix}
\leq \begin{bmatrix} I & I \\ I & -I \end{bmatrix} \begin{bmatrix}x \\ t\end{bmatrix}
\leq \begin{bmatrix}\infty \\ 0\end{bmatrix}.
\end{aligned}
\]
From this formulation, identify $P$, $q$, $M$, $l$, and $u$.
The code to solve this QP is below.

\begin{minted}[bgcolor=LightGray]{julia}
# Assume that A, b, m, n, lambda have all been defined
P = blockdiag(sparse(A'*A), spzeros(n, n))
q = vcat(-A'*b, lambda*ones(n))
M = [
    sparse(I, n, n)     sparse(I, n, n);
    sparse(I, n, n)     -sparse(I, n, n)
]
l = [zeros(n); -Inf*ones(n)]
u = [Inf*ones(n); zeros(n)]
solver = GeNIOS.QPSolver(P, q, M, l, u)
res = solve!(solver)
\end{minted}

\subsection{Machine Learning}
\label{sec:app-ml}
\method{}'s \mlsolver{} solves convex machine learning problem of the form
\begin{equation}\label{eq:problem-formlation-ml}
\begin{aligned}
&\text{minimize} && \sum_{i=1}^N \ell(a_i^Tx - b_i) + \lambda_1 \|x\|_1 + (1/2)\lambda_2 \|x\|_2^2,
\end{aligned}
\end{equation}
where $\ell: \reals \to \reals_+$ is some convex, per-sample loss function, and $\lambda_1$, $\lambda_2 \in \reals_+$ are regularization parameters.
The variable $x \in \reals^n$ is often called the (learned) model weights,
$a_i \in \reals^n$ the feature vectors, and $b_i \in \reals$ the responses for $i = 1, \dots, N$.
Put this problem in the general form~\eqref{eq:problem-formulation} by setting $f(x) = \sum_{i=1}^N \ell(a_i^Tx - b_i) + (1/2)\lambda_2\|x\|_2^2$ and $g(z) = \lambda_1 \|z\|_1$ with the constraint $x = z$.

\paragraph{Defining the problem.}
The problem~\eqref{eq:problem-formlation-ml} is defined by the per-sample loss function $\ell$, the feature matrix $A \in \reals^{N \times n}$ with rows $a_i^T$ for $i = 1, \dots, N$, the response vector $b \in \reals^N$, and the regularization parameters $\lambda_1$, $\lambda_2 \in \reals_+$.
The gradient and Hessian of $f(x)$ are
\[
\nabla f(x) = A^T \ell'(Ax - b) + \lambda_2 x \quad \text{and} \quad \nabla^2 f(x) = A^T \diag{\ell''(Ax - b)} A + \lambda_2 I,
\]
where $\ell'$ and $\ell''$ are applied elementwise.

\paragraph{Solving the linear system.}
Since $M = I$, the linear system associated with this problem~\eqref{eq:x-update-lin-sys} is always positive definite: at iteration $k$, \method{} solves the system
\[
\ifsubmit
    \begin{aligned}
    \left(A^T \diag{{\ell''}^k} A + (\lambda_2 + \rho) I\right)x = &\left(A^T \diag{{\ell''}^k} A + \lambda_2I\right) x^k \\
                                                                   &- A^T {\ell'}^k - \lambda_2 x^k + \rho(z^k + c - u^k),
    \end{aligned}
\else
    \left(A^T \diag{{\ell''}^k} A + (\lambda_2 + \rho) I\right)x = \left(A^T \diag{{\ell''}^k} A + \lambda_2I\right) x^k - A^T {\ell'}^k - \lambda_2 x^k - \rho(z^k - c + u^k),
\fi
\]
where ${\ell'}^k$, ${\ell''}^k \in \reals^N$ are shorthand for $\ell'$ and $\ell''$ applied elementwise to the vector $Ax^k - b$.
To construct the preconditioner, \method{} sketches $A^T \diag{{\ell''}^k} A$ and adds a regularization parameter of $\rho + \lambda_2$.
For some problems, including lasso regression, $\ell''$ is constant, and \method{} only need to sketch the left-hand-side matrix once.
However, for others, including logistic regression (see~\S\ref{sec:ex-logistic}), \method{} re-sketches the matrix occasionally as the weights change.

\paragraph{Custom convergence criterion.}
For machine learning problems, \method{} uses a
bound on the duality gap of~\eqref{eq:problem-formlation-ml} to determine convergence.
The solver constructs a dual feasible point $\nu$ for the dual function of an equivalent reformulation of~\eqref{eq:problem-formlation-ml} and then uses the termination criterion
\[
    \frac{\ell(x) - g(\nu)}{\min\left(\ell(x), \lvert{g(\nu)}\rvert\right)} \le \eps_\mathrm{dual}.
\]
We call the quantity on the left the \emph{relative duality gap}, as it is clearly a bound for the same quantity evaluated at the optimal dual variable $\nu^\star$. For a full derivation, see appendix~\ref{appendix:dual-gap}.
\method{} also allows the user to specify other custom criteria, such as the relative change in the loss function or the norm of the gradient.

\paragraph{Interface.}
To use \method{}'s machine learning interface \mlsolver{}, the user defines the per-sample loss function $\ell:\reals \to \reals_+$ and the nonnegative regularization parameters $\lambda_1$ and $\lambda_2$, as well as the data matrix $A$ and response vector $b$.
\method{} can use forward-mode automatic differentiation (using \texttt{ForwardDiff.jl}~\cite{RevelsLubinPapamarkou2016}) to define the first and second derivatives of $\ell$, or these can also be supplied directly by the user.
\method{} defaults to using the relative duality gap convergence criterion for lasso, elastic net, and logistic regression problems, which have their own interfaces.
The user must provide the conjugate function of the per-sample loss, $\ell^*$, to compute the duality gap for a custom loss $\ell$.

\paragraph{Code example.}
The lasso problem~\eqref{eq:ex-lasso} is easily recognized as a machine learning problem~\eqref{eq:problem-formlation-ml} without transformation.
The user may call the solver directly:
\begin{minted}[bgcolor=LightGray]{julia}
# Assume that A, b, m, n, lambda have all been defined
f(x) = 0.5*x^2
reg_l1 = lambda
reg_l2 = 0.0
fconj(x) = 0.5*x^2
solver = MLSolver(f, reg_l1, reg_l2, A, b; fconj=fconj)
res = solve!(solver; options=SolverOptions(use_dual_gap=true))
\end{minted}
For the lasso, elastic net, and logistic regression problems, \method{} also provides specialized interfaces that only require the regularization parameter(s).
The Lasso interface is:
\begin{minted}[bgcolor=LightGray]{julia}
solver = GeNIOS.LassoSolver(lambda, A, b)
res = solve!(solver; options=SolverOptions(use_dual_gap=true))
\end{minted}


\section{Numerical experiments}
\label{sec:experiments}
\begin{table}[H]
    \centering
    \ra{1.3}
    \begin{tabular}{@{}lr@{}}
    \toprule
        parameter & default value \\
    \midrule
        linear system offset $\sigma$ (\S\ref{sec:linsys}) & \texttt{1e-6} \\
        linear system tolerance exponent $\gamma$ (\S\ref{sec:linsys}) & $1.2$  \\
        sketch size (\S\ref{sec:preconditioner}) & $\min(50,\, n/20)$ \\
        resketching frequency (\S\ref{sec:preconditioner}) & every $20$ iterations \\
        norm for residuals (\S\ref{sec:convergence}) & $\ell_2$ \\
        stopping tolerances $\eps_\mathrm{abs}$ and $\eps_\mathrm{rel}$ (\S\ref{sec:convergence}) & \texttt{1e-4} \\
        infeasibility tolerance $\eps_\mathrm{inf}$ (\S\ref{sec:infeasibility}) & \texttt{1e-8} \\
        over-relaxation parameter $\alpha$ (\S\ref{sec:performance-imp}) & $1.6$ \\
        penalty update factor $\tau$ (\S\ref{sec:performance-imp}) & $2$ \\
        penalty update threshold $\mu$ (\S\ref{sec:performance-imp}) & $10$\\
    \bottomrule
    \end{tabular}
    \caption{Default parameters for \method{}}
    \label{tab:default-params}
\end{table}

In this section, we showcase the performance of \method{} on a variety of problems with simulated and real-world data.
The first four examples are all machine learning problems.
The final two examples come from finance and signal processing respectively.
Each highlights one or more specific features of \method{}.
For some examples, we compare \method{} against popular open-source, ADMM solvers \osqp{} and \cosmo{}, and against commercial, interior-point solver \texttt{Mosek}.
For other examples, we are primarily interested in comparing different \method{} interfaces or options against each other.
In brief, the numerical examples are outlined below:
\begin{itemize}
    \item The \emph{elastic net} problem demonstrates the impact of \method{}'s Nystr\"{o}m preconditioning and inexact subproblem solves on a dense, low-rank machine learning problem.
    \item The \emph{logistic regression} problem shows the impact of \method{}'s approximation of the $x$-subproblem and of avoiding conic reformulation.
    \item The \emph{Huber regression} problem compares \method{}'s \mlsolver{} to a conic reformulation.
    \item The \emph{constrained least squares} problem compares \method{} to \osqp{}, \cosmo{}, and \texttt{Mosek} on a dense, low-rank QP.
    \item The \emph{portfolio optimization} problem compares \method{}'s \qpsolver{} and \ifsubmit \\\fi \genericsolver{} to \osqp{}, \cosmo{}, and \texttt{Mosek} on a sparse, structured QP.
    \item Finally, the \emph{signal decomposition} problem demonstrates that \method{}'s \ifsubmit \\\fi\genericsolver{} can be used to solve nonconvex problems as well. (Of course, \method{} loses the convergence guarantees of the convex case.)
\end{itemize}

All numerical examples can be found in the repository,
\begin{center}
   \texttt{https://github.com/tjdiamandis/GeNIOSExperiments.jl}.
\end{center}
All experiments were run using \method{} v0.2.0 on a  MacBook Pro with a M1 Max processor (8 performance cores) and 64GB of RAM.
Unless otherwise stated, we use default parameters for \method{}, listed in table~\ref{tab:default-params}.
The experiments employ no parallelization except for multithreaded BLAS.
Of course, \method{}'s use of CG to solve the linear system can naturally benefit from additional parallelization, and many proximal operators can be efficiently parallelized as well.
We leave exploration of GPU acceleration and other forms of parallelism to future work.
Similar examples to those in this section, with toy data, can be found in the \texttt{examples} section of the documentation.


\subsection{Elastic net}\label{sec:ex-elastic-net}
We first use the elastic net problem, a standard quadratic machine learning problem, to highlight the advantages of randomized
preconditioning and of inexact subproblem solves. (Because the objective is quadratic, the $x$-subproblem is not approximated here.)
The elastic net problem is
\[
\begin{aligned}
    &\text{minimize} && (1/2)\|Ax - b\|_2^2 + \lambda_1 \|x\|_1 + (\lambda_2 / 2)\|x\|_2^2,
\end{aligned}
\]
where $x \in \reals^n$ is the variable, $A \in \reals^{N \times n}$ is the feature matrix, $b \in \reals^N$ is the response vector, and $\lambda_1$, $\lambda_2$ are regularization parameters.
The lasso problem can be recovered by setting $\lambda_2 = 0$.
Clearly, this problem is a special case of~\eqref{eq:problem-formulation} with
\[
f(x) = (1/2)\|Ax - b\|_2^2 + (\lambda_2/2)\|x\|_2^2 \quad \text{and} \quad g(x) = \lambda_1 \|x\|_1,
\]
and of the machine learning problem formulation~\eqref{eq:problem-formlation-ml} with $\ell(w) = (1/2)w^2$.
The second derivative is fixed across iterations, so \method{} never needs to update the sketch used to build the preconditioner.
(The preconditioner itself may be updated if the penalty parameter $\rho$ changes.)
Furthermore, since the Hessian of $f$ is constant, the $x$-subproblem is not approximated---the linear system is simply solved inexactly.

\paragraph{Problem data.} We solve the elastic net problem with both the sparse real-sim dataset~\cite{libsvm-realsim} and the dense YearMSD dataset~\cite{misc_yearpredictionmsd_203}, which we augment with random features~\cite{rahimi2007random,rahimi2008uniform} as in~\cite{pmlr-v162-zhao22a}.
The dataset statistics are summarized in table~\ref{tab:datasets}.
We set $\lambda_2 = \lambda_1 = 0.1\|A^T b\|_\infty$.
Both of these datasets are approximately low rank, and we estimate the maximum and minimum eigenvalues of the Gram matrix $A^TA$ using the randomized Lanczos method~\cite[Alg. 5]{martinsson2020randomized}.
In both cases, the minimum eigenvalue is significantly less than the penalty parameter $\rho = 1$.
As a result, the condition number of the $\rho$-regularized linear system is approximately the maximum eigenvalue $\lambda_\mathrm{max}$.
\begin{table}[h]
    \centering
    \ra{1.3}
\ifsubmit\begin{adjustbox}{max width=\textwidth}\fi
\begin{tabular}{@{}lrrrrrr@{}}
\toprule
                & samples $N$   & features $n$  & nonzeros  & density   & $\lambda_\mathrm{max}$ (est.) & $\lambda_\mathrm{min}$ (est.) \\
\midrule
real-sim        & 72.3k         & 21.0k         & 3.709M    & 0.24\%    & 920.8                         & 0.0020\\
YearMSD         & 10k           & 20k           & 200M      & 100\%     & 4450                          & 0.0002\\
\bottomrule
\end{tabular}
\ifsubmit\end{adjustbox}\fi
    \caption{Dataset statistics.}
    \label{tab:datasets}
\end{table}

\paragraph{Experiments.}
We examine the impact of the preconditioner and the inexact solves for the elastic net problem with each of these datasets.
Figure~\ref{fig:en-sparse} shows the convergence on the sparse real-sim dataset, and figure~\ref{fig:en-dense} shows the convergence on the dense YearMSD dataset.
For both datasets, inexact subproblem solves speed up convergence by about 2$\times$.
For the dense dataset, the preconditioner reduces the time to solve the linear system by approximately 50\%. 
For the sparse dataset, the preconditioner provides only a modest reduction in the time to solve the linear system and so a modest improvement in the overall solve time: the dataset is so sparse that applying the dense preconditioner introduces considerable overhead compared to the low cost of a CG iteration.
Detailed timings are in tables~\ref{tab:elastic-net-timings-sparse} and~\ref{tab:elastic-net-timings-dense}.
We also show a high precision solve in figures~\ref{fig:elastic-net-high-prec} and \ref{fig:elastic-net-high-prec-2}, which illustrates the linear convergence of \method{} past the stopping tolerances used in our comparisons. \method{} outperforms other general-purpose convex optimization solvers by a factor of 5-10x on these problems. We provide a comparison in table~\ref{tab:elastic-net-compare} in appendix~\ref{appendix:timing}.

\begin{figure}[h]
    \captionsetup[sub]{font=scriptsize}
    \centering
    \begin{subfigure}[t]{\ifsubmit 0.48\textwidth \else 0.48\textwidth \fi}
        \centering
        \includegraphics[width=\columnwidth]{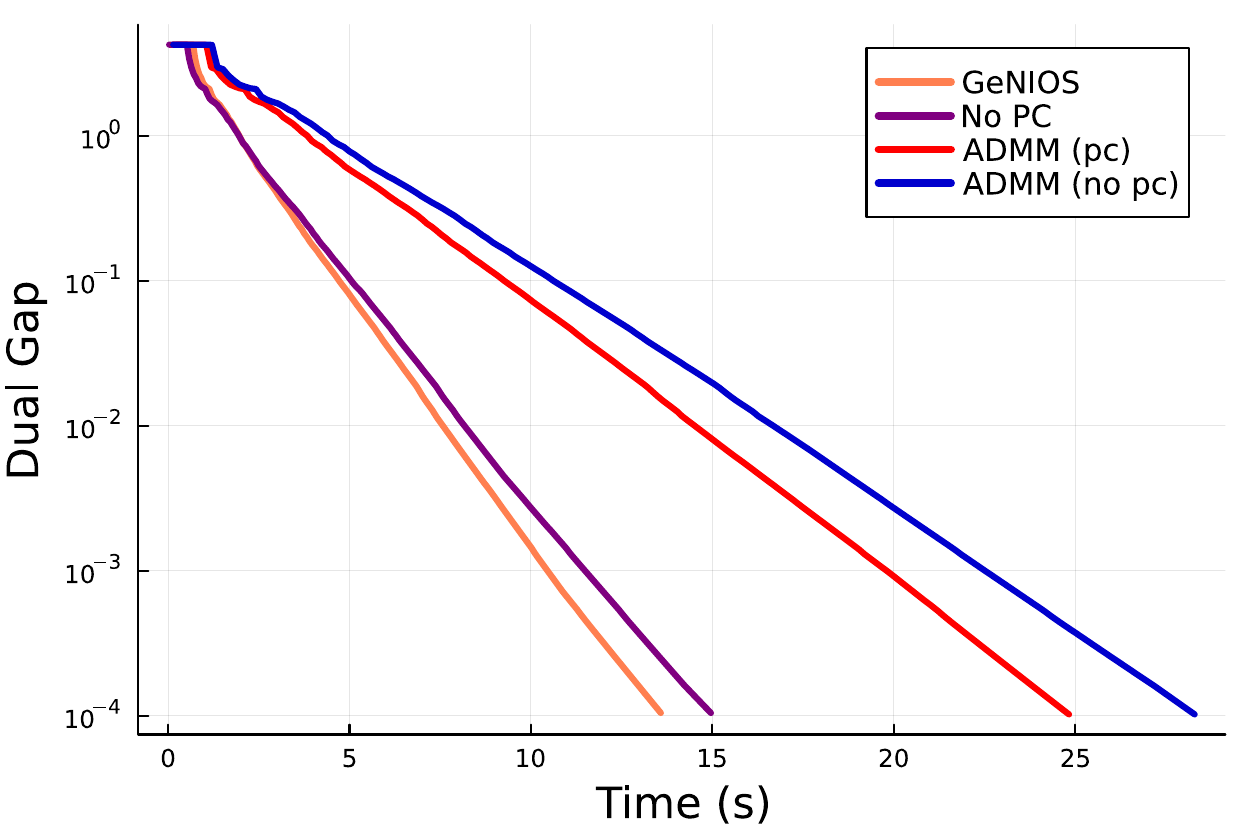}
        \caption{Relative duality gap}
    \end{subfigure}
    \hfill
    \begin{subfigure}[t]{\ifsubmit 0.48\textwidth \else 0.48\textwidth \fi}
        \centering
        \includegraphics[width=\columnwidth]{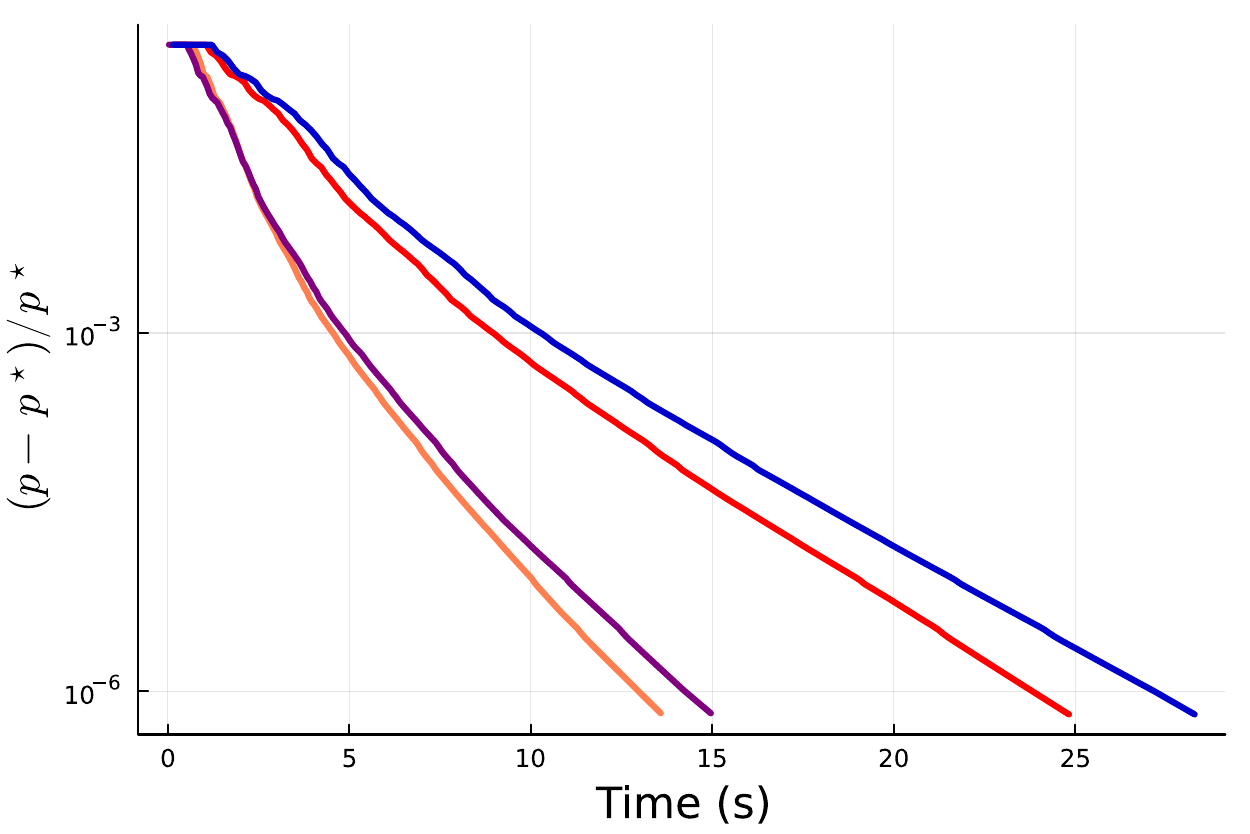}
        \caption{Objective value}
    \end{subfigure}
    \begin{subfigure}[t]{\ifsubmit 0.48\textwidth \else 0.48\textwidth \fi}
        \centering
        \includegraphics[width=\columnwidth]{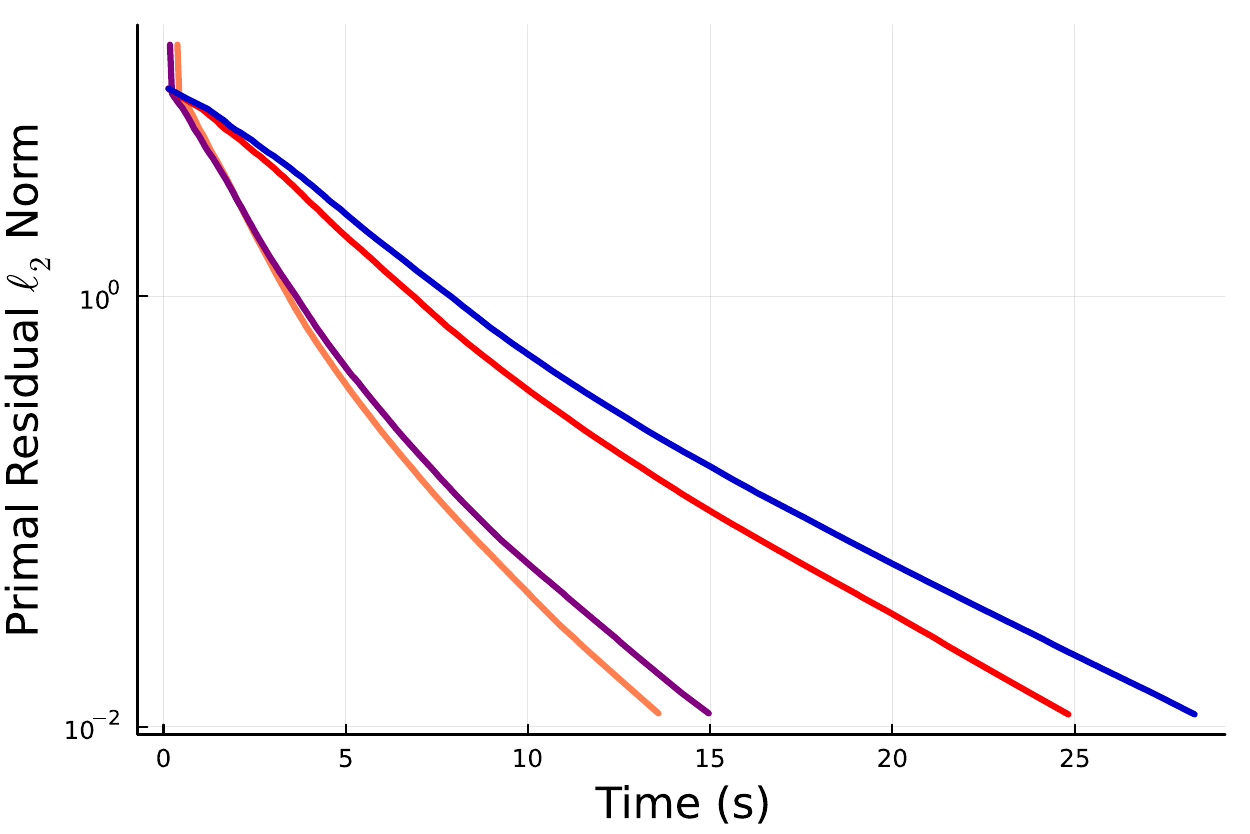}
        \caption{Primal residual}
    \end{subfigure}
    \hfill
    \begin{subfigure}[t]{\ifsubmit 0.48\textwidth \else 0.48\textwidth \fi}
        \centering
        \includegraphics[width=\columnwidth]{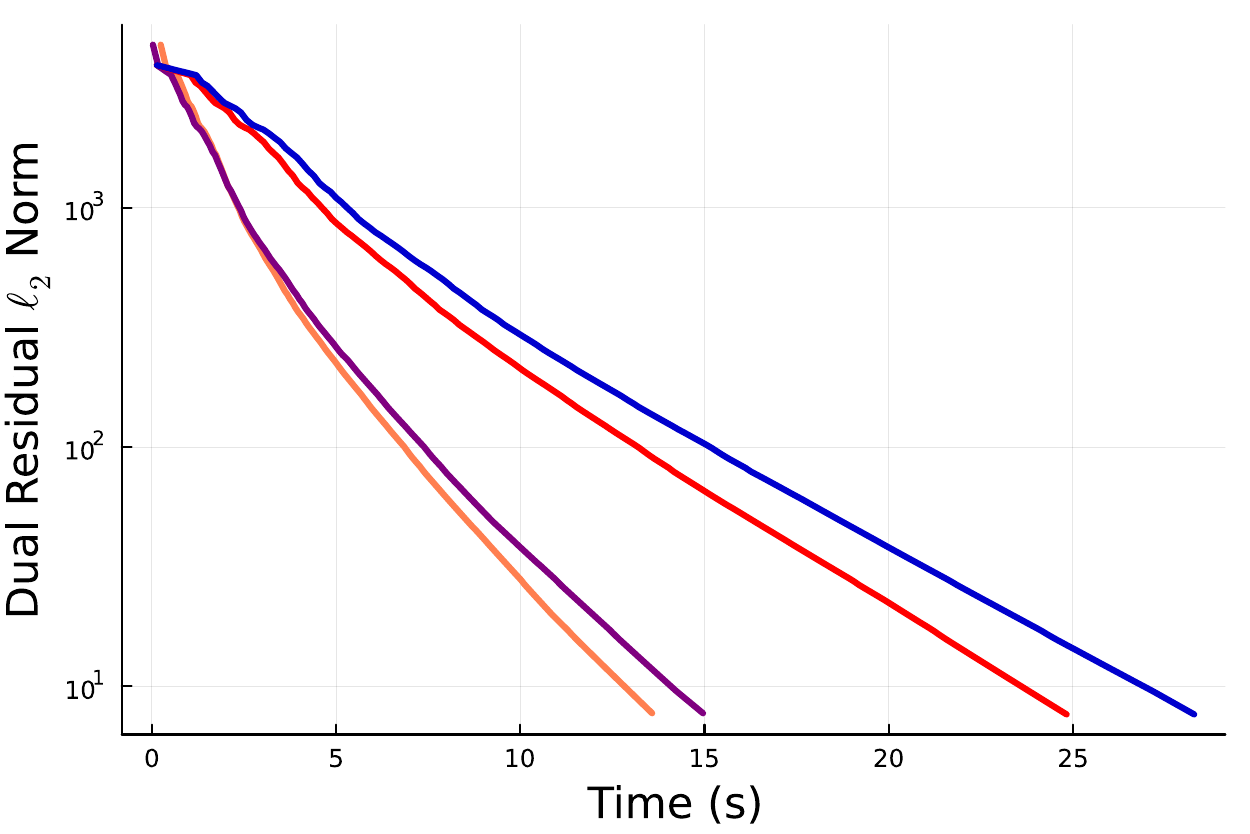}
        \caption{Dual residual}
    \end{subfigure}
     \caption{Inexact subproblem solves improve \method{}'s convergence time on the elastic net problem with the real-sim dataset, but the preconditioner has little effect since the dataset is very sparse. Note that `pc' (`no pc') indicates that we did (did not) use a preconditioner.}
     \label{fig:en-sparse}
\end{figure}

\begin{figure}[h]
    \captionsetup[sub]{font=scriptsize}
    \centering
    \begin{subfigure}[t]{\ifsubmit 0.48\textwidth \else 0.48\textwidth \fi}
        \centering
        \includegraphics[width=\columnwidth]{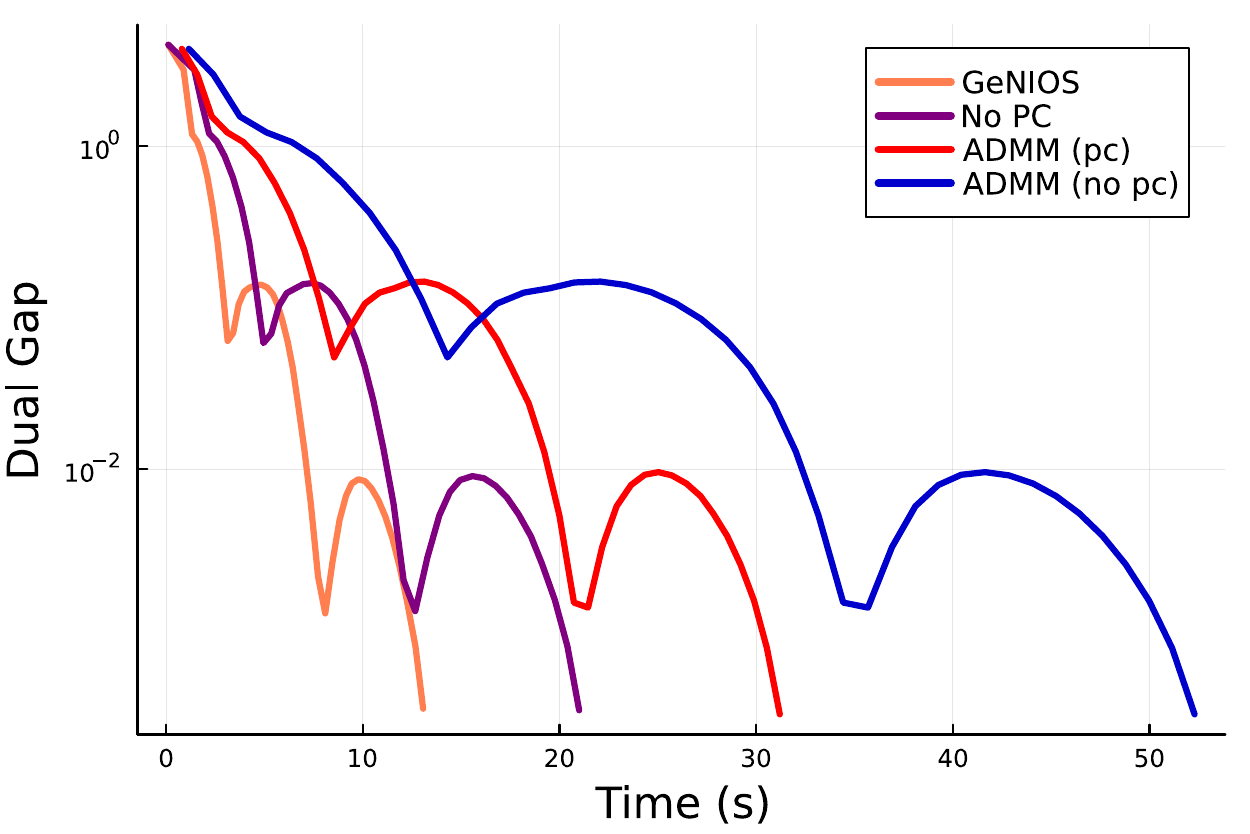}
        \caption{Relative duality gap}
    \end{subfigure}
    \hfill
    \begin{subfigure}[t]{\ifsubmit 0.48\textwidth \else 0.48\textwidth \fi}
        \centering
        \includegraphics[width=\columnwidth]{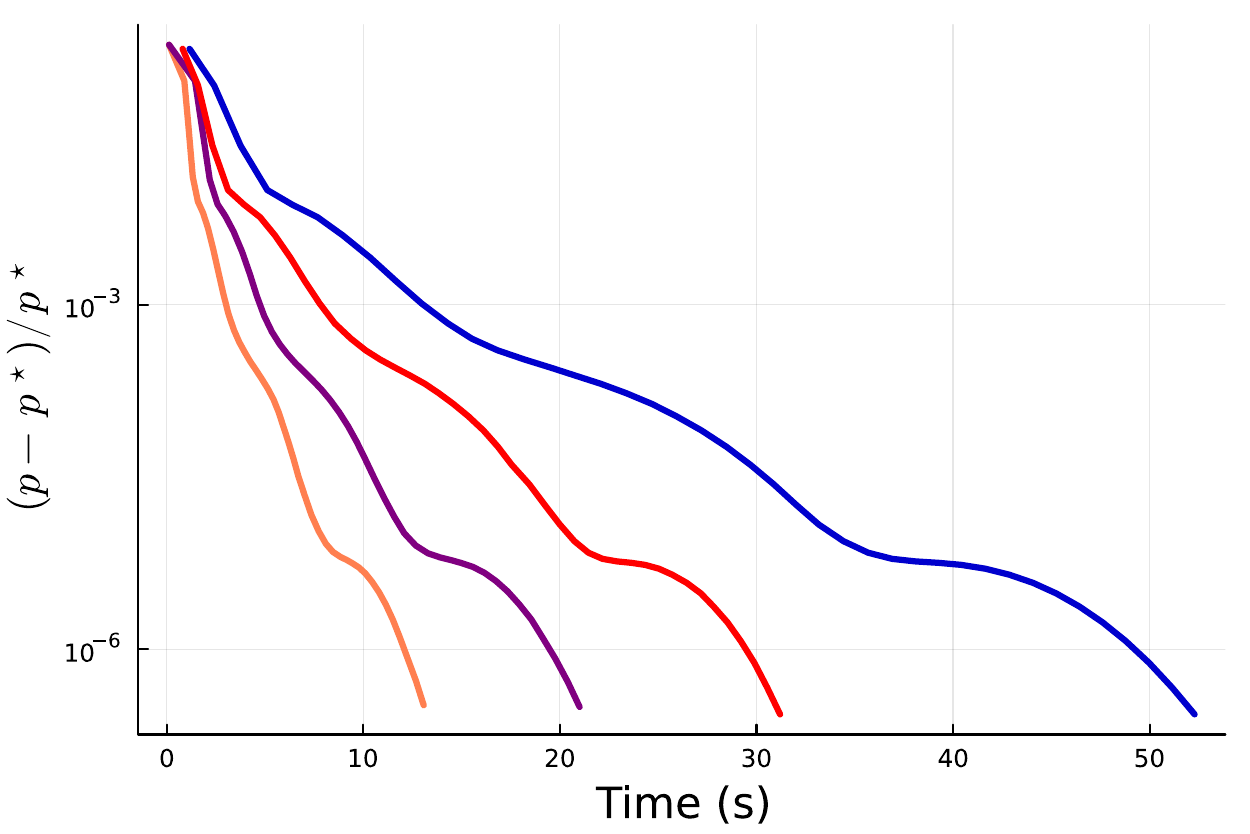}
        \caption{Objective value}
    \end{subfigure}
    \begin{subfigure}[t]{\ifsubmit 0.48\textwidth \else 0.48\textwidth \fi}
        \centering
        \includegraphics[width=\columnwidth]{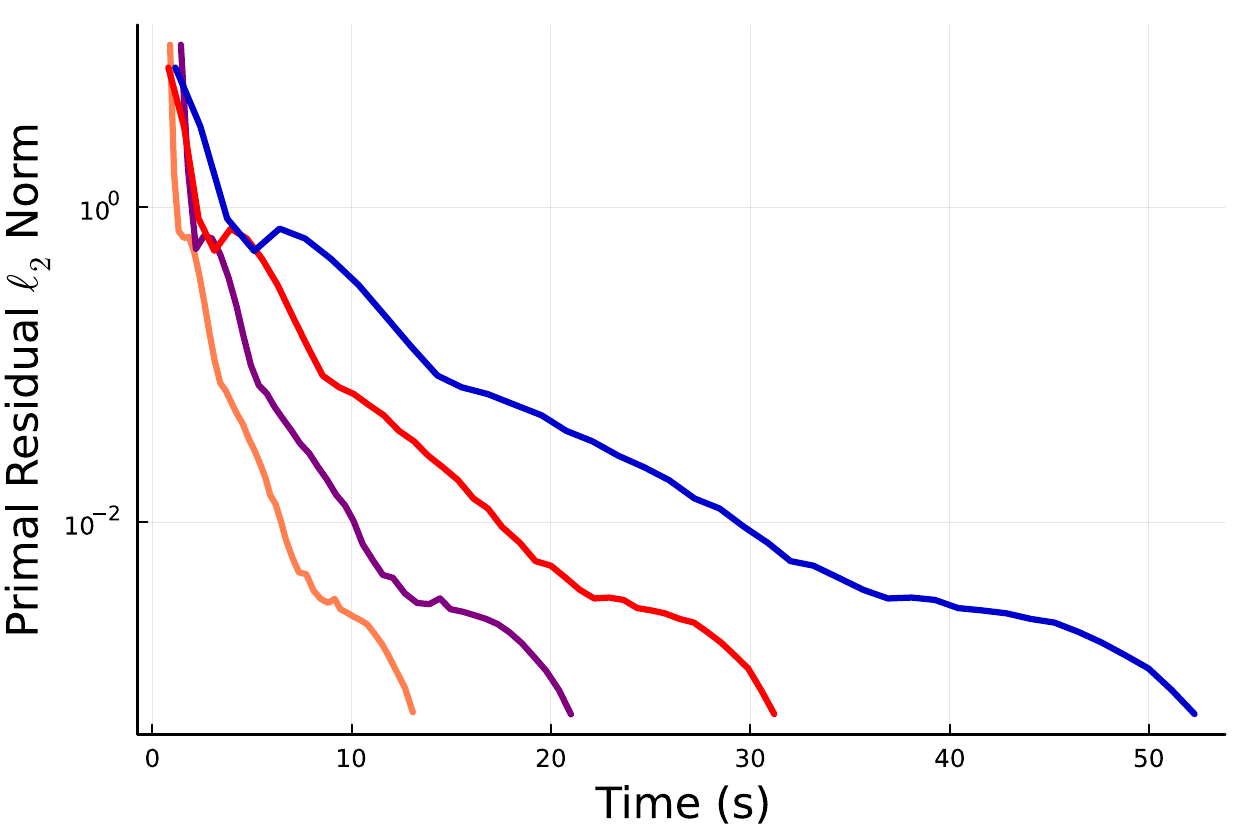}
        \caption{Primal residual}
    \end{subfigure}
    \hfill
    \begin{subfigure}[t]{\ifsubmit 0.48\textwidth \else 0.48\textwidth \fi}
        \centering
        \includegraphics[width=\columnwidth]{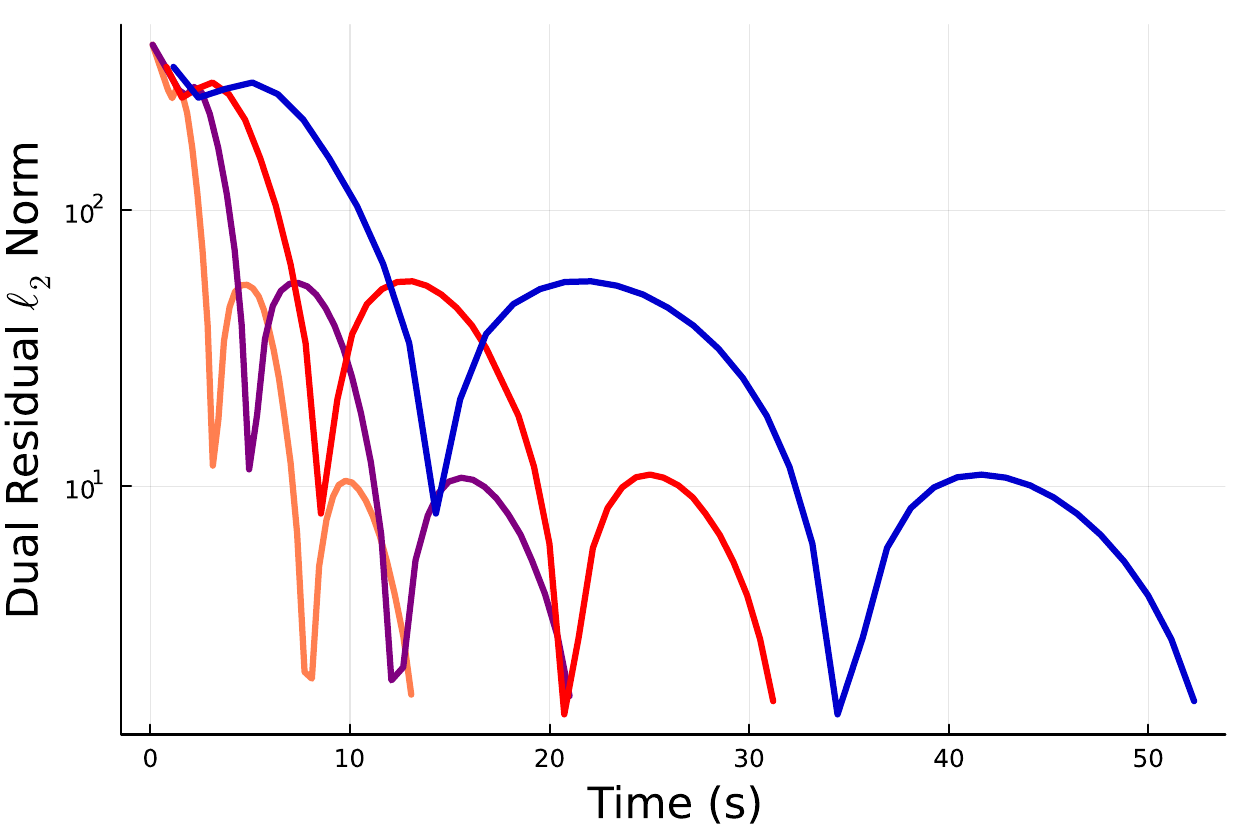}
        \caption{Dual residual}
    \end{subfigure}
     \caption{Both inexact subproblem solves and the preconditioner improve \method{}'s convergence time on the elastic net problem with the dense YearMSD dataset.}
     \label{fig:en-dense}
\end{figure}

\begin{table}[h]
    \centering
    \ra{1.3}
\ifsubmit\begin{adjustbox}{max width=\textwidth}\fi
\begin{tabular}{@{}lrrrr@{}}
\toprule
&GeNIOS & GeNIOS (no pc) & ADMM & ADMM (no pc)\\
\midrule
setup time (total) & 0.837s & 0.011s & 0.401s & 0.012s\\
\quad preconditioner time & 0.802s & 0.000s & 0.390s & 0.000s\\
solve time & 13.659s & 15.053s & 24.954s & 28.428s\\
\quad number of iterations &  190 &  190 &  190 &  190\\
\quad total linear system time & 8.682s & 10.039s & 19.898s & 23.356s\\
\quad avg. linear system time & 45.695ms & 52.838ms & 104.724ms & 122.928ms\\
\quad total prox time & 0.004s & 0.003s & 0.004s & 0.004s\\
\quad avg. prox time & 0.019ms & 0.018ms & 0.018ms & 0.019ms\\
total time & 14.497s & 15.065s & 25.355s & 28.440s\\
\bottomrule
\end{tabular}
\ifsubmit\end{adjustbox}\fi
    \caption{Timings for \method{} with and without preconditioning (indicated by `pc' and `no pc' respectively) and inexact solves for the elastic net problem with the sparse real-sim dataset.}
    \label{tab:elastic-net-timings-sparse}
\end{table}

\begin{table}[h]
    \centering
    \ra{1.3}
\begin{adjustbox}{max width=\textwidth}
\begin{tabular}{@{}lrrrr@{}}
\toprule
&GeNIOS & GeNIOS (no pc) & ADMM & ADMM (no pc)\\
\midrule
setup time (total) & 2.072s & 0.052s & 1.893s & 0.050s\\
\quad preconditioner time & 2.023s & 0.000s & 1.844s & 0.000s\\
solve time & 13.477s & 21.603s & 31.858s & 53.345s\\
\quad number of iterations &   42 &   42 &   42 &   42\\
\quad total linear system time & 8.394s & 16.350s & 26.517s & 47.612s\\
\quad avg. linear system time & 199.848ms & 389.286ms & 631.365ms & 1133.625ms\\
\quad total prox time & 0.001s & 0.001s & 0.001s & 0.001s\\
\quad avg. prox time & 0.018ms & 0.018ms & 0.018ms & 0.018ms\\
total time & 15.549s & 21.655s & 33.751s & 53.395s\\
\bottomrule
\end{tabular}
\end{adjustbox}
    \caption{Timings for \method{} with and without preconditioning and inexact solves for the elastic net problem with the dense YearMSD dataset.}
    \label{tab:elastic-net-timings-dense}
\end{table}

\begin{figure}[h]
    \captionsetup[sub]{font=scriptsize}
    \centering
    \begin{subfigure}[t]{0.32\textwidth}
        \centering
        \includegraphics[width=\columnwidth]{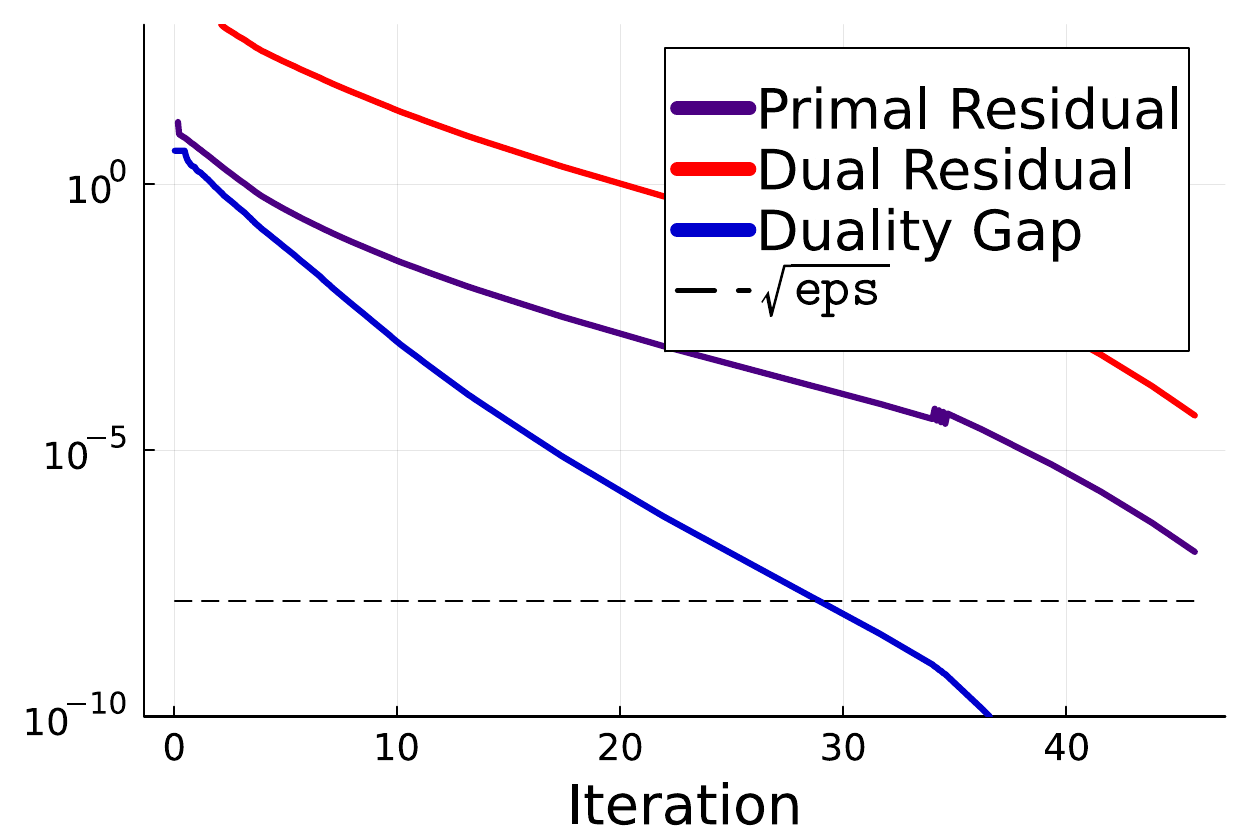}
        \caption{Elastic net, real-sim}
        \label{fig:elastic-net-high-prec}
    \end{subfigure}
    \hfill
    \begin{subfigure}[t]{0.32\textwidth}
        \centering
        \includegraphics[width=\columnwidth]{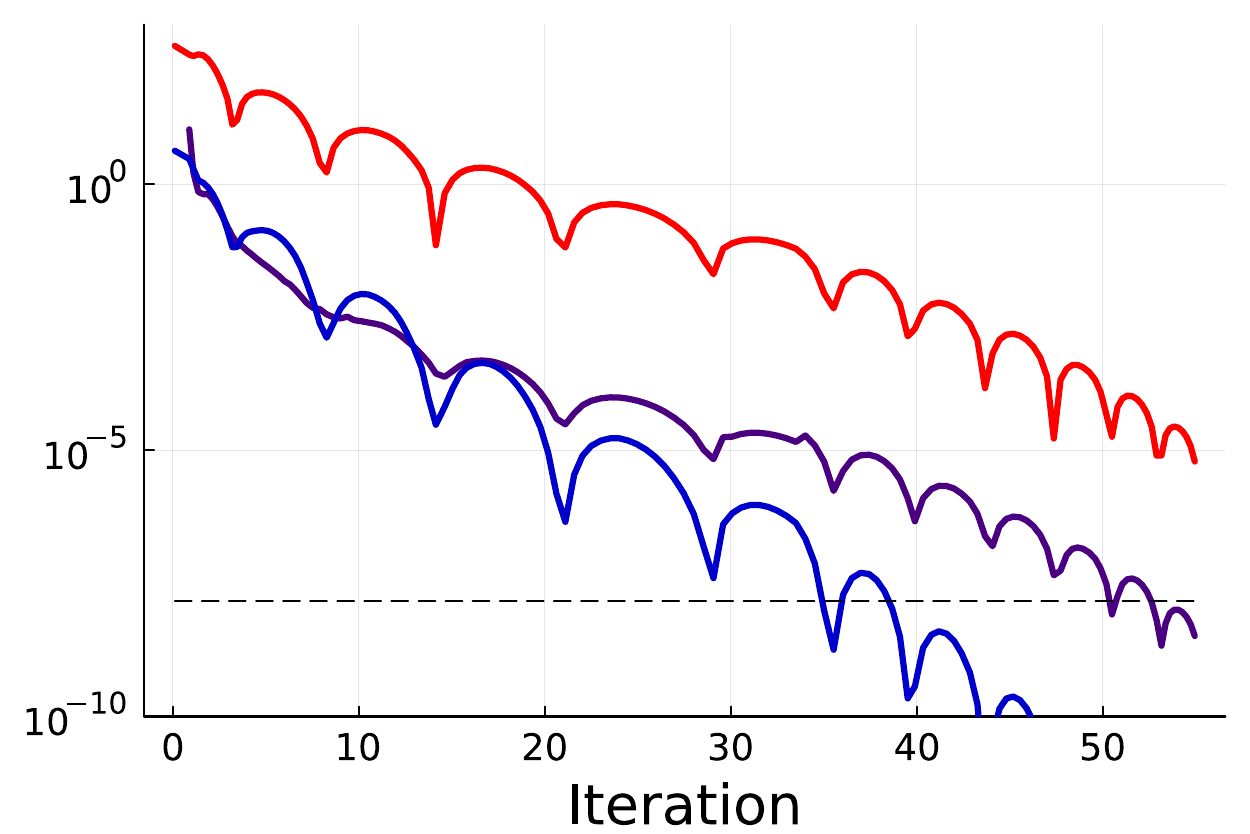}
        \caption{Elastic net, YearMSD}
        \label{fig:elastic-net-high-prec-2}
    \end{subfigure}
    \begin{subfigure}[t]{0.32\textwidth}
        \includegraphics[width=\columnwidth]{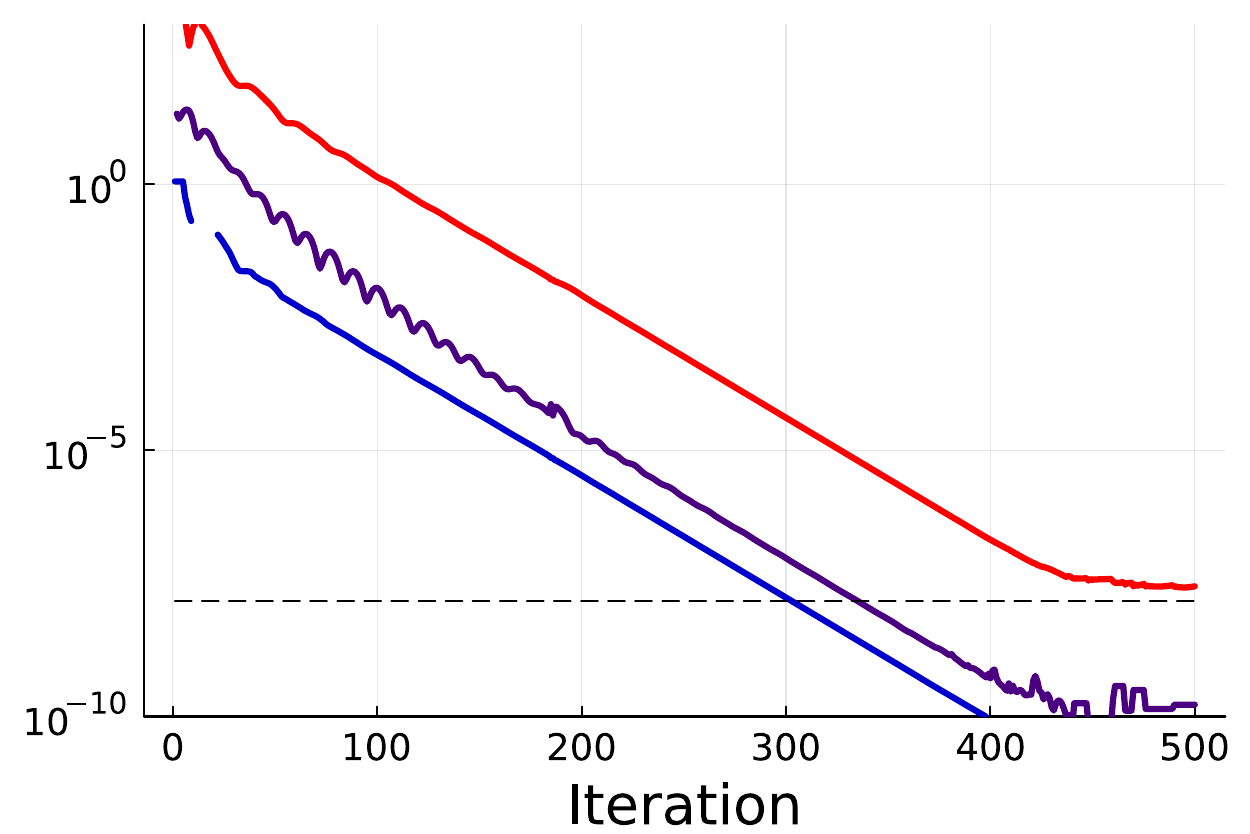}
        \caption{Logistic regression, real-sim}
        \label{fig:logistic-high-prec}
    \end{subfigure}
    \caption{\method{} converges linearly on the elastic net problem (a, b) and the logistic regression problem (c) up to high accuracy. Here, \texttt{eps} denotes machine epsilon.}
    \label{fig:high-prec-solves}
\end{figure}

\subsection{Logistic Regression}
\label{sec:ex-logistic}
This example highlights the benefits of using an approximate $x$-subproblem, in addition to an inexact solve.
For the logistic regression problem, with problem data $\tilde a_i \in \reals^n$ and $\tilde b_i \in \reals$ for $i = 1, \dots, N$, where $\tilde b_i \in \{\pm 1\}$, define
\[
\mathbb{P}\left({\tilde b_i \mid \tilde a_i}\right) = \frac{1}{1 + \exp\left(\tilde b_i(\tilde a_i^Tx)\right)} = \frac{1}{1 + \exp(a_i^Tx)},
\]
where $a_i = \tilde b_i \tilde a_i$.
Use the negative of the log likelihood as the loss, giving
the optimization problem
\[
\begin{aligned}
& \text{minimize} && \sum_{i=1}^m \log\left(1 + \exp(a_i^Tx)\right) + \lambda_1 \|x\|_1.
\end{aligned}
\]
Recognize that, in the form of~\eqref{eq:problem-formlation-ml}, the per-sample loss function is $\ell(w) = \log(1 + \exp(w))$.

\paragraph{Experiments.}
We use the real-sim dataset as in the elastic net experiment (see table~\ref{tab:datasets}), which corresponds to a binary classification problem and take $\lambda_1 = 0.1\|A^T\ones\|_\infty$.
We solve this problem with three distinct methods: \method{}'s \mlsolver{}, with and without both preconditioning and inexact solves; 
the \mlsolver{} with an exact $x$-update, which we solve with \texttt{Optim.jl}'s~\cite{mogensen2018optim} implementation of L-BFGS~\cite{nocedal1980-lbfgs,liu1989-lbfgs} with default parameters\footnote{
    The stopping criterion is when the infinity norm of the gradient is under \texttt{1e-8}.
}; 
and \method{}'s \qpsolver{}, which we modified to handle exponential cone constraints.
In the conic form problem, we solve the $x$-subproblem inexactly and we use the fast projection onto the exponential cone from Friberg~\cite{friberg2023projection} for the $z$-subproblem.
In our formulation, the conic form problem has $2n + 5N$ variables and $2n + 9N$ constraints instead of the $2n$ variables and $n$ constraints in the \mlsolver{} formulation (see appendix~\ref{app:logistic-conic} for the equivalent conic problem).
We show convergence in figure~\ref{fig:logistic} and breakdown timing in table~\ref{tab:logistic}. 
The conic form solve is not plotted because its residuals refer to different quantities, making direct comparison difficult.
In this example, using the preconditioner does not add more of a speedup than its overhead since the dataset is quite sparse.
Standard ADMM is slower than \method{}, but it is faster than using an exact solve for the approximate $x$-subproblem.
The conic form problem has the slowest solve time, and its per-iteration time is longer than \method{} with inexact solves.

\paragraph{Discussion.}
The results in table~\ref{tab:logistic} tell an interesting story.
First, as in the previous example, inexact solves of the approximate $x$-subproblem have little impact on the number of iterations \method{} takes to converge.
However, there is a clear trade off between approximating the $x$-subproblem and solving the exact problem: the exact problem takes longer to solve but significantly cuts down the number of ADMM iterations the algorithm takes to converge.
The conic form's $x$-update is also faster than standard ADMM with an exact solve, as the linear system is solved inexactly and the constraint matrix is very sparse and structured. 
However, the algorithm pays for this speed in the $z$-update, which requires computing $2N$ projections onto the exponential cone.
Overall, these results highlight the benefits of solving logistic regression in a more natural form, such as~\eqref{eq:problem-formlation-ml}, instead of the conic form, and of approximating the $x$-subproblem to speed up solve time, even if the number of iterations required to converge increases.
\method{} also outperforms Mosek on these problems by a factor of 2-4x. (See table~\ref{tab:logistic-compare} in appendix~\ref{appendix:timing}).

\begin{figure}[h]
    \captionsetup[sub]{font=scriptsize}
    \centering
    \begin{subfigure}[t]{0.48\textwidth}
        \centering
        \includegraphics[width=\columnwidth]{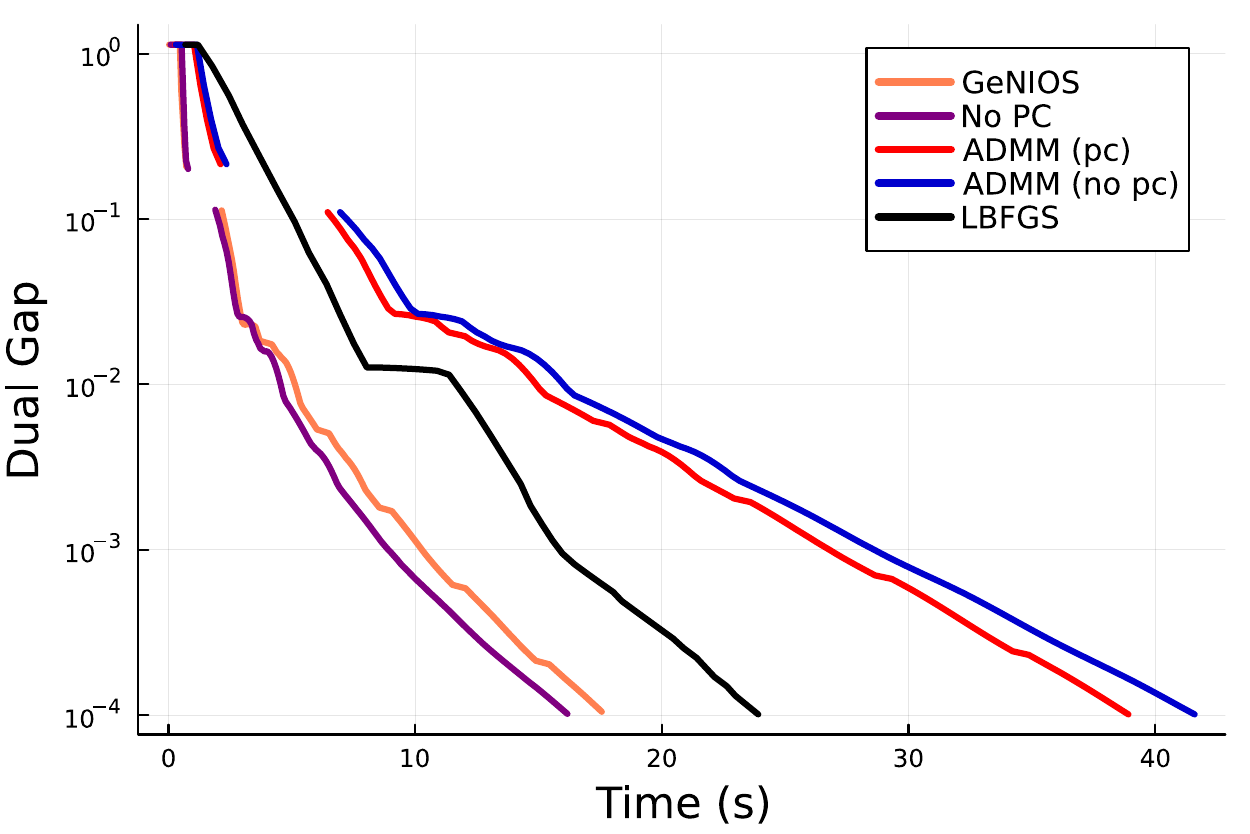}
        \caption{Relative duality gap convergence}
    \end{subfigure}
    \hfill
    \begin{subfigure}[t]{0.48\textwidth}
        \centering
        \includegraphics[width=\columnwidth]{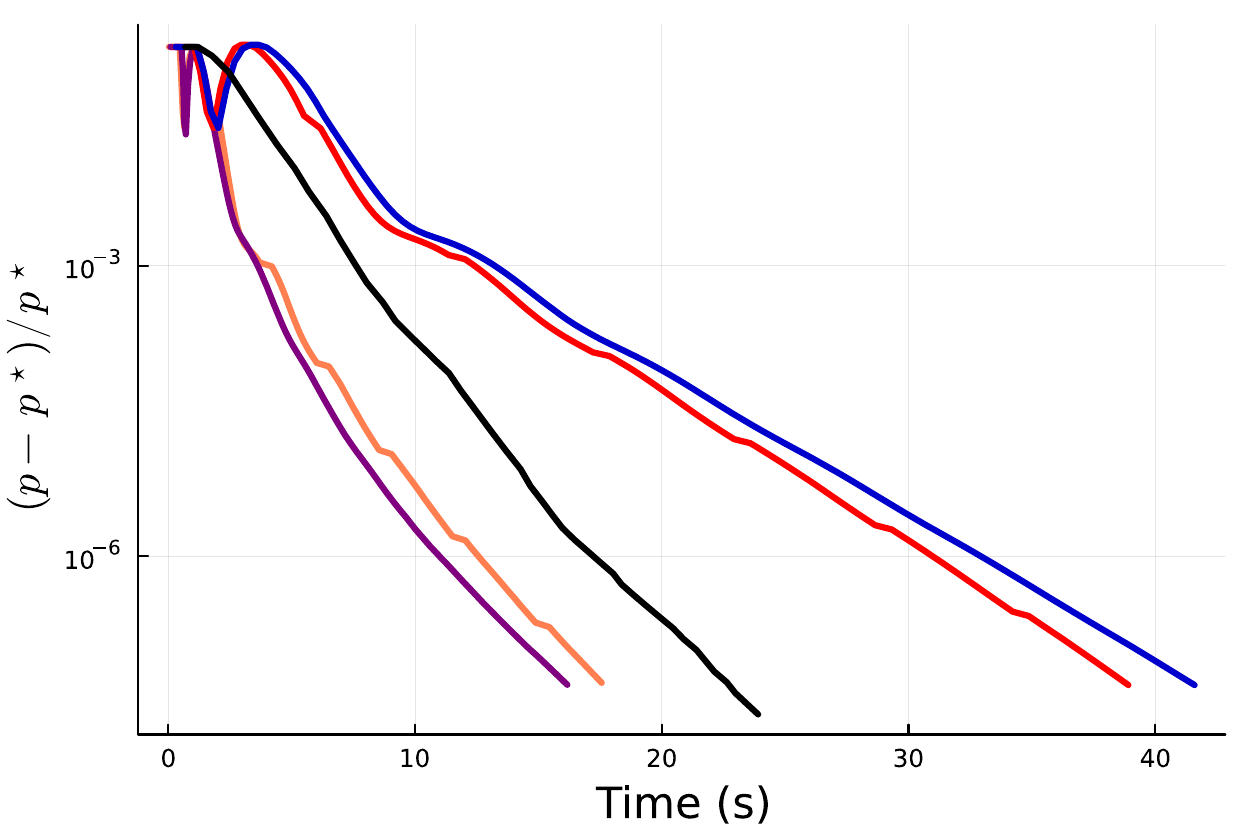}
        \caption{Objective value convergence}
    \end{subfigure}
    \begin{subfigure}[t]{0.48\textwidth}
        \centering
        \includegraphics[width=\columnwidth]{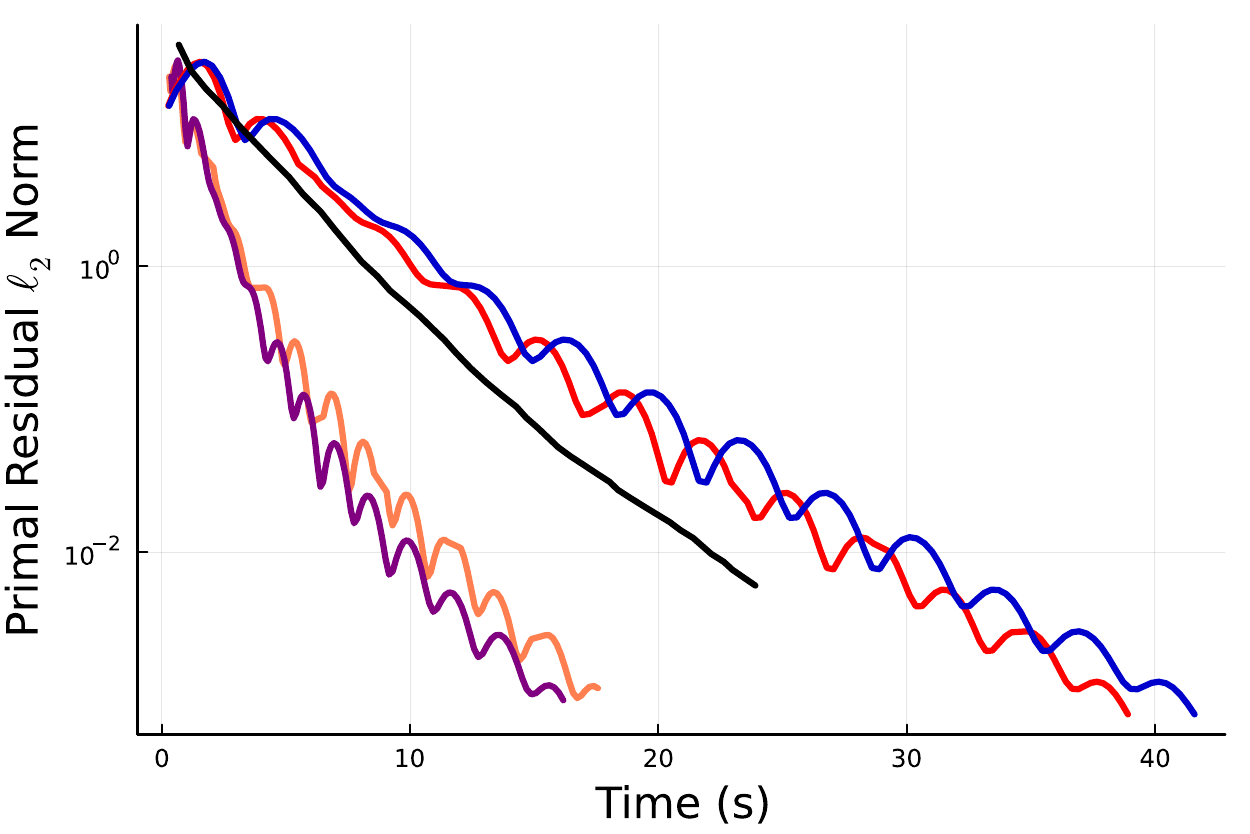}
        \caption{Primal residual convergence}
    \end{subfigure}
    \hfill
    \begin{subfigure}[t]{0.48\textwidth}
        \centering
        \includegraphics[width=\columnwidth]{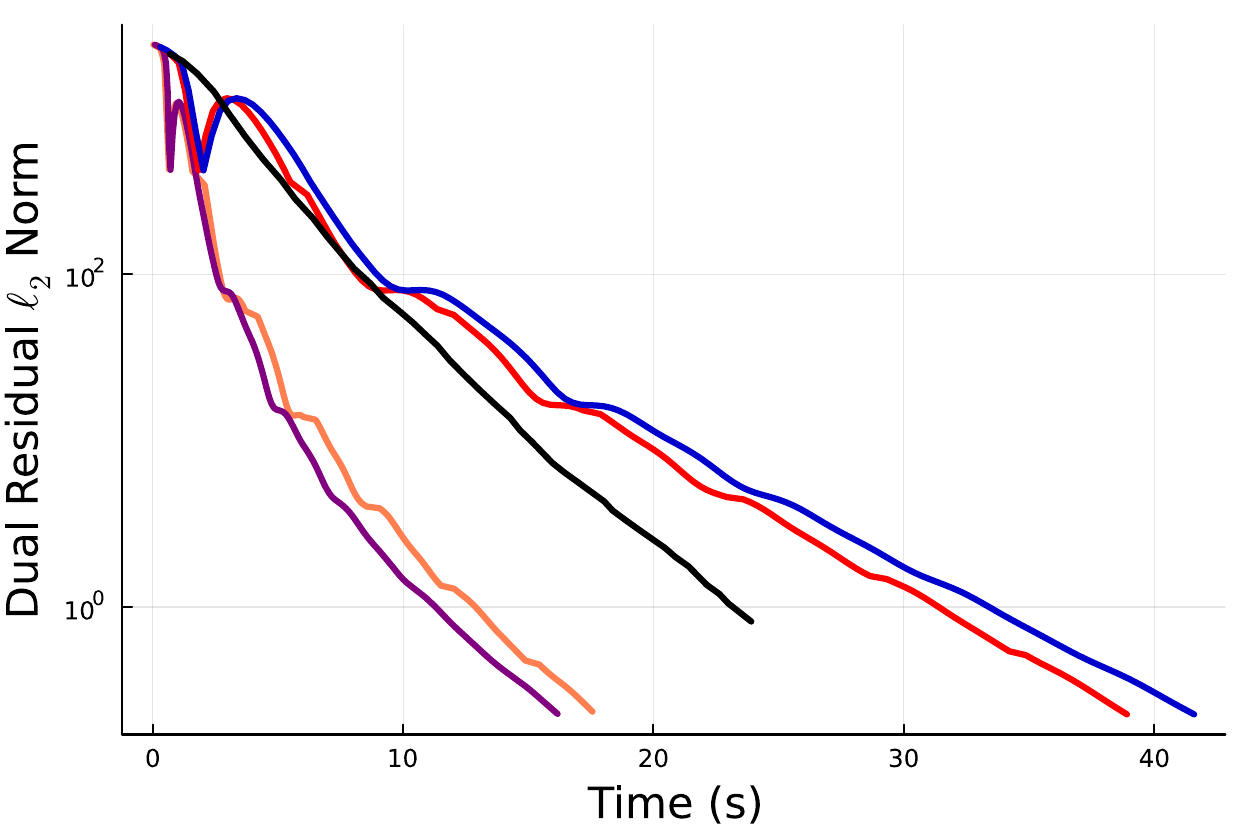}
        \caption{Dual residual convergence}
    \end{subfigure}
    \caption{
    Both subproblem approximation and inexact solves improve \method{}'s convergence time on the logistic regression problem with the real-sim dataset. Similarly to the elastic net problem, the preconditioner has little effect since this dataset is very sparse. (Note that the relatively duality gap may be ill-defined due to domain or divide-by-zero issues. See appendix~\ref{appendix:dual-gap} for details.)
    }
    \label{fig:logistic}
\end{figure}

\begin{table}[h]
    \centering
    \footnotesize
\begin{adjustbox}{max width=\textwidth}
\begin{tabular}{@{}lrrrrrr@{}}
\toprule
&GeNIOS & G. (no pc) & G. (exact) & G. (no pc, exact) & ADMM LBFGS & Conic\\
\midrule
setup time (total) & 0.416s & 0.010s & 0.404s & 0.010s & 0.010s & 0.019s\\
\quad preconditioner time & 0.394s & 0.000s & 0.393s & 0.000s & 0.000s & 0.000s\\
solve time & 17.731s & 16.355s & 39.165s & 41.886s & 24.389s & 81.700s\\
\quad number of iterations &  134 &  136 &  137 &  137 &   45 &  128\\
\quad total linear system time & 11.507s & 12.420s & 32.774s & 37.830s & 23.307s & 41.788s\\
\quad avg. linear system time & 85.877ms & 91.326ms & 239.224ms & 276.133ms & 517.943ms & 326.471ms\\
\quad total prox time & 0.002s & 0.003s & 0.003s & 0.003s & 0.001s & 35.173s\\
\quad avg. prox time & 0.019ms & 0.019ms & 0.019ms & 0.019ms & 0.019ms & 274.789ms\\
total time & 18.147s & 16.365s & 39.568s & 41.896s & 24.399s & 81.719s\\
\bottomrule
\end{tabular}
\end{adjustbox}
    \caption{Timings for \method{} with and without preconditioning, subproblem approximations, and inexact solves for the logistic regression problem with the real-sim dataset. We also compare with the conic form problem.}
    \label{tab:logistic}
\end{table}

\subsection{Huber fitting}
\label{sec:ex-huber}
In this example, we showcase \method{}'s ability to handle custom loss functions, which are often not supported by standard machine learning solvers.
Handling these custom loss functions directly, instead of solving the equivalent conic program, speeds up the solution of machine learning problems, which we demonstrate for the case of Huber fitting~\cite{huber1964robust}.
The Huber fitting problem replaces the standard squared-error loss function (\cf the elastic net problem in~\S\ref{sec:ex-elastic-net})
with a loss function that is less sensitive to outliers, defined as
\[
\ell^\mathrm{hub}(w) =
\begin{cases}
w^2 & \lvert w\rvert \leq 1 \\
2|w| - 1 & \lvert w \rvert > 1.
\end{cases}
\]
This function is easily verified to be convex and smooth.
The $\ell_1$-regularized Huber fitting problem is then
\[
\begin{aligned}
&\text{minimize}     && \sum_{i=1}^N \ell^\mathrm{hub}(a_i^T x - b_i) + \lambda_1\|x\|_1,
\end{aligned}
\]
with problem data $a_i \in \reals^n$ and $b_i \in \reals$ for $i = 1, \dots, N$, variable $x \in \reals^n$, and regularization parameter $\lambda_1 \ge 0$.
Huber fitting, a form of \emph{robust regression}, often has superior performance on real-world data and obviates the need for outlier detection in data pre-processing.
However, the Huber loss and other robust loss functions are usually not supported by standard solvers.
\method{}'s \mlsolver{}, on the other hand, provides full support for custom convex loss functions, as described in~\S\ref{sec:app-ml}.

\paragraph{Equivalent quadratic program.}
The $\ell_1$-regularized Huber fitting problem can also be written as a quadratic program (see~\cite{mangasarian2000robust} for details) by introducing new variables $q \in \reals^n$, $r \in \reals^N$, $s \in \reals^N$, and $t \in \reals^N$:
\[
\begin{aligned}
&\text{minimize}        && r^Tr + 2\ones^T(s + t) + \lambda_1 q\\
&\text{subject to}      && Ax - r - s + t = b \\
                       &&& -q \le x \le q\\
                       &&& 0 \le s, t
\end{aligned}
\]
where $A$ is a $N \times n$ matrix with rows $a_i^T$ and $b$ is a vector with elements $b_i$.
This problem now has $2n + 3N$ variables, and the constraint matrix in~\eqref{eq:problem-formulation-qp} is of size $2n + 3N \times 2n + 3N$.
Putting other robust regression problems into standard forms can also lead to substantial increases in problem size.

\paragraph{Problem data.}
For this problem, we generate random data with $n$ features and $N = n/2$ samples for varying values of $n$.
We sample each feature $A_{i,j}$ independently from the standard normal distribution $\mathcal{N}(0, 1)$ and then normalize the columns to have zero mean and unit $\ell_2$ norm.
The true value of weights $x^\star$ has 10\% nonzero entries, each sampled from $\mathcal{N}(0, 1)$.
The response vector $b$ is computed as $Ax^\star + 0.1v$, where $v \sim \mathcal{N}(0, 1)$.
Finally, we make 5\% of $b$ outliers by adding $u$ to these entries, where each $u$ is drawn uniformly at random from $\{-10, 10\}$. 
We take $\lambda_1 = 0.1\|A^Tb\|_\infty$.
We vary $n$ from 250 to 16,000.

\paragraph{Comparing solver interfaces.}
First, we examine the difference between using the \method{} \texttt{MLSolver} interface and using the \texttt{QPSolver} interface.
Since the residuals for the \texttt{MLSolver} and \texttt{QPSolver} interfaces are different, we overwrite the convergence criteria calculation to examine the subdifferential of the original loss function:
\[
    A^T {\ell^\mathrm{hub}}'(Ax - b) + \lambda \partial \|x\|_1,
\]
where the derivative of the loss is applied elementwise and
\[
\partial \|x\|_1 = \left\{v \mid v_i = \mathbf{sign}(x_i) ~ \text{if} ~ \abs{x_i} > 0 ~ \text{and otherwise} ~ v_i \in [-1,1]\right\}.
\]
We say that the solver has converged to an optimal solution when the distance from the vector $0$ to the subdifferential set is less than \texttt{1e-4}, \ie when there is some vector in this set with $\ell_2$ norm less than \texttt{1e-4}.
We do not use randomized preconditioning for this problem, as the random data matrix $A$ is not approximately low rank.
We solve the problem for $n$ varied from $250$ to $16,000$.
Figure~\ref{fig:huber-interface} shows timing for the linear system solve in the $x$-subproblem and for the overall solve for varying $n$, and it shows the residuals' convergence as a function of cumulative time spent on solving the linear system for $n = 16,000$.
We provide detailed solve time breakdowns in table~\ref{tab:huber} for $n=16,000$.

\paragraph{Discussion.}
The \mlsolver{} speeds up solve time by over an order of magnitude compared to the \qpsolver{}, even for small problem sizes.
The dominant operation per iteration---the linear system solve---is over an order of magnitude faster
(table~\ref{tab:huber}).
\mlsolver{} is also faster than \qpsolver{} even after subtracting the time to solve the linear system.
Both methods require about the same number of iterations to reach the stopping criterion.
Multiplication by $A$ and $A^T$ are two most expensive operations in the algorithm.
\mlsolver{} represents $A$ as a dense matrix and uses optimized BLAS operations to compute matrix-vector products.
In contrast, \qpsolver{} stores the data matrix $A$ as part of the large, sparse constraint matrix $M$ (see \eqref{eq:problem-formulation-qp}) and
so cannot apply $A$ and $A^T$ efficiently.

\begin{figure}[h]
    \captionsetup[sub]{font=scriptsize}
    \centering
    \begin{subfigure}[t]{0.48\textwidth}
        \centering
        \includegraphics[width=\columnwidth]{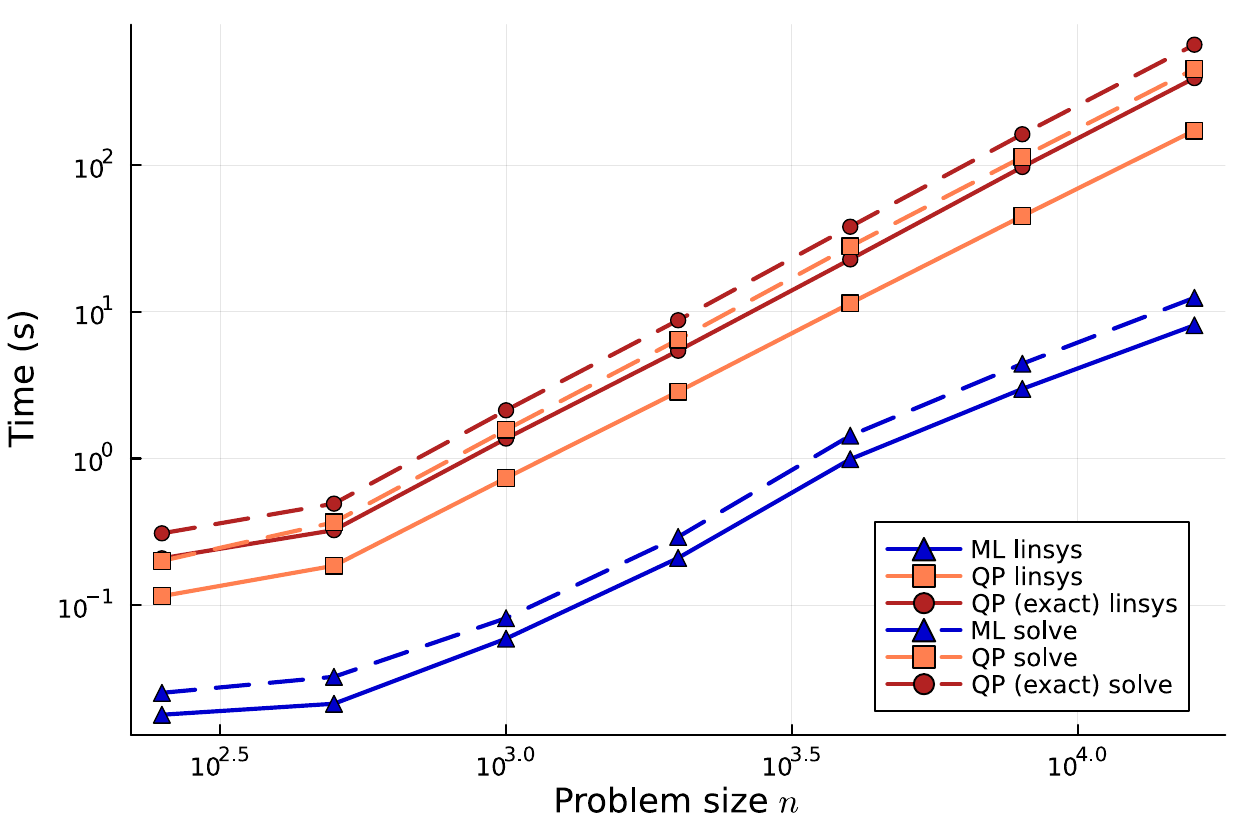}
        \caption{Solve time comparisons for $n = 250$ to $n=16,000$.}
        \label{fig:huber-solve-time}
    \end{subfigure}
    \hfill
    \begin{subfigure}[t]{0.48\textwidth}
        \centering
        \includegraphics[width=\columnwidth]{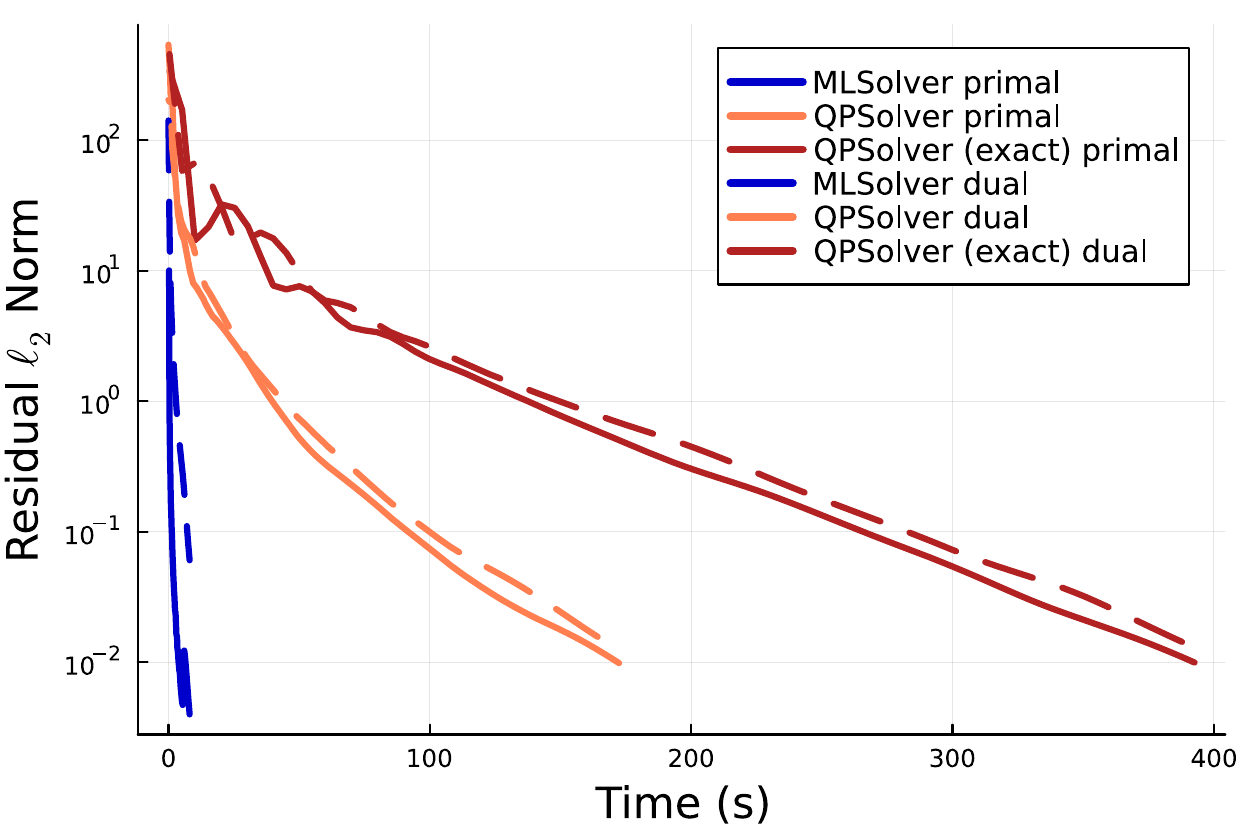}
        \caption{Primal and Dual residual convergence for $n=16,000$.}
    \end{subfigure}
    \caption{\method{}'s \mlsolver{} avoids conic reformulation and significantly outperforms the \qpsolver{} (with and without exact subproblem solves) for the Huber fitting problem.}
    \label{fig:huber-interface}
\end{figure}

\begin{table}[]
    \centering
\begin{tabular}{@{}lrrrr@{}}
\toprule
&MLSolver & QPSolver & QPSolver (exact)\\
\midrule
setup time (total) & 0.012s  & 0.000s & 0.000s\\
\quad preconditioner time & 0.000s  & 0.000s & 0.000s\\
solve time & 12.470s  & 453.366s & 663.108s\\
\quad number of iterations &   63  &   92 &   88\\
\quad total linear system time & 8.103s  & 172.414s & 392.644s\\
\quad avg. linear system time & 128.624ms & 1874.069ms & 4461.860ms\\
\quad total prox time & 0.001s  & 0.005s & 0.005s\\
\quad avg. prox time & 0.014ms  & 0.051ms & 0.052ms\\
total time & 12.482s & 453.366s & 663.109s\\
\bottomrule
\end{tabular}
    \caption{Timing comparisons of \method{}'s \mlsolver{} and \qpsolver{} for the Huber fitting problem for $n = 16,000$.}
    \label{tab:huber}
\end{table}

\subsection{Bounded Least Squares}\label{sec:ex-bounded-ls}
We compare \method{} one-to-one with other solvers on the bounded least squares problem
\begin{equation}\nonumber
    \begin{aligned}
        &\text{minimize} && (1/2)\|Ax - b\|^2_2 \\
        &\text{subject to} && 0 \leq x \leq 1.
    \end{aligned}
\end{equation}
The problem is a QP~\eqref{eq:problem-formulation-qp} with $P = A^TA$, and $q = A^Tb$.
Unlike in the unconstrained machine learning problems, the $x$-subproblem iterates are not necessarily feasible; however, the $z$-subproblem iterates are.

\paragraph{Problem data.}
We use the first $N$ samples in the YearMSD dataset and generate $n = N/2$ random features. We precompute $P$ and $q$.
The matrix in the linear system solve, $P + \rho I$, is dense.
Since the matrix $P$ comes from real world data, we expect it to also be low-rank.
The fact that $P$ is low-rank and dense suggest that using a randomized preconditioner will lead to significant speedups, similar to those for the elastic net problem in~\S\ref{sec:ex-elastic-net}.

\paragraph{Ablation study.}
First, we examine the impact of the preconditioner and inexact solves for this problem.
(Since this problem is a QP, \method{} does not approximate the $x$-subproblem.)
We solve the bounded least squares problem with $N = 20,000$ data samples and $n = N/2$ features with and without both the randomized preconditioner and inexact $x$-subproblem solves.
We set the absolute and relative termination tolerances to be \texttt{1e-5}.
Figure~\ref{fig:constrained-ls-ablation} shows the convergence of the primal residual, the dual residual, and the objective value.
The inexact solves reduce solve time by over 50\%, and
the preconditioner introduces a very modest overhead (under 2\%) for an approximately 40\% linear system solve time reduction.
Table~\ref{tab:constrained-ls-ablation} details these solve times.
Again, the addition of inexact solves and the preconditioner do not affect the number of iterations it takes for the solution to converge.
As with many QPs, the linear system solve time dominates.

\begin{figure}[h]
    \captionsetup[sub]{font=scriptsize}
    \centering
    \begin{subfigure}[t]{0.32\textwidth}
        \centering
        \includegraphics[width=\columnwidth]{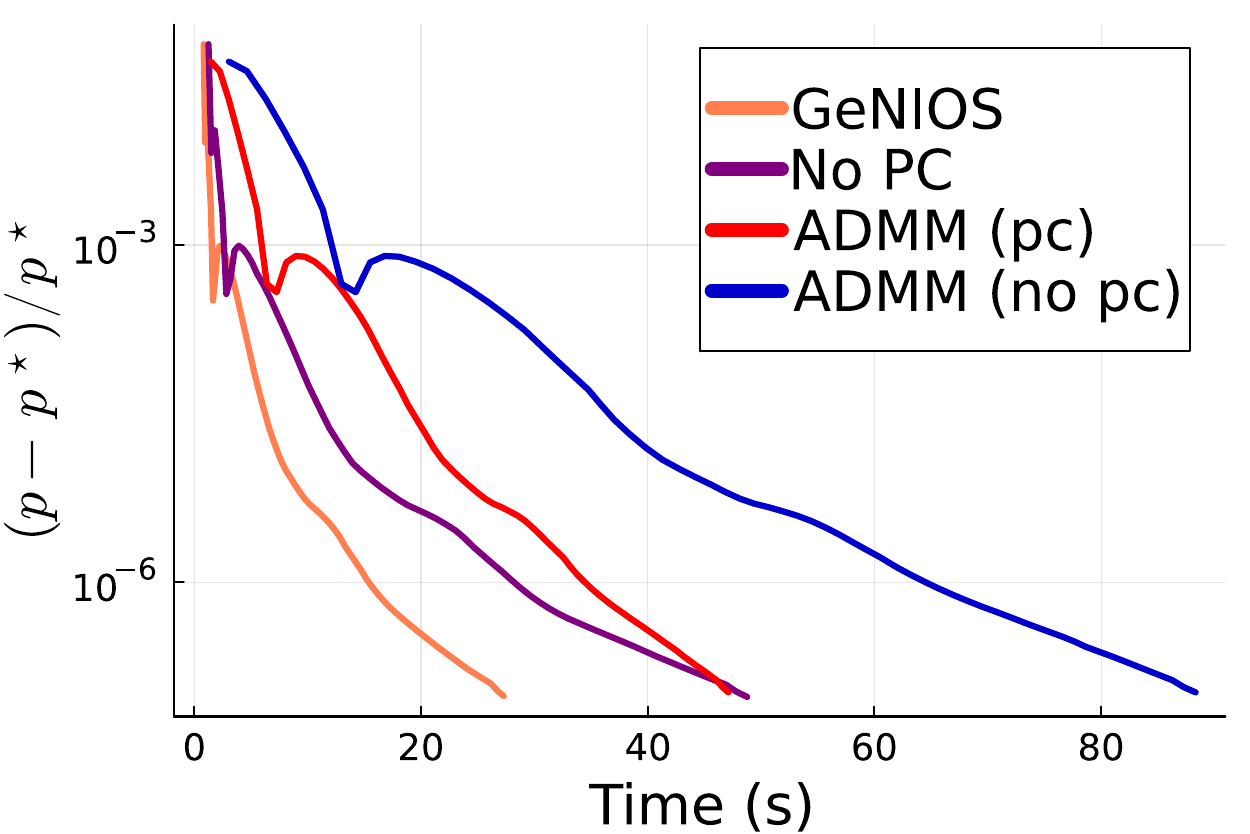}
        \caption{Objective value}
    \end{subfigure}
    \hfill
    \begin{subfigure}[t]{0.32\textwidth}
        \centering
        \includegraphics[width=\columnwidth]{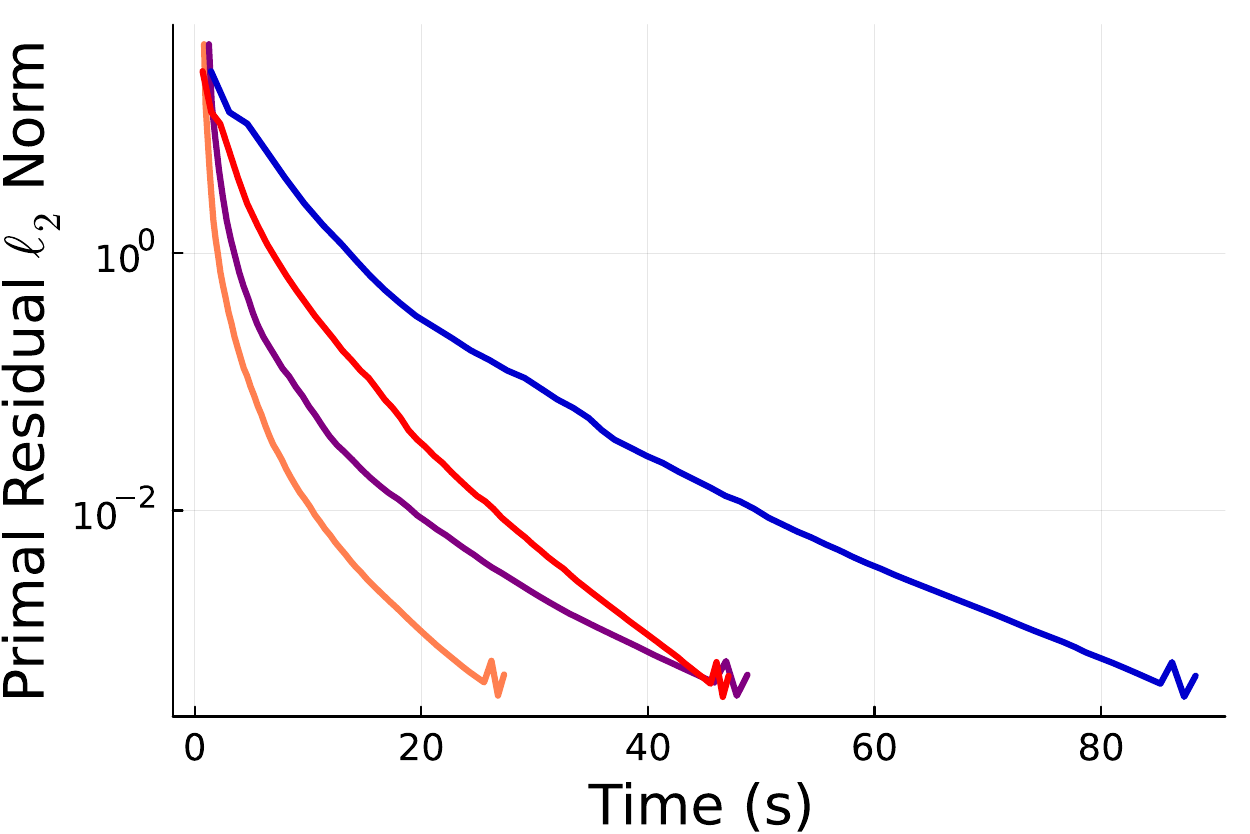}
        \caption{Primal residual}
    \end{subfigure}
    \hfill
    \begin{subfigure}[t]{0.32\textwidth}
        \centering
        \includegraphics[width=\columnwidth]{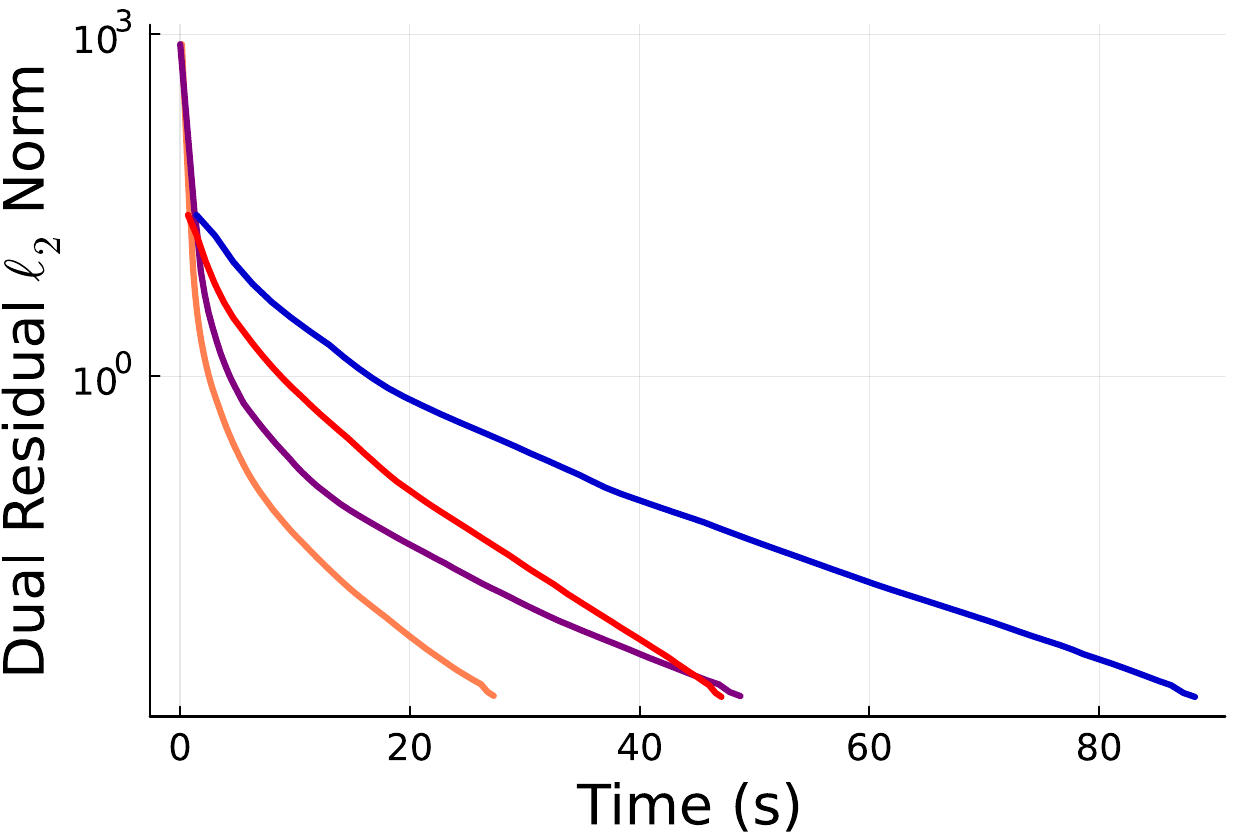}
        \caption{Dual residual}
    \end{subfigure}
    \caption{
    Both inexact subproblem solves and the preconditioner improve \method{}'s convergence time on the bounded least squares problem with the dense YearMSD dataset (\cf the elastic net problem in figure~\ref{fig:en-dense}).
    }
    \label{fig:constrained-ls-ablation}
\end{figure}

\begin{table}[]
    \centering
\ifsubmit\begin{adjustbox}{max width=\textwidth}\fi
\begin{tabular}{@{}lrrrr@{}}
\toprule
&GeNIOS & GeNIOS (no pc) & ADMM & ADMM (no pc)\\
\midrule
setup time (total) & 0.869s & 0.010s & 0.563s & 0.011s\\
\quad preconditioner time & 0.858s & 0.000s & 0.551s & 0.000s\\
solve time & 27.845s & 49.759s & 47.724s & 89.383s\\
\quad number of iterations &   67 &   67 &   67 &   67\\
\quad total linear system time & 25.747s & 47.287s & 45.553s & 86.811s\\
\quad avg. linear system time & 384.285ms & 705.773ms & 679.892ms & 1295.687ms\\
\quad total prox time & 0.001s & 0.001s & 0.001s & 0.001s\\
\quad avg. prox time & 0.013ms & 0.012ms & 0.013ms & 0.013ms\\
total time & 28.714s & 49.768s & 48.286s & 89.394s\\
\bottomrule
\end{tabular}
\ifsubmit\end{adjustbox}\fi
    \caption{Timing comparisons of \method{}'s \mlsolver{} and \qpsolver{} for the constrained least squares problem with $n = 16,000$ features (augmented) and $N = 32,000$ samples from the YearMSD dataset.}
    \label{tab:constrained-ls-ablation}
\end{table}

\paragraph{Comparison with other solvers.}
We solve the same problem, with varying value of $n$ from $250$ to $16,000$, with our solver, the popular QP solver \osqp~\cite{osqp}, the pure Julia conic solver \cosmo~\cite{cosmo}, and the commercial solver \texttt{Mosek}~\cite{mosek}. Both \osqp{} and \cosmo{} are ADMM-based solvers, while \texttt{Mosek} uses an interior point method.
For \cosmo{}, we use both the QDLDL direct solver (the default) and the CG indirect solver, which solves the same reduced system as \method{}'s \qpsolver.
We set all ADMM-based solvers to have absolute and relative termination tolerances of \texttt{1e-4} and use the infinity norm of the residuals, since this is the default in \osqp{} and \cosmo{}. We set \texttt{Mosek}'s primal and dual tolerances to \texttt{1e-4} but otherwise use default parameters.
Both \method{} and \cosmo{}'s indirect solver use inexact solves in this example.
Figure~\ref{fig:constrained-ls-compare} shows the results.
\method{} with and without preconditioning begins to perform much better than other methods as the problem size becomes large.
Preconditoning provides a nontrivial solve time reduction---about 50\%---as the problem sizes grows.
While \osqp{} and \cosmo{} with a direct linear system solve are faster than \method{} for smaller problem sizes, they have worse scaling as the problem size grows due to the matrix factorization.
\begin{figure}[h]
    \begin{subfigure}[t]{0.49\textwidth}
        \centering
        \includegraphics[width=\columnwidth]{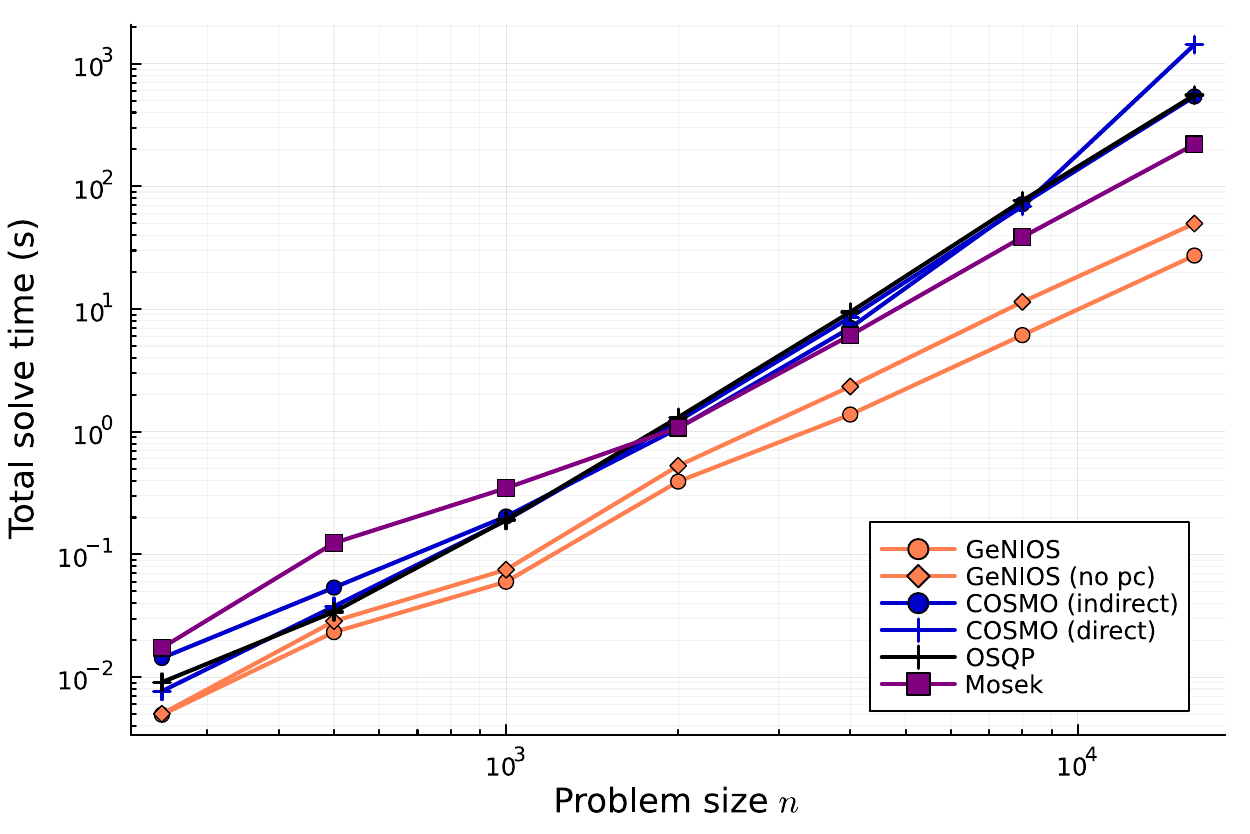}
        \caption{Bounded least squares}
        \label{fig:constrained-ls-compare}
    \end{subfigure}
    \begin{subfigure}[t]{0.48\textwidth}
        \centering
        \includegraphics[width=\columnwidth]{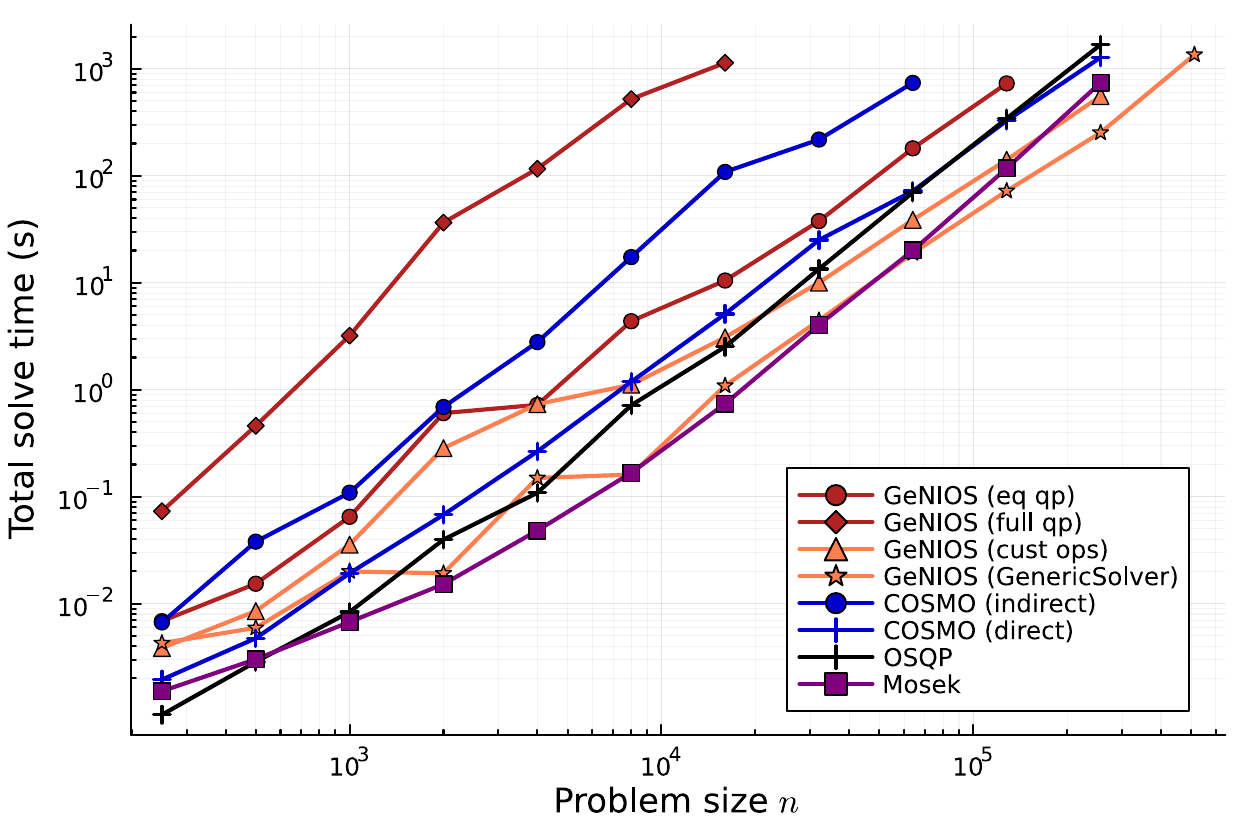}
        \caption{Portfolio optimization}
        \label{fig:portfolio}
    \end{subfigure}
    \caption{\method{} outperforms the \texttt{COSMO}, \texttt{OSQP}, and \texttt{Mosek} solvers for the dense bounded least squares problem (left). For the sparse portfolio optimization problem, \method{} outperforms direct solvers at large problem sizes and enjoys better scaling with problem data size.}
\end{figure}

\subsection{Portfolio optimization}\label{sec:ex-portfolio-opt}
The portfolio optimization problem highlights how \method{}'s \qpsolver{} allows the user to speed up standard QP's by leveraging Julia's multiple dispatch, and how \method{}'s \genericsolver{} provides powerful tools for advanced users to exploit structure.
The portfolio optimization problem finds the fraction of wealth to be invested across a universe of $n$ assets to maximize risk-adjusted return.
If the portfolio is constrained to be long-only, this problem can be written as
\begin{equation}\nonumber
\begin{aligned}
    &\text{minimize}     &&-\mu^Tx + (\gamma/2) x^T \Sigma x \\
    &\text{subject to}   && \ones^Tx = 1 \\
                        &&& x \geq 0,
\end{aligned}
\end{equation}
where the variable $x \in \reals^n$ represents the allocation,
$\mu \in \reals^n$ and $\Sigma \in \symm_+^n$ are the (estimated) return mean and covariances respectively,
and $\gamma \in \reals_{++}$ is a risk-aversion parameter.
Often, the covariance matrix $\Sigma$ has a diagonal-plus-low-rank structure, \ie
\[
\Sigma = FF^T + D,
\]
where $D \in \reals^{n \times n}$ is a diagonal matrix indicating asset-specific risk and $F \in \reals^{n \times k}$ is a factor matrix with $k \ll n$.
This problem is clearly a QP as written.
To fit into the form of \eqref{eq:problem-formulation-qp}, set
\[
    P = \gamma \Sigma, \qquad q = -\mu, \qquad  M = \bmat{\mathbf{1}^T \\ I} \
    \qquad l = \bmat{1 \\ 0}, \qquad u = \bmat{1 \\ \infty}.
\]
We solve the portfolio optimization problem using both \method's \qpsolver{} and its \genericsolver{}, which permits additional performance improvements.
We compare solve times against \osqp{} and \cosmo{} with absolute and relative tolerances set to \texttt{1e-4}, and against \texttt{Mosek} with its primal and dual tolerances set to \texttt{1e-4}.
Again, we use the infinity norm of the residuals in \method{} for the sake of comparison.

\paragraph{Equivalent QP.}
In portfolio optimization problem, the right-hand-side matrix of the $x$-subproblem linear system is a dense matrix; the diagonal-plus-low-rank structure of $\Sigma$ is lost.
To take advantage of structure, the portfolio optimization problem can be reformulated as the following equivalent quadratic program:
\[
\begin{aligned}
&\text{minimize} && (\gamma/2) x^T D x + (\gamma/2) y^Ty - \mu^T x \\
&\text{subject to} && y = F^Tx \\
&&& \ones^T x = 1 \\
&&& x \ge 0.
\end{aligned}
\]
In this equivalent formulation, the matrix $\Sigma$ is no longer formed.
We solve this QP instead of the original QP unless otherwise stated.

\paragraph{QP interface.}
While we can solve the equivalent problem introduced above, \method{} allows us to avoid this reformulation. Instead, we create types for $P$ and the constraint matrix $M$ that implement fast multiplication by taking advantage of structure.
We can multiply by $P$ in $O(nk)$ time, and we can multiply by $M$ or $M^T$ in $O(n)$ time.
Julia's multiple dispatch allows the user to easily define new objects that implement fast operations. 
We provide a tutorial in the `Markowitz Portfolio Optimization, Three Ways' example in the documentation.

\paragraph{Generic interface.}
The constraint set of the portfolio optimization problem is the $n$-dimensional simplex, for which there exists a fast projection.
We can take advantage of this fast projection using \method{}'s \genericsolver{}.
Instead of using the QP formulation~\eqref{eq:problem-formulation-qp}, we formulate the portfolio optimization problem as
\[
\begin{aligned}
    &\text{minimize}     &&-\mu^Tx + (\gamma/2) x^T \Sigma x + I_S(z)\\
    &\text{subject to}   && x - z = 0,
\end{aligned}
\]
where $I_S(z)$ is the indicator function of the simplex $S = \{z \mid \ones^Tz = 1 \text{ and } z \ge 0\}$.
Recognize this formulation as~\eqref{eq:problem-formulation} with $f(x) = (\gamma/2) x^T\Sigma x - \mu^T x$, $g(z) = I_S(z)$, $M = I$, and $c = 0$.
The proximal operator of $g$ evaluated at $v$ is a projection of $v$ onto the set $S$, defined as the solution to the optimization problem
\[
\begin{aligned}
    &\text{minimize} && (1/2)\|\tilde z - v\|_2^2 \\
    &\text{subject to} && \ones^T\tilde z = 1 \\
    &&& \tilde z \ge 0.
\end{aligned}
\]
This problem can be efficiently solved by considering the optimality condition for its dual (see appendix~\ref{app:portfolio-opt}):
\[
\ones^T(v - \nu \ones)_+ = 1.
\]
The primal solution $\tilde z$ can be reconstructed as
\[
\tilde z_i = (v_i - \nu)_+.
\]
Thus, the proximal operator of $g$ reduces to a single-variable root finding problem, which can be efficiently solved via bisection search.
In this experiment, we solve this problem to an accuracy of \texttt{1e-8}.
While $z$-subproblem inexactness is also permitted in \method{} (see~\S\ref{sec:convergence}), we do not fully explore the implications in this work.

\paragraph{Problem data.}
For this problem, we use synthetic data generated as in~\cite[App. A]{osqp}.
The asset specific risk, \ie the entries of the diagonal matrix $D$, are sampled independently and uniformly from the interval $[0, \sqrt{k}]$.
The factor loading matrix $F$ is sparse, with half of the entries randomly selected to be independently sampled standard normals and the other half to be zero.
We set $n = 100k$.
Finally, the return vector $\mu$ also has independent standard normal entries.
We set the risk parameter $\gamma$ to be $1$.
We vary the number of assets $n$ from $250$ to $512,000$.

\paragraph{Numerical results.}
For each value of $n$, we solve the portfolio optimization problem with \method{}, \cosmo{}, \osqp{}, and \texttt{Mosek}, with a time limit of 30 minutes.
We use \method{}'s \qpsolver{} to solve the original QP, the equivalent QP, and the original problem with custom $P$ and $M$ operators.
We use \method{}'s \genericsolver{} to solve the equivalent problem where we deal with the constraint set directly.
We use both \cosmo{}'s indirect and direct linear system solvers.
The results are shown in figure~\ref{fig:portfolio}.
Clearly solving the original, full QP, is a bad idea; the matrix $\Sigma$ becomes dense and much of the structure is lost, resulting in a slow solve time.
When solving the equivalent QP, \method{} outperforms \cosmo{}'s indirect solver (which also uses inexact solves), but both of these solvers are slower than \cosmo{}'s direct solver and \osqp{}.
In this problem, the constraint matrix in the equivalent formulation is sparse (approximately $0.5\%$ nonzeros as $k$ gets large) and highly structured.
As a result, sparse factorization methods perform relatively well, even for large problem sizes (over $80$M nonzeros).
However, the direct method solve times increase at a faster rate, so for a large enough problem size, indirect methods perform better than direct methods.
Solving the original QP using custom operators with \method{} becomes competitive with the direct methods much faster as the problem size gets large, and it also scales at a slower rate.
Finally, \method{}'s \genericsolver{}, which exploits additional structure, begins to outperform the direct methods for moderately-sized problems, while scaling at a rate comparable to the other indirect methods.
The direct solvers, COSMO and Mosek, both ran out of memory for the largest problem size ($n = 512,000$), and \method{}'s QP solver with custom operations takes just over the 30 minute time limit.

\paragraph{Discussion.}
The solve time speedup of \method{}'s \qpsolver{} with custom operators over the equivalent formulation is likely, at least in part, a result of how the data is stored.
Using custom operators permits \method{} to store and use $F$ as a dense matrix instead of storing and using $F$ as part of the sparse constraint matrix.
(This phenomenon is similar to what we observed in the Huber fitting example in~\S\ref{sec:ex-huber}.)
The ability to create and use these custom operators in the original QP highlights the benefits of leveraging Julia's multiple dispatch to define optimization problems.
Because \method{} is written in pure Julia, it does not need the problem data to be a matrix; it only needs linear operators.
Often, it is much easier for the user to identify a fast way of applying $P$, $M$, and $M^T$ in the original formulation~\eqref{eq:problem-formulation-qp} than to identify a good problem reformulation.
Here, for example, we easily defined a fast multiply with the diagonal-plus-low-rank matrix $\Sigma$ and the original constraint matrix $M = \bmat{ \ones & I}^T$.
Furthermore, the Julia ecosystem has many packages that assist with constructing these fast operators, including \texttt{LinearMaps.jl}~\cite{linearmaps_jl}.
The speedup from \method{}'s \genericsolver{} further emphasizes the advantage of allowing the user to naturally specify problem structure to exploit.
In this example, we recognized that projection onto the simplex is fast, and \method{} allows us to handle this part of the problem directly.

\subsection{Signal decomposition}
\label{sec:ex-signal}
The signal decomposition problem demonstrates how \method{}'s \genericsolver{} interface is flexible enough to handle nonconvex problems.
Of course, none of the convergence guarantees apply in this setting; however, ADMM applied to nonconvex problems has been shown to sometimes work well in practice (see~\cite{diamond2018-nonconvex} and references therein).
In the signal decomposition problem~\cite{meyers2023signal}, a time series $y_1, \dots, y_T \in \reals$ is modeled as a sum of components $x^1, \dots, x^K \in \reals^{T}$.
Each of these $K$ components is associated with a loss function $\phi_k :\reals^{T} \to \reals \cup \{\infty\}$, which denotes the implausibility of the current guess of $x^k$.
The function $\phi_k$ can take on the value $\infty$ to encode constraints; a signal $x$ is feasible if $\phi_k(x) < \infty$.
Given an observed signal $y$ and component classes $1, \dots, K$ associated with loss functions $\phi_1, \dots, \phi_K$, the signal decomposition problem is
\[
\begin{aligned}
    & \text{minimize} && (1/T)\|y - x^2 - \cdots - x^K\|_2^2 + \phi_2(x^2) + \cdots + \phi_K(x^K),
\end{aligned}
\]
with variables $x^2, \dots, x^K$.
The first component $x^1$ is assumed to be mean-square small.

In \method{}'s framework, this problem can be phrased as
\[
\begin{aligned}
    & \text{minimize} && (1/T)\|y - x^2 - \cdots - x^K\|_2^2 + \phi_2(z^2) + \cdots + \phi_K(z^K) \\
    & \text{subject to} && x - z = 0,
\end{aligned}
\]
where $x = (x_2, \dots, x_K)$ and $z = (z_2, \dots, z_K)$.
The function $f(x)$ is a quadratic, and the function $g(z)$ is separable across the $K$ components, so the proximal operators for $\phi_2, \dots, \phi_K$ can be evaluated in parallel.

\begin{figure}[h!]
    \captionsetup[sub]{font=scriptsize}
    \centering
    \begin{subfigure}[t]{0.48\textwidth}
        \centering
        \includegraphics[width=\columnwidth]{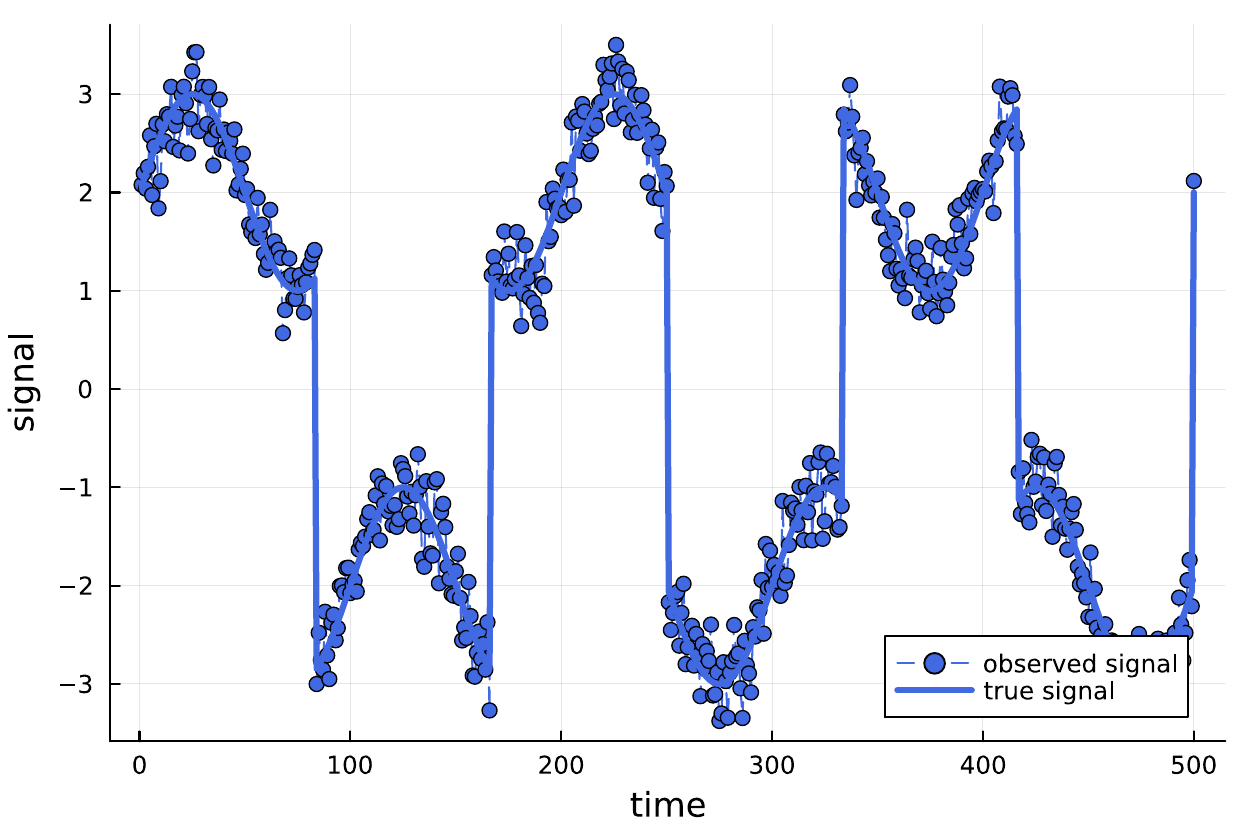}
        \caption{Observed signal}
        \label{fig:signal-decomp-observed}
    \end{subfigure}
    \hfill
    \begin{subfigure}[t]{0.48\textwidth}
        \centering
        \includegraphics[width=\columnwidth]{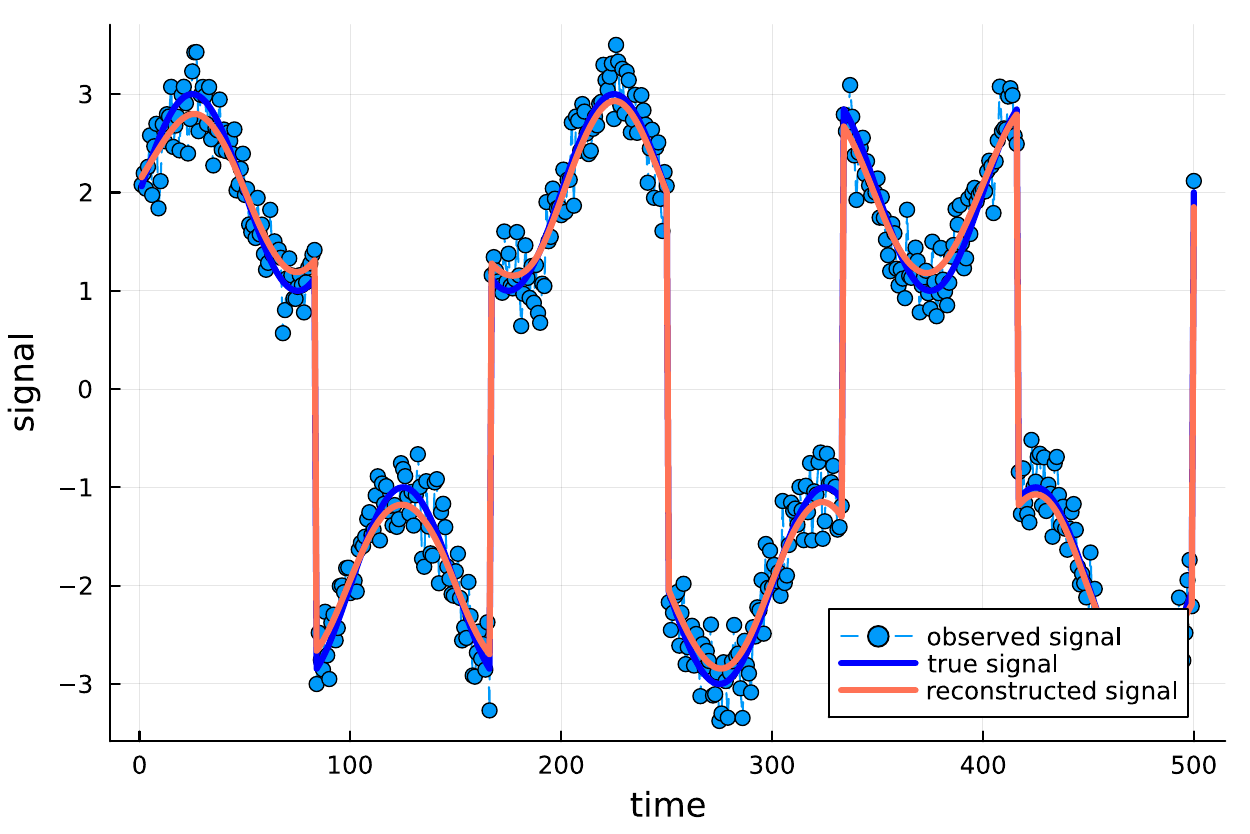}
        \caption{Reconstructed signal}
    \end{subfigure}
    \begin{subfigure}[t]{\textwidth}
        \centering
        \includegraphics[width=\columnwidth]{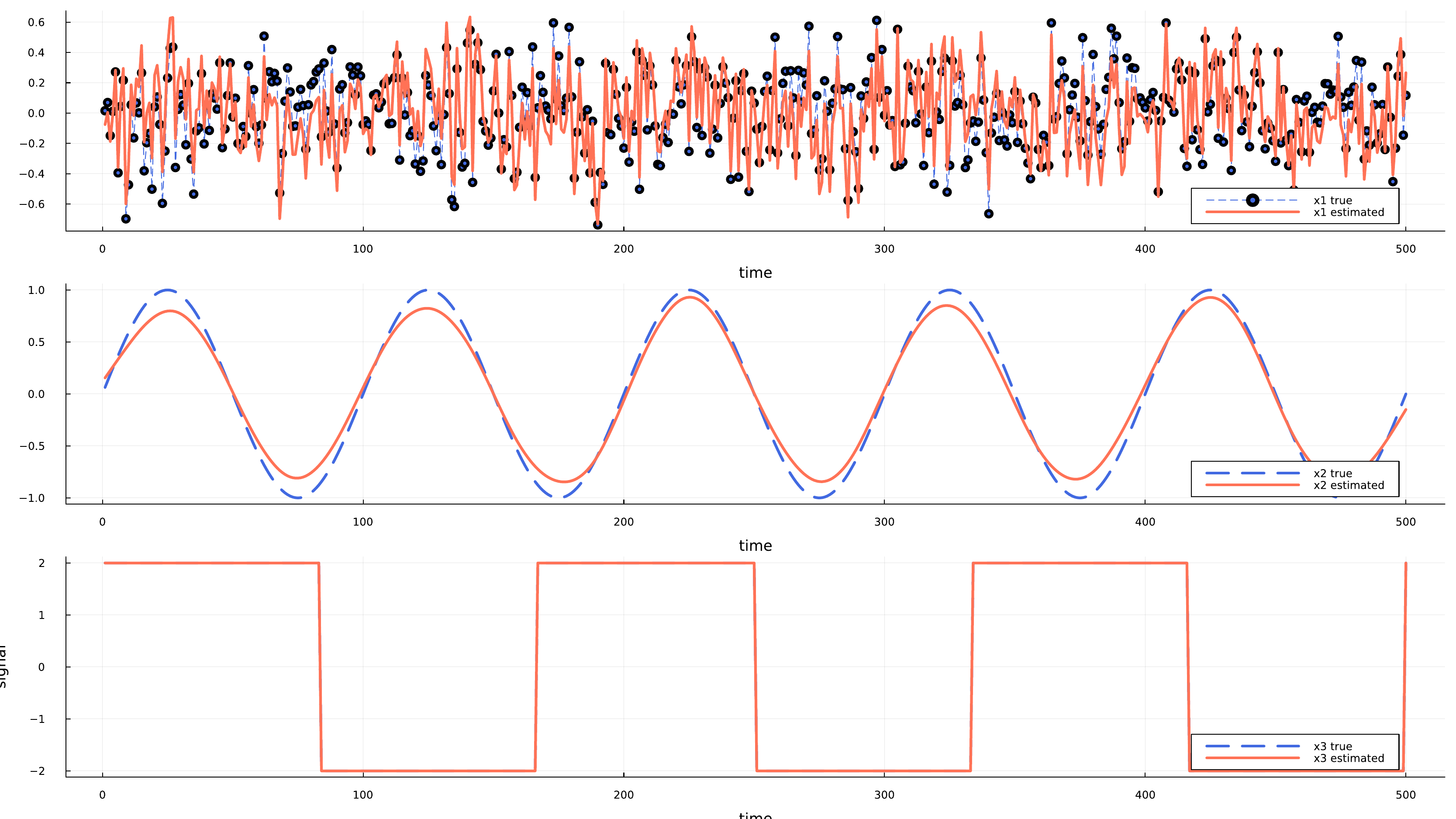}
        \caption{Individual components of the reconstructed signal.}
    \end{subfigure}
    \caption{\method{} accurately decomposes the observed signal into its components.}
    \label{fig:signal-decomp}
\end{figure}

\paragraph{Problem data.}
We generate synthetic data similar to the example in~\cite[\S 2.9]{meyers2023signal} with $T = 500$, and $K = 3$ components.
The first component is Gaussian noise.
The second component is a sine wave.
The last component is a square wave.
We observe the sum of these components (see figure~\ref{fig:signal-decomp-observed}).
We choose three component classes for the signal decomposition problem: mean-square small, mean-square second-order smooth, and a signal constrained to only take on values $\pm \theta$:
\[
\begin{aligned}
    \phi_1(x) &= \frac{1}{T}\|x\|_2^2, \\
    \phi_2(x) &= \frac{1}{T-2}\sum_{t=2}^{T-1} (x_{t+1} - 2x_t + x_{t-1})^2, \\
    \phi_3(x) &= I_{\{\pm \theta\}^T}(x).
\end{aligned}
\]
These classes select for a small signal, a smoothly changing signal, and a binary signal with known amplitude respectively.
Here, we assume the amplitude of the binary signal is known, but in a real example, we would solve the problem for varying values of the amplitude hyperparameter.

\paragraph{Numerical Results.}
The proximal operator for $\phi_2$ requires the solution of an unconstrained convex quadratic program.
Thus, this operator can be evaluated with a simple linear system solve, and a banded matrix can be used to take advantage of structure.
The proximal operator for $\phi_3$ requires the projection onto the nonconvex set $\{\pm \theta\}^T$.
This projection can be computed very quickly, but the nonconvexity of $\phi_3$ makes this optimization problem nonconvex. 
As a result, \method{} is not guaranteed to converge,
but it empirically still works well for this problem.
Figure~\ref{fig:signal-decomp} shows that \method{} recovers the true components from the observed signal.
This example highlights the flexibility afforded by \method{}'s \genericsolver{}, which can be used to solve problems with nonconvex constraint sets.



\section{Conclusion}
\label{sec:conclusion}
This paper introduced the new inexact ADMM solver \method{}, implemented in the Julia language.
This solver approximates the ADMM subproblems and solves these approximate subproblems inexactly.
These approximations yield several benefits.
First, the $x$-subproblem becomes a linear system solve, which \method{} solves efficiently even for large problem sizes with the (Nystr\"{o}m preconditioned) conjugate gradient method.
Second, \method{} avoids conic reformulations by handling the objective function directly, reducing the problem size and improving memory locality.
Despite the approximations and inexact solves, \method{} retains the convergence rate of classic ADMM and can detect infeasibility or unboundedness.
Moreover, \method{} offers a flexible interface, allowing the user to specify the objective function with zeroth, first, and second order oracles. 
It can also work from just the zeroth order oracle, computing higher order derivatives with automatic differentiation.

Through examples, we demonstrate that \method{}'s algorithmic improvements yield substantial speedups over classic ADMM and over existing solvers for large problem sizes.
These improvements allow the user to exploit problem structure using \method{}'s \genericsolver{} interface.
\method{} also includes a specialized \qpsolver{} interface for quadratic programs and a \mlsolver{} interface for machine learning problems. 
Finally, we show that \method{}'s flexible interface allows the user to specify and solve nonconvex problems as well.

There is more room to speed up \method{}.
Parallelization of the dominant operation---matrix-vector multiplies---using GPUs presents an obvious, likely significant speedup.
Future work on the algorithm itself should investigate speedups from approximations or inexact solutions to the $z$-subproblem.
GeNIOS may also benefit from acceleration, for example as in~\cite{tang2024self}, which we leave for future work.
Additionally, the Nystr\"{o}m preconditioner in~\S\ref{sec:preconditioner} performs best when the linear system matrix has the form of low-rank-plus-identity~\eqref{eq:reduc-lin-sys}. Developing optimization problem modeling tools to automatically compile problems into a form with this structure presents another interesting avenue for future work. 
Finally, we only use standard normal test matrices to construct the preconditioner used in our experiments. Other test matrices, such as subsampled trigonometric transforms, may perform better for sparse optimization problems.

\ifsubmit
    \begin{acknowledgements}
\else
    \section*{Acknowledgements}
\fi
The authors thank Chris Rackauckas, Gaurav Arya, Axel Feldmann, and Guillermo Angeris for helpful discussions.
T.\ Diamandis is supported by the Department of Defense (DoD) through the National Defense Science \& Engineering Graduate (NDSEG) Fellowship Program.
B.\ Stellato is supported by the NSF CAREER Award ECCS-2239771.
Z.\ Frangella, S.\ Zhao, and M.\ Udell gratefully acknowledge support from
the National Science Foundation (NSF) Award IIS-2233762,
the Office of Naval Research (ONR) Award N000142212825 
and N000142312203, 
and the Alfred P. Sloan Foundation.
\ifsubmit
    \end{acknowledgements}
\fi


\ifsubmit
    \bibliography{refs}
\else
    \printbibliography
\fi

\appendix
\section{Duality gap bounds for machine learning problems}
\label{appendix:dual-gap}

In this section, we derive bounds for the duality gap for the machine learning examples we consider in~\S\ref{sec:app-ml}.
The derivation is largely inspired by \cite{kim2007interior,koh2007interior}.
We consider problems of the form
\begin{equation}\label{eq:primal}
\begin{aligned}
&\text{minimize} && \ell(x) = \sum_{i=1}^m f(a_i^Tx - b_i) + \lambda_1 \|x\|_1 + (1/2)\lambda_2\|x\|_2^2,
\end{aligned}
\end{equation}
where $f(\cdot)$ is some loss function, which we assume to be convex and differentiable and $\lambda_1,\; \lambda_2 \ge 0$ are coefficients of the regularization terms which are assumed to not both be $0$. (We do not need differentiability of $f$, but is convenient for our purposes.)
Tthe optimality conditions of~\eqref{eq:primal} are that $0$ is contained in the subdifferential, {\it i.e.},
\[
    0 \in \sum_{i=1}^m f'(a_i^Tx - b_i) \cdot a_i + \lambda_1 \partial \|x\|_1 + \lambda_2 x.
\]
By rearranging, we have the condition that
\begin{equation}
    \label{eq:dual-gap-optimality-cond}
    \left(\sum_{i=1}^m f'(a_i^Tx - b_i) \cdot a_i \right)_i + \lambda_2 x_i \in \begin{cases}
    \{+\lambda_1\} & x_i < 0 \\
    \{-\lambda_1\} & x_i > 0 \\
    [-\lambda_1, \lambda_1] & x_i = 0
    \end{cases}
\end{equation}
for $i = 1, \dots, n$.
These optimality conditions indicate that $x = 0$ is a solution to~\eqref{eq:primal} if and only if $\|A^T f'(-b) + \lambda_2 x\|_\infty \le \lambda_1$, where the function $f'$ is applied elementwise.

\paragraph{Lagrangian and primal-dual optimality.}
We introduce new variable $y_i$ and reformulate the primal problem~\eqref{eq:primal} as
\begin{equation}
\label{eq:primal-reformulated}
\begin{aligned}
&\text{minimize} && l(x) = \sum_{i=1}^m f(y_i) + \lambda_1 \|x\|_1 + (1/2)\lambda_2\|x\|_2^2 \\
&\text{subject to} && y = Ax - b.
\end{aligned}
\end{equation}
The Lagrangian is then
\begin{equation}
\label{eq:dual-gap-lagrangian}
L(x, y, \nu) =
\sum_{i=1}^m f(y_i) + \lambda_1 \|x\|_1 + (1/2)\lambda_2\|x\|_2^2 + \nu^T(Ax - b - y).
\end{equation}
A primal dual point $(x, y, \nu)$ is optimal when it is feasible, \ie, $y = Ax - b$ and the gradient of the Lagrangian vanishes:
\begin{equation}
\label{eq:dual-gap-optimality-cond-lagrangian}
    \nu_i = f'(a_i^Tx - b_i)
    \qquad \text{and} \qquad
    (A^T\nu + \lambda_2 x)_i  \in
    {
        \begin{cases}
        \{+\lambda_1\} & x_i < 0 \\
        \{-\lambda_1\} & x_i > 0 \\
        [-\lambda_1, \lambda_1] & x_i = 0.
        \end{cases}
    }
\end{equation}
In deriving the dual function, we will have to consider the case when $\lambda_2 = 0$, in which case we are solving an $\ell_1$-regularized problem, separately from the case that $\lambda_2 > 0$, in which case the regularization term is smooth.

\paragraph{Optimality gap.}
First, we will consider the case when $\lambda_2 > 0$. Partially minimizing the Lagrangian over $x$ and $y$, we get the dual function
\[
g(\nu) = -\sum_{i=1}^m f^*(\nu_i) - b^T\nu - (1/2\lambda_2)\sum_{i=1}^n \left(\left(\abs{(A^T\nu)_i} - \lambda_1 \right)_+ \right)^2,
\]
where $f^*$ is the Fenchel conjugate of $f$~\cite[\S 3.3]{cvxbook}.
If, instead, we have that $\lambda_2 = 0$, then
\[
g(v) = \begin{cases}
-\sum_{i=1}^m f^*(\nu_i) - b^T\nu, & \|A^T\nu\|_\infty \le \lambda_1 \\
-\infty, & \text{otherwise}.
\end{cases}
\]
The dual problem is simply
\begin{equation} \label{eq:dual}
\begin{aligned}
&\text{maximize} && g(\nu),
\end{aligned}
\end{equation}
where the norm ball constraint is encoded in the objective.
Importantly, for any dual feasible point $\nu$, we have that
\[
\ell(x) \ge \ell(x^\star) = g(\nu^\star) \ge g(\nu),
\]
where $x^\star$ and $\nu^\star$ are optimal solutions to~\eqref{eq:primal} and~\eqref{eq:dual} respectively.
It immediately follows that
\begin{equation}\label{eq:dual-gap}
    \frac{\ell(x) - g(\nu)}{\min\left(\ell(x), \lvert{g(\nu)}\rvert\right)} \ge \frac{\ell(x) - \ell(x^\star)}{\min\left(\ell(x), \lvert{g(\nu)}\rvert\right)}.
\end{equation}
We will call the left hand side the \emph{relative duality gap}, which clearly bounds the relative suboptimality of a primal-dual feasible point $(x, \nu)$.

\paragraph{Dual feasible points.}
Now, we must devise a way to construct dual feasible points $\nu$.
Inspired by the optimality conditions~\eqref{eq:dual-gap-optimality-cond-lagrangian}, we construct a dual feasible point by taking $\nu_i = f'(a_i^Tx - b_i)$ and then, when $\lambda_2 = 0$, projecting onto the norm ball given by the first optimality condition. Specifically, in this case, we take
\[
\nu_i = \frac{\lambda_1}{\left\|\sum_{i=1}^m f'(a_i^Tx - b_i) \cdot a_i + \lambda_2 x\right\|_\infty} \cdot f'(a_i^Tx - b_i).
\]
When $x \neq x^\star$, this projection ensures that $\|A^T\nu\|_\infty \le \lambda_1$, and, therefore, the function $g(\nu)$
is finite-valued if the conjugate function is finite valued.
When $x = x^\star$, this dual variable will be optimal, as it satisfies the optimality conditions by construction; the original optimality condition~\eqref{eq:dual-gap-optimality-cond} ensures that $\nu$ will be unaffected by the projection and, therefore, satisfy both conditions in~\eqref{eq:dual-gap-optimality-cond-lagrangian}.

This construction, along with the bound~\eqref{eq:dual-gap} suggests a natural stopping criterion.
For any primal feasible point, we construct a dual feasible point $\nu$. Using the dual feasible point, we evaluate the duality gap.
We then terminate the solver if the relative gap is less than some tolerance $\epsilon$ as
\[
\frac{l(x) - g(\nu)}{\max\left(\lvert{g(\nu)\rvert}, l(x)\right)} \le \epsilon.
\]
If this condition is met, we are guaranteed that the true relative error is also less than $\epsilon$ from \eqref{eq:dual-gap}.

\section{Logistic regression conic form} \label{app:logistic-conic}
Here, we consider a modified version of~\eqref{eq:problem-formulation-qp},
\begin{equation}\label{eq:problem-formulation-logistic-conic}
    \begin{aligned}
        &\text{minimize}     && (1/2)x^TPx + q^Tx \\
        &\text{subject to}   && Mx \in C
    \end{aligned}
\end{equation}
where $C = C_1 \times \cdots \times C_d$ and each $C_i$ is either a box or an exponential cone, defined as
\[
K_\mathrm{exp} = \left\{(x,y,z) \mid y > 0 \text{ and } y\exp(x/y) \le z\right\}.
\]
(We work with this form because it requires only a slight modification of our \qpsolver{} interface~\eqref{eq:problem-formulation-qp}.)
The logistic regression problem is
\[
\begin{aligned}
& \text{minimize} && \sum_{i=1}^N \log\left(1 + \exp(a_i^Tx)\right) + \lambda_1 \|x\|_1,
\end{aligned}
\]
with variable $x \in \reals^n$.
We can reformulate this into an exponential cone program by introducing new variables 
$q \in \reals^n$, $t \in \reals^N$, $r \in \reals^N$, $u \in \reals^N$, $s \in \reals^N$, and $v \in \reals^N$ 
via transformation into epigraph form (see~\cite[\S4.2.4]{cvxbook}).
We also notice that
\[
\log\left(1 + \exp(a_i^Tx)\right) \le t_i \iff \exp(-t_i) + \exp(a_i^Tx) \le 1.
\]
We can transform convex exponentials with inequalities into exponential cone constraints by using the fact
\[
\exp(a) \le b \iff (a, 1, b) \in K_\mathrm{exp}.
\]
Putting these transformations together, we have the equivalent problem
\[
\begin{aligned}
    & \text{minimize}   && \sum_{i=1}^N t_i + \lambda_1 q_i \\
    & \text{subject to} && -q \le x \le q \\
    &&& u + v \le 1 \\
    &&& (-t_i, r_i, u_i) \in K_\mathrm{exp}, \qquad i = 1, \dots N \\
    &&& (a_i^Tx, s_i, v_i) \in K_\mathrm{exp}, \qquad i = 1, \dots N \\
    &&& r = 1 \\
    &&& s = 1.
\end{aligned}
\]
After completing the transformation and putting it into the conic form~\eqref{eq:problem-formulation-logistic-conic}, 
the matrix $P = 0$, $q \in \reals^{2n + 5N}$, and $M \in \reals^{2n + 9N \times 2n + 5N}$. 
Although the side dimensions of $M$ are large, this matrix is very sparse and highly structured. 
Note that we need the constraints $r = 1$ and $s = 1$ to write the problem in GeNIOS's conic form, 
although these constraints can be mathematically included in the exponential cone constraint.

\section{Projection onto the simplex} \label{app:portfolio-opt}
In this section, we derive an efficient method to project onto the simplex, \ie to solve the optimization problem
\[
\begin{aligned}
    &\text{minimize} && (1/2)\|\tilde z - v\|_2^2 \\
    &\text{subject to} && \ones^T\tilde z = 1 \\
    &&& \tilde z \ge 0.
\end{aligned}
\]
Our approach is similar to the approach used to solve the water-filling problem in communications.
The Lagrangian for this problem is
\[
\mathcal{L} ({\tilde{z}}, \nu, \lambda) = (1/2)\|{\tilde z} - v\|_2^2 + \nu(\ones^T{\tilde z} - 1) - \lambda^T {\tilde z}.
\]
Denote the optimal primal and dual variables by ${\tilde z}^\star$, $\nu^\star$, and $\lambda^\star$.
The optimality conditions are that ${\tilde z}^\star$ is primal feasible, dual feasibility $\lambda^\star \ge 0$, complementary slackness ${\tilde z}^\star_i \lambda^\star_i = 0$,
and that the gradient of the Lagrangian vanishes:
\[
{\tilde z}^\star - v + \nu^\star \ones - \lambda^\star = 0.
\]
From these conditions, we can conclude that
\[
{\tilde z}^\star_i > 0 \implies \tilde z_i^\star = v_i - \nu^\star.
\]
Otherwise, we have that
\[
v_i - \nu^\star + \lambda^\star_i = 0 \implies v_i - \nu^\star \le 0.
\]
Putting these together, we can write ${\tilde z}^\star$ in terms $\nu^\star$ as
\[
    \tilde z^\star = (v - \nu^\star \ones)_+.
\]
Plugging this expression for $\tilde z^\star$ into the constraint that the sum of the entries is equal to $1$, we see that $\nu^\star$ must satisfy
\[
    \sum_{i=1}^n (v_i - \nu^\star)_+ = 1.
\]
We can solve this single variable equation with bisection search for $\nu^\star$, and then solve for $\tilde z^\star$.
Easy upper and lower bounds for $\nu^\star$ are $\max_i \abs{v_i}$ and $ -(\max_i \abs{v_i} + 1)$ respectively.

\section{Timing tables for solver comparisons}\label{appendix:timing}

\begin{table}[h!]
    \centering
    \ra{1.3}
\begin{adjustbox}{max width=\textwidth}
\begin{tabular}{@{}lrrrrrr@{}}
\toprule
$n$ &GeNIOS & GeNIOS (no pc) & COSMO (indirect) & COSMO (direct) & OSQP & Mosek\\
\midrule
250 & 0.005 & 0.005 & 0.014 & 0.008 & 0.009 & 0.017 \\
500 & 0.023 & 0.029 & 0.054 & 0.038 & 0.034 & 0.124 \\
1000 & 0.060 & 0.075 & 0.204 & 0.190 & 0.191 & 0.347 \\
2000 & 0.393 & 0.528 & 1.068 & 1.220 & 1.309 & 1.079 \\
4000 & 1.381 & 2.337 & 7.033 & 8.540 & 9.448 & 6.109 \\
8000 & 6.131 & 11.428 & 71.613 & 68.764 & 76.332 & 38.550 \\
16000 & 27.301 & 49.767 & 540.350 & 1427.843 & 554.088 & 220.001 \\
\bottomrule
\end{tabular}
\end{adjustbox}
    \caption{Timings for the constrained least squares experiment in \S\ref{sec:ex-bounded-ls} and figure~\ref{fig:constrained-ls-compare}. All units are seconds.}
    \label{tab:ls-compare}
\end{table}

\begin{table}[h!]
    \centering
    \ra{1.3}
\begin{adjustbox}{max width=\textwidth}
\begin{tabular}{@{}lrrrrrrrr@{}}
\toprule
$n$ & GeNIOS (eq qp) & GeNIOS (full qp) & GeNIOS (cust ops) & GeNIOS (GenericSolver) & COSMO (indirect) & COSMO (direct) & OSQP & Mosek \\
\midrule
250 & 0.007 & 0.073 & 0.004 & 0.004 & 0.007 & 0.002 & 0.001 & 0.002 \\
500 & 0.015 & 0.460 & 0.009 & 0.006 & 0.038 & 0.005 & 0.003 & 0.003 \\
1,000 & 0.065 & 3.199 & 0.036 & 0.020 & 0.109 & 0.019 & 0.008 & 0.007 \\
2,000 & 0.604 & 36.451 & 0.284 & 0.019 & 0.689 & 0.068 & 0.040 & 0.015 \\
4,000 & 0.721 & 116.178 & 0.732 & 0.149 & 2.784 & 0.265 & 0.109 & 0.048 \\
8,000 & 4.366 & 520.425 & 1.113 & 0.161 & 17.360 & 1.185 & 0.715 & 0.167 \\
16,000 & 10.483 & 1135.043 & 3.083 & 1.089 & 108.480 & 5.110 & 2.519 & 0.738 \\
32,000 & 37.676 & - & 10.011 & 4.445 & 217.755 & 25.026 & 13.421 & 4.038 \\
64,000 & 180.258 & - & 38.647 & 18.851 & 737.456 & 71.895 & 69.671 & 20.011 \\
128,000 & 729.207 & - & 140.823 & 72.045 & - & 326.915 & 343.087 & 117.370 \\
256,000 & - & - & 554.239 & 252.642 & - & 1274.387 & 1677.730 & 740.564 \\
512,000 & - & - & - & 1353.556 & - & o.o.m. & - & o.o.m. \\
\bottomrule
\end{tabular}
\end{adjustbox}
    \caption{Timings for the portfolio optimization experiment in \S\ref{sec:ex-portfolio-opt} and figure~\ref{fig:portfolio}. We gave solvers 30 minutes to solve this problem. Mosek and COSMO's direct solver ran out of memory (o.o.m.) for $n=512,000$. All units are seconds.}
    \label{tab:portfolio-compare}
\end{table}

\subsection{Additional solver comparisons for~\S\ref{sec:ex-elastic-net} and~\S\ref{sec:ex-logistic}}
Tables~\ref{tab:ls-compare} and~\ref{tab:portfolio-compare} contain the data in figures~\ref{fig:constrained-ls-compare} and~\ref{fig:portfolio} respectively. 
Tables~\ref{tab:elastic-net-compare} and~\ref{tab:logistic-compare} compare \method{} to the convex optimization solvers OSQP, Mosek, and COSMO, and to the fast iterative shrinkage thresholding algorithm (FISTA)~\cite{beck2009fast}. FISTA is a special-purpose algorithm for $\ell_1$-regularized machine learning problems. For FISTA, we use the theoretical step size parameter. We measure convergence using the primal and dual residuals for all solvers, instead of using the dual gap as we do in~\S\ref{sec:ex-elastic-net} and~\S\ref{sec:ex-logistic}.
Since OSQP only solves quadratic programs, we cannot use to solve the logistic regression problem.
\begin{table}[h]
    \centering
    \ra{1.3}
\begin{adjustbox}{max width=\textwidth}
\begin{tabular}{@{}lrrrrrr@{}}
\toprule
Dataset & GeNIOS & OSQP & COSMO (indirect) & COSMO (direct) & Mosek & FISTA \\
\midrule
YearMSD & 7.153s & 132.716s & - & 122.610s & 35.262s & 120.528s \\
real-sim & 2.710s & 540.195s & 209.921s & 501.304s & 37.262s & 9.206s \\
\bottomrule
\end{tabular}
\end{adjustbox}
    \caption{We compare \method{} to other solvers on the elastic net experiment in \S\ref{sec:ex-elastic-net}. We used the first 21k samples from real-sim and cut YearMSD to 5k samples and 10k features in these comparisons to keep solve times reasonable. COSMO's indirect solver did not converge in under 10 minutes on the YearMSD dataset.}
    \label{tab:elastic-net-compare}
\end{table}

\begin{table}[h]
    \centering
    \ra{1.3}
\begin{adjustbox}{max width=\textwidth}
\begin{tabular}{@{}lrrrrr@{}}
\toprule
Dataset & GeNIOS & COSMO (indirect) & COSMO (direct) & Mosek & FISTA \\
\midrule
real-sim & 37.593s & 2798.561s & 5256.906s & 98.882s & 42.133s \\
news20 & 25.284s & 633.446s & 1424.136s & 98.205s & 57.623s \\
\bottomrule
\end{tabular}
\end{adjustbox}
    \caption{We compare \method{} to other solvers on the logistic regression experiment in \S\ref{sec:ex-logistic}. The news20 dataset has 20k samples and 62k features. }
    \label{tab:logistic-compare}
\end{table}
\end{document}